\documentclass[reqno,11pt]{amsart}
\usepackage{a4wide,color,eucal,enumerate,mathrsfs}
\usepackage[normalem]{ulem}
\usepackage{amsmath,amssymb,epsfig,amsthm} 
\usepackage[latin1]{inputenc}
\usepackage{psfrag}
\usepackage{hyperref,cleveref,tikz}

\def\dist{{\mathop\mathrm{\,dist\,}}}
\def\loc{{\mathop\mathrm{\,loc\,}}}

\def\dz{\delta}

\def\ez{\epsilon}

\def\kz{\kappa}
\def\bz{\beta}

\def\sz{\sigma}

\def\bint{{\ifinner\rlap{\bf\kern.35em--}
\int\else\rlap{\bf\kern.45em--}\int\fi}\ignorespaces}

\def\bbint{{\ifinner\rlap{\bf\kern.35em--}
\hspace{0.078cm}\int\else\rlap{\bf\kern.45em--}\int\fi}\ignorespaces}

\newcommand{\R}{\mathbb R}

\newcommand{\HH}{\mathscr H}

\newtheorem{thm}{Theorem}[section]
\newtheorem{lem}[thm]{Lemma}
\newtheorem{prop}[thm]{Proposition}
\newtheorem{cor}[thm]{Corollary}
\newtheorem{defn}[thm]{Definition}
\numberwithin{equation}{section}

\theoremstyle{remark}
\newtheorem{rem}[thm]{Remark}

\def\bint{{\ifinner\rlap{\bf\kern.35em--}
\int\else\rlap{\bf\kern.45em--}\int\fi}\ignorespaces}

\newcommand{\mean}[1]{\,-\hskip-1.08em\int_{#1}} 

\usepackage{graphicx}
\usepackage{import}
\usepackage{xifthen}
\usepackage{pdfpages}
\usepackage{transparent}

\title[Sharp stability of Alexandrov's theorem for $C^1$ domains]{Sharp stability of Alexandrov's theorem\\ for $C^1$ domains in the small-excess regime}
\author{Alessio Figalli, Yi Ru-Ya Zhang}
\date{\today}

\address{ETH Z\"urich, Department of Mathematics, R\"amistrasse 101, 8092, Z\"urich, Switzerland}
\email{alessio.figalli@math.ethz.ch}  

\address{Academy of Mathematics and Systems Science, the Chinese Academy of Sciences, Beijing 100190, China}
\email{yzhang@amss.ac.cn}

 \thanks{The second author is funded by the National Key R\&D Program of China (Grant No. 2025YFA1018400 \&  No. 2021YFA1003100), NSFC Grant No. 12288201 \& No. 12571128, the Chinese Academy of Sciences, and CAS Project for Young Scientists in Basic Research, Grant No. YSBR-031.}
\date{\today}

\begin{document}

\begin{abstract}
We prove a sharp quantitative stability result for Alexandrov's theorem in arbitrary dimension for bounded $C^1$ open sets in a small-excess regime. More precisely, if $E\subset \mathbb R^n$ is a bounded $C^1$ open set with the same volume as the unit ball $B$, small excess, and scalar distributional mean curvature $\mathcal H_{\partial E}\in L^2(\partial E)$, then, up to a translation,
$$
\operatorname{Exc}(E)+|E\Delta B|^2+|\mu-(n-1)|^2
\le C(n)\|\mathcal H_{\partial E}-\mu\|_{L^2(\partial E)}^2
\qquad \forall\,\mu\in \mathbb R.
$$
In other words, both the excess and the symmetric difference from the ball are controlled by the optimal $L^2$-oscillation of the mean curvature. This yields a sharp stability estimate in a genuinely non-parametric regime.

The proof combines a $BV$ version of Fuglede's spectral-gap argument, a star-shaped rearrangement for sets of finite perimeter, quantitative estimates for the part of the boundary contained in the tentacles, and a polyhedral approximation argument for the non-graphical region. 
We note that the $C^1$ regularity assumption enters only as a qualitative technical ingredient of the proof, but all constants in the final estimate depend only on the dimension.
\end{abstract}

\maketitle

\section{Introduction}

Alexandrov's Soap Bubble Theorem has been the subject of extensive investigation. In its most basic form, the theorem states that the only compact, connected, embedded hypersurface in  $\mathbb R^n$ with constant mean curvature is the  $(n-1)$-dimensional sphere $\partial B=\mathbb S^{n-1}$ up to translation and scaling.

A natural and fundamental question arises regarding the stability of this statement: if a compact, connected, embedded hypersurface in $\mathbb R^n$ has nearly constant mean curvature, is it close to a sphere? The answer is known to be negative if the control over the mean curvature is insufficiently strong, as bubbling can occur, resulting in the hypersurface resembling a union of disjoint spheres. For a detailed discussion on this phenomenon, refer to \cite{K1990, K1991, B2011, BM2012, CM2017-2}, and for applications of these types of results, see \cite{S2000, H2014}.

Following a series of foundational results that elucidated the stability of this bubbling behavior \cite{CM2017-2, KM2017, DMMN2018}, several works have focused on understanding the stability of Alexandrov's theorem in the absence of bubbling. The first significant result in this direction was achieved in \cite[Theorems 1.8 and 1.10]{KM2017}, where sets that are Lipschitz-close to a sphere are considered and precise stability estimates are established in terms of the  $L^2$-oscillations of the mean curvature around the ``natural'' value $n-1$ (corresponding to the mean curvature of the unit ball $B$). This result was improved in \cite{MPS2022} (see also \cite{FZ2023}), where the $L^2$ oscillations of the mean curvature are taken around its average.

To state these results more precisely, we need to recall the following definition of nearly spherical sets. In the whole paper, $B$ shall denote the unit ball in $\mathbb R^{n}$ (so, $\partial B=\mathbb S^{n-1}$).

\begin{defn}
\label{def:nearly spherical}
For every function $v\colon \mathbb S^{n-1}\to \mathbb R$ with $\|v\|_{L^\infty(\mathbb S^{n-1})}\le \frac 1 2 $, we write
$$
B+v:=\big\{y\in \mathbb R^n\colon y=t\omega(1+v(\omega))\quad \text{for some }\omega\in \mathbb S^{n-1}, \, 0\le t<1\big\}.
$$
Then we say that $B+v$ is a $\sigma$-nearly spherical set if $\|v\|_{L^{\infty}(\mathbb S^{n-1})}+\|\nabla_\tau v\|_{L^{\infty}(\mathbb S^{n-1})}\leq \sigma,$
where $\nabla_\tau$ denotes the tangential gradient.
\end{defn}

The following result was proved in \cite[Theorem 1.10]{KM2017}. 
\begin{thm}\label{thm:Fuglede}
There exists $\sz=\sz(n)>0$ such that the following holds:
let $H=B+w$ be a $\sigma$-nearly spherical set such that
$$|H|=|B|\qquad \text{ and }\qquad 0=\int_{H} x\, dx,$$
and let $\mathcal H_{\partial H}$ denote  the (distributional) mean curvature of $\partial H$.
Assume that $\mathcal H_{\partial H} \in L^2(\partial H)$. Then (via the radial parametrization)\footnote{Here we wrote the statement from \cite[Theorem 1.10]{KM2017}.  In \cite[Theorem 1.3]{MPS2022} (see also \cite[Theorem 1.6]{FZ2023}), the slightly stronger bound
$$
\|w\|_{W^{1,2}(\mathbb S^{n-1})}\le C(n) \|\mathcal H_{\partial H}-\overline{\mathcal H}_{\partial H}\|_{L^{2}(\mathbb S^{n-1})},\qquad \overline{\mathcal H}_{\partial H}:=\mean{\partial H}\mathcal H_{\partial H}\, d\mathscr H^{n-1}.
$$
is proved.}
$$
\|w\|_{W^{1,2}(\mathbb S^{n-1})}\le C(n) \|\mathcal H_{\partial H}-(n-1)\|_{L^{2}(\mathbb S^{n-1})},
$$
for some dimensional constant $C(n)>0.$
\end{thm}

In a later work \cite{CV2018}, it was demonstrated that for a connected closed hypersurface $\partial E$ of class $C^2$ with small oscillations in the mean curvature $\mathcal H_{\partial^* E}$, there exist two concentric balls $B_{r_i}(x_0)$ and $B_{r_e}(x_0)$ such that
\begin{equation}\label{CV}
B_{r_i}(x_0)\subset E\subset B_{r_e}(x_0),\qquad r_e -r_i \le C \left(\max_{x\in \partial E} \mathcal H_{\partial E}(x) -\min_{x\in \partial E} \mathcal H_{\partial E}(x)\right).
\end{equation}
This result was obtained using a refined moving plane method, and the constant $C$ depends on the $C^2$ norm of $\partial E$, subject to a uniform interior-ball condition.

This was later improved in \cite{MP2020-2} by replacing the $L^\infty$-oscillation of $\mathcal H_{\partial^* E}$ in \eqref{CV} with an $L^2$-oscillation. The resulting estimate involves a non-sharp exponent when $n>3$, as discussed in \cite{MP2019,MP2020}, where stability estimates with an $L^1$-oscillation were also studied. 

A recent result \cite{BZ2022} states that, for a closed smooth surface $\Sigma^2\subset \mathbb R^3$ with a small $L^2$-oscillation of its mean curvature and a no-bubbling assumption in terms of the total Willmore energy, there exists an immersion of $\Sigma^2$ that is close to the identity in the $W^{2,2}$-norm.

More recently, a sharp quantitative version of the Alexandrov inequality in three dimensions for $C^2$-regular sets with a perimeter bound has been obtained in \cite{JMOS2025}, and subsequently applied to volume-preserving geometric flows. More precisely, using a parametric approach based on conformal immersions, it was shown that, for every $\delta_0>0$, there exists $C=C(\dz_0)>0$ such that, if $E\subset \mathbb R^3$ is a $C^2$-regular set with
$$|E|=|B| \qquad \text{ and } \qquad P(E)\le 4\pi 2^{\frac 1 3}-\delta_0,$$
then
$$P(E)-P(B)\le C\|\mathcal H_{\partial E}-\bar {\mathcal H}_{\partial E}\|^2_{L^2(\partial E)}.$$
Here
$$\bar {\mathcal H}_{\partial E}=\frac{1}{P(E)}\int_{\partial^* E}{\mathcal H}_{\partial E} \,d\mathcal H^2.$$

\smallskip

In this paper, we prove a sharp quantitative stability result for Alexandrov's theorem in arbitrary dimension in the class of bounded $C^1$ open sets with small excess. In particular, our result extends the sharp nearly spherical theory of \cite{KM2017, MPS2022, FZ2023} to a non-parametric setting, where one does not assume a priori that the boundary is given as a graph over the sphere.

Recall that a Borel set $E\subset \mathbb R^n$ is a \emph{set of finite perimeter} if there exists a vector-valued Radon measure $D\chi_E$, with finite total variation,  such that 
$$\int_{E}{\rm div} \phi\, dx=- \int_{\mathbb R^n} \phi \cdot D\chi_{E}\qquad \forall\, \phi\in C_c^1(\mathbb R^n;\,\mathbb R^n). $$
Then, the perimeter of $E$ is defined as $P(E):=|D\chi_E|(\mathbb R^n)$. 
The \emph{reduced boundary} $\partial^*E\subset\partial E$ of $E$ is defined as the collection of points 
$x\in \partial E$ at which
$$\nu_E(x):=-\lim_{r\to 0} \frac{D\chi_E(B_r(x))}{|D\chi_E|(B_r(x))}\in \mathbb R^n, \quad \text{ and } \quad |\nu_E|=1,$$
where we call $\nu_E$ the (measure-theoretic) unit outer normal of $E$ at $x$. 
It holds $P(E)=\mathscr H^{n-1}(\partial^*E)$.
Finally, when the first variation of $\partial^*E$ is absolutely continuous with respect to $\mathscr H^{n-1}\llcorner\partial^*E$, we denote by
$\mathcal H_{\partial^* E}\in L^1_\loc(\partial^* E)$ the scalar distributional mean curvature of $E$, defined via the identity
$$\int_{\partial^* E} {\rm div}_\tau \Phi \, d\mathscr H^{n-1} = \int_{\partial^* E} \mathcal H_{\partial^* E} \, \Phi\cdot \nu_E\, d\mathscr H^{n-1} \qquad \forall\, \Phi\in C_c^1(\mathbb R^n;\,\mathbb R^n), $$
where 
${\rm div}_\tau \Phi = {\rm div} \Phi - (D\Phi\, \nu_E) \cdot \nu_E$
denotes the tangential divergence of $\Phi$ along $\partial^*E$.
For an introduction to the theory of sets of finite perimeter and the properties stated above, we refer to  \cite[Sections 12-15]{M2012}.

\smallskip

We now define the excess of a set of finite perimeter $E$ as
$${\rm Exc}(E)=\frac{1}{2}\int_{\partial^*E} \left|\nu_E - \frac{x}{|x|}\right|^2 \, d\mathscr H^{n-1}=\int_{\partial^*E} \left(1- \frac{x}{|x|}\cdot \nu_E \right) \, d\mathscr H^{n-1}.$$
As shown in \cite[Proposition 1.2]{FJ2014}, see also Lemma~\ref{inverse inequ} below,\footnote{Since the definitions in \cite{FJ2014} are slightly different, we give both a short proof of this inequality, and of a reverse one, in Lemma~\ref{inverse inequ}.}
 if $|E|=|B|$ then
 \begin{align}
|E\Delta B|^2 + P(E)-P(B) \le C(n) {\rm Exc}(E), \label{difference perimeter}
\end{align}
where 
$$E\Delta B=(E\setminus B)\cup (B\setminus E). $$
 In particular, since $P(E)\geq P(B)$ (by the isoperimetric inequality), 
  \begin{align}
|E\Delta B|^2 \le C(n) {\rm Exc}(E),  \label{exc symm diff}
\end{align}

\begin{thm}\label{main thm}
Let $n \geq 2$, and let $E\subset \mathbb R^n$ be a bounded $C^1$ open set with $|E|=|B|$ and whose scalar distributional
mean curvature satisfies $\mathcal H_{\partial E}\in L^2(\partial E).$ Then there exist $\delta=\delta(n)>0$ and $C=C(n)>0$ such that, whenever
${\rm Exc}(E) \leq \delta,$
there exists a translation vector $x_0\in\R^n$, with
$$
|x_0|\le C(n)\operatorname{Exc}(E)^{1/2},
$$
such that
\begin{equation}
\label{eq:main thm}
\operatorname{Exc}(E-x_0)+|(E-x_0)\Delta B|^2+|\mu-(n-1)|^2
\le C(n) \|\mathcal H_{\partial E}-\mu\|^2_{L^2(\partial E)} \qquad \forall\,\mu \in \R.
\end{equation}
In particular  one can choose $\mu$ as $\bar {\mathcal H}_{\partial E}$ (the integral average of $\mathcal H_{\partial E}$ over ${\partial^* E}$), and recalling \eqref{difference perimeter} we have
$$|(E-x_0)\Delta B |+{\rm Exc}(E-x_0)^{1/2} \le C(n) \inf_{\mu\in \mathbb R}\|\mathcal H_{\partial E}-\mu\|_{L^2(\partial E)}=C(n) \|\mathcal H_{\partial E}-\bar {\mathcal H}_{\partial E}\|_{L^2(\partial E)}.$$
\end{thm}

Compared with Theorem~\ref{thm:Fuglede}, the smallness assumption is expressed in terms of the excess, rather than through a prescribed radial graph, while the estimate retains the sharp quadratic scaling and purely dimensional constants. 

It is also worth emphasizing the relation between the present result and the recent work \cite{JMOS2025}. There, a sharp quantitative Alexandrov inequality in $\mathbb R^3$ is proved for $C^2$-regular sets under a perimeter-gap assumption, via a parametric approach based on conformal immersions.  By contrast, our argument works in every dimension $n\ge 2$ for bounded $C^1$ open sets and is genuinely non-parametric. 

\begin{rem}\label{rem:perimeter-gap}
Theorem~\ref{main thm} can be combined with the compactness/no-bubbling theory of \cite{DMMN2018} under the additional hypotheses required there. More precisely, if one assumes
$$
|E|=|B|\qquad \text{and}\qquad P(E)\le 2^{1/n}P(B)-\eta_0,
$$
together with a uniform diameter bound and a pointwise positive lower bound on the mean curvature, then the compactness result of \cite{DMMN2018} implies that any sequence with vanishing $L^2$-oscillation of the mean curvature converges, up to translations, to a single unit ball. Theorem~\ref{main thm} then yields the corresponding quantitative estimate for all sufficiently small oscillations. 

If instead one has control of the critical norm of the oscillation, namely
$L^{n-1}(\partial E)$ in our notation (the $L^n$-norm in the notation of
\cite{JN2023,Santilli2026}, where the ambient space is $\R^{n+1}$), then one can
use the bubbling results of \cite{JN2023} and 
\cite{Santilli2026}. These results give convergence, under the natural perimeter
and diameter bounds, to a finite union of equal balls (more generally, to Wulff
shapes in the anisotropic setting of \cite{Santilli2026}) without imposing a
pointwise positive lower bound on the mean curvature. The same perimeter-gap
argument then excludes more than one ball.
\end{rem}

\begin{rem}\label{rem:beyond smooth}
We note that the constant $C(n)$ in \eqref{eq:main thm} depends only on the dimension. Indeed, the $C^1$ regularity assumption on $\partial E$ is only qualitative and is used in an essential way only in two points of the proof of Proposition~\ref{leveling oscillation}. First, together with the assumption $\mathcal H_{\partial E}\in L^2(\partial E)$, it yields the strong polyhedral approximation of Corollaries~\ref{cor:smooth polyhedral approximation} and~\ref{cor:piecewise polyhedral approximation}, which is needed in Lemma~\ref{lem:polyhedral-edge}. Second, after a generic translation, it ensures that the radial tangency set
$$
\{x\in \partial E:\ x\cdot\nu_E(x)=0\}
$$
projects onto a closed $\mathscr H^{n-1}$-negligible subset of $\mathbb S^{n-1}$, in the precise sense of Lemma~\ref{lem:generic-sigma}. All the other key results of the paper can be stated and proved for sets of finite perimeter.
\end{rem}

%

%
%

The proof of Theorem~\ref{main thm} proceeds in four main steps. First, in Section~\ref{sec:preliminaries} we collect the basic consequences of small excess: We relate $\mathrm{Exc}(E)$ to both the perimeter deficit and the symmetric difference from the ball, show that only a small amount of perimeter can lie in the tentacular region away from $\partial B$, normalize the weighted barycenter by a small translation, and record a curvature estimate for truncations of the set by spheres. 

Second, in Section~\ref{sec:spectral sets FP} we establish the main analytic input of the paper, namely a $BV$-version of Fuglede's spectral-gap argument for sets of finite perimeter contained in a thin annulus around the sphere. This is combined with the star-shaped rearrangement to obtain a stability estimate for annular competitors, given in Proposition~\ref{stability sets}. In particular, the weighted  barycenter normalization $\int_E  \frac{x}{|x|}\,dx=0$ is more natural here, since it is preserved by the star-shaped rearrangement. 

Third, in Section~\ref{main proof} we return to the original set and use the estimates from Section~\ref{sec:preliminaries} to cut away the exterior and interior tentacles, thereby reducing the problem to an annular configuration to which the result of Section~\ref{sec:spectral sets FP} applies. The remaining difficulty is the non-graphical region, which is handled by the leveling argument of Proposition~\ref{leveling oscillation}; this is the only point where the $C^1$ regularity is used. 

Finally, still in Section~\ref{main proof}, we combine the annular estimate with the control of the non-graphical part to conclude the proof of Theorem~\ref{main thm}.

\medskip
{\noindent \bf Notation.} Throughout the paper, we use the notation $C(\cdot)$ and $c(\cdot)$ to denote positive constants that depend on certain parameters, with the parentheses indicating the dependencies. If a constant is absolute, we simply write $C$ or $c$. Typically, $C(\cdot)$ denotes a constant greater than 1, and $c(\cdot)$ denotes a constant less than 1.
Note that the value of $C(\cdot)$ may differ between different appearances, and even within a sequence of inequalities.
To simplify the notation, sometimes we write $a\sim b$ when $a$ and $b$ are positive numbers that are comparable up to dimensional constants, that is, $c(n)a\leq b \leq C(n)a$.

We use $B_r(x)$ to denote the Euclidean ball centered at $x$ with radius $r$, and $B$ to denote the unit ball centered at the origin. $\mathbb S^{n-1}$ is the unit sphere in $\R^n$, $\mathscr H^{n-1}$ the $(n-1)$-dimensional Hausdorff measure, and $\nabla_\tau$ the tangential gradient. We write $A_{r,R}=B_R\setminus \overline{B}_r$ for $0<r<R<\infty. $

\section{Preliminaries}\label{sec:preliminaries}

In this section, we establish several preliminary results that will be crucial for the proof of our main theorem. 

\begin{lem}
\label{lem:1 x}
Let $E\subset \mathbb R^n, n\ge 2$ be a set of finite volume.
Then
$$
\int_E\frac{1}{|x|}\,dx \leq C(n)|E|^{\frac{n-1}{n}}.
$$
Assume in addition that $|E|=|B|$. Then 
\begin{equation}\label{GB}
\int_{E\setminus B} \left||x|-1\right|\, dx\ge  c(n) |E\Delta B|^2.
\end{equation}
\end{lem}
\begin{proof}
Let $r>0$ be such that $|B_r|=|E|$. Since $\frac 1 {|x|}$ is radially decreasing, 
it follows that 
$$
\int_{E\setminus B_r}\frac 1 {|x|}\,dx \leq \frac1r |E\setminus B_r|=\frac{1}{r}|B_r\setminus E|\leq \int_{B_r\setminus E}\frac 1 {|x|}\,dx,
$$
therefore
\begin{multline*}
\int_E\frac{1}{|x|}\,dx=\int_{E\cap B_r}\frac{1}{|x|}\,dx+\int_{E\setminus B_r}\frac 1 {|x|}\,dx \\
\leq \int_{E\cap B_r}\frac{1}{|x|}\,dx+\int_{B_r\setminus E}\frac 1 {|x|}\,dx
= \int_{B_r}\frac{1}{|x|}\,dx=C(n)r^{n-1}=C(n)|E|^{\frac{n-1}{n}}.    
\end{multline*}

Assume  now that $|E|=|B|$. Since $\left||x|-1\right|$ is radial and it is increasing away from $\partial B$, arguing similarly as above one can prove that the minimum  of $\int_{E\setminus B} \left||x|-1\right|\,dx$ among sets with prescribed exterior volume is attained when $E\setminus B$ coincides with the annulus $A_{1, 1+\ez}$, with $\ez \in (0,1)$ so that
$$|A_{1,\,1+\ez}|=|E\setminus B|.$$
Therefore
\begin{equation}
    \int_{ E\setminus B} \left||x|-1\right|\, dx \ge   c(n)\int_{1}^{1+\ez} |r-1|r^{n-1}\, dr  
=   \frac{(1+\ez)^{n}(n\ez-1)+1}{n(n+1)} \sim \ez^2.\label{epsilons} 
\end{equation}
Noticing that $\ez\sim |E\setminus B|=\frac12|E\Delta B|$ (recall that $|E|=|B|$),
this concludes the proof of the lemma.
\end{proof}

Our second lemma proves \eqref{difference perimeter}, as well as its converse.
 
\begin{lem}\label{inverse inequ}
Let $E\subset \mathbb R^n$ be a set of finite perimeter such that $|E|=|B|$. Then
$$\frac{1}{C(n)}|E\Delta B|^2 + P(E)-P(B) \leq {\rm Exc}(E)\le P(E)-P(B)+ C(n)|E\Delta B|^{\frac {n-1} {n}}.  $$ 
\end{lem}
\begin{proof}
As shown in \cite[Equation (3.19)]{FJ2014}, the following identity holds:
$${\rm Exc}(E)=P(E)-P(B) + \int_{B\setminus E}\frac{n-1}{|x|}\, dx-\int_{E\setminus B} \frac{n-1}{|x|}\, dx .$$
Also, it follows from the proof of \cite[Proposition 1.2]{FJ2014} that
$$
\int_{B\setminus E}\frac{n-1}{|x|}\, dx-\int_{E\setminus B} \frac{n-1}{|x|}\, dx \geq c(n)|E\Delta B|^2.
$$
This proves the first inequality.
To obtain the second inequality, it suffices to note that applying Lemma~\ref{lem:1 x} to $B\setminus E$,
$$
\int_{B\setminus E}\frac{n-1}{|x|}\, dx-\int_{E\setminus B} \frac{n-1}{|x|}\, dx \leq \int_{B\setminus E}\frac{n-1}{|x|}\, dx\leq C(n)|B\setminus E|^{\frac{n-1}{n}}\leq C(n)|E\Delta B|^{\frac{n-1}{n}}.
$$
\end{proof}

The next lemma shows that, if the excess is small, then a set $E$ cannot have too much perimeter outside a small annulus around $\partial B$. The estimates can be done first for almost every $r$, and then for all radii by approximating an arbitrary radius by regular radii and using the lower semicontinuity of perimeter under $L^1_{\rm loc}$convergence of the truncations. We suppress this technical point.

\begin{lem}\label{small tentacle excess}
Let $E\subset \mathbb R^n$ be a set of finite perimeter such that $|E|=|B|$ and ${\rm Exc}(E)\leq 1$.
Then, for every $\gamma \in (0,1/2)$ there exists a constant $C(n,\gamma)$ such that 
\begin{align*}
&P(E\setminus B_{r})\le C(n,\gamma) {\rm Exc}(E)^{\frac{1}2} \qquad \text{ for any } \ r \geq 1+\gamma \\
&P( B_{r} \setminus E)\le C(n,\gamma) {\rm Exc}(E)^{\frac{1}2} \qquad  \text{ for any } \ r \in (1/2,1-\gamma),
\end{align*}
\end{lem}
\begin{proof}
We first consider the part of $\partial^*E$ outside the unit ball.   
By the coarea formula \cite[Theorem 1, Section 3.4.2]{EG1992},  
$$|E\Delta B|\ge  \left|  E\setminus  B   \right|\ge \int_{1+\frac \gamma 2}^{1+\gamma} \mathscr H^{n-1}(  \partial B_s\cap E)\, ds.$$
Hence, there exists $\rho\in [1+\frac \gamma 2,\,1+  \gamma]$ such that
$$\mathscr H^{n-1}(  \partial B_{\rho}\cap E)\le  2 \gamma^{-1}|E\Delta B|.$$
In particular, by the bound above, the definition of excess, and the divergence theorem, we have
\begin{equation}
\label{eq:bound coarea1}
\begin{split}
\mathscr H^{n-1}(\partial^*E \setminus B_{\rho}) &\leq 2 {\rm Exc}(E)+ \int_{\partial^*E \setminus B_{\rho}} \frac{x}{|x|}\cdot \nu_E \,d\mathscr H^{n-1}\\
&=2 {\rm Exc}(E)+ \int_{\partial^*(E \setminus B_{\rho})} \frac{x}{|x|} \cdot \nu_{E \setminus B_{\rho}}\,d\mathscr H^{n-1}+\mathscr H^{n-1}(\partial B_\rho \cap E)\\
&=2 {\rm Exc}(E)+ \int_{E \setminus B_{\rho}} \frac{n-1}{|x|} \,dx+\mathscr H^{n-1}(\partial B_\rho \cap E)\\
&\leq 2 {\rm Exc}(E)+ (n-1)|E \setminus B_{\rho}| +2\gamma^{-1}|E\Delta B|\\
&\leq 2 {\rm Exc}(E)+ C(n,\gamma)|E \Delta B|,
\end{split}
\end{equation}
where we used that $\frac{1}{|x|} \leq 1$ outside $B_\rho$. 
Therefore,
$$
\mathscr H^{n-1}(\partial^*E \setminus B_{r}) \leq \mathscr H^{n-1}(\partial^*E \setminus B_{\rho}) \leq 2 {\rm Exc}(E)+ C(n,\gamma)|E \Delta B| \qquad \text{for all }r \geq \rho.
$$
Also by noting that
$$\nu_{B_r}(x) = -\nu_{E\setminus B_r}(x)\quad \text{ for any } \ x\in \partial B_r\cap E,$$
it follows from the divergence theorem that 
\begin{equation}
\label{eq:bound div1}
\begin{split}
\mathscr H^{n-1}(\partial B_r\cap E)&=\int_{\partial B_r\cap E}\frac{x}{|x|}\cdot \nu_{B_r} \, d\mathscr H^{n-1}\\
&\le  \mathscr H^{n-1}(\partial^* E\setminus B_r) - \int_{\partial^*(E\setminus B_r)} \frac{x}{|x|}\cdot \nu_{E\setminus B_r} \, d\mathscr H^{n-1}  \\
&=  \mathscr H^{n-1}(\partial^* E\setminus B_r) - \int_{E\setminus B_r} \frac{n-1}{|x|}\, dx\le  \mathscr H^{n-1}(\partial^* E\setminus B_r).
\end{split}
\end{equation}
Thus \eqref{eq:bound coarea1} and \eqref{eq:bound div1} imply that
\begin{align*}
P(E\setminus B_{r})&=\mathscr H^{n-1}(\partial B_r\cap E)+  \mathscr H^{n-1}(\partial^* E\setminus B_r)
 \leq 2\mathscr H^{n-1}(\partial^* E\setminus B_r)\\
&\leq 4 {\rm Exc}(E)+ C(n,\gamma)|E \Delta B|\qquad  \text{for all }r \geq \rho.
\end{align*}
Recalling that $|E \Delta B| \leq C(n) {\rm Exc}(E)^{\frac{1}2} $ (see Lemma~\ref{inverse inequ}), the first bound follows.

The bound inside the unit ball follows by a similar argument, but some small modifications are needed.
First, by the coarea formula, we can find
$\rho\in [1-\gamma ,\,1-\frac \gamma 2]$ such that
$$\mathscr H^{n-1}(  \partial B_{\rho}\setminus E)\le   \gamma^{-1}|E\Delta B|.$$
Then, since
$\nu_{B_\rho\setminus E}(x)= -\nu_{E}(x)$ for $\mathscr H^{n-1}$-a.e. $x\in \partial^*E\cap B_{\rho},$
arguing similarly to \eqref{eq:bound coarea1} and using now Lemma~\ref{lem:1 x}, we have
\begin{equation*}
\begin{split}
\mathscr H^{n-1}(\partial^*E \cap B_{\rho}) &\leq 2 {\rm Exc}(E)- \int_{B_{\rho}\setminus E} \frac{n-1}{|x|} \,dx+\mathscr H^{n-1}(\partial B_\rho \setminus E)\\
&\leq 2 {\rm Exc}(E)+\mathscr H^{n-1}(\partial B_\rho \setminus E)\\
&\leq 2 {\rm Exc}(E)+ \gamma^{-1}|E\Delta B|,
\end{split}
\end{equation*}
which still implies (as it was for \eqref{eq:bound coarea1}) that $\mathscr H^{n-1}(\partial^*E \cap B_{r})$ is small for all $r \in(1/2, \rho).$
Concerning \eqref{eq:bound div1}, in this case we use the bounded vector-field
$x/r$ in $B_r$. Since $x/r=\nu_{B_r}$ on $\partial B_r$, the divergence theorem
gives
\begin{equation*}
\begin{split}
\mathscr H^{n-1}(\partial B_r\setminus E)&=\int_{\partial B_r\setminus E}\frac{x}{|x|}\cdot \nu_{B_r} \, d\mathscr H^{n-1} \\
&\le  \mathscr H^{n-1}(\partial^* E\cap B_r)
+\int_{\partial^*(B_r\setminus E)} \frac{x}{r}\cdot \nu_{B_r\setminus E} \, d\mathscr H^{n-1}  \\
&= \mathscr H^{n-1}(\partial^* E\cap B_r) + \frac n r |B_r\setminus E|\\
&\leq \mathscr H^{n-1}(\partial^* E\cap B_r) + 2n|B\setminus E|\\
&= \mathscr H^{n-1}(\partial^* E\cap B_r) + n|B\Delta E|,
\end{split}
\end{equation*}
where in the last inequality we used $r \geq 1/2$ and $r\leq 1$.
The rest of the argument follows similarly to the case when $r\ge  1 +\gamma$.
\end{proof}

We now show that we can slightly translate a set close to the ball to ensure that its ``weighted barycenter'' coincides with the origin.
We write 
$$\mean{A} f\,dx := |A|^{-1}\int_A f\,dx$$
whenever $0<|A|<\infty$.

\begin{lem}\label{translation}
There exists a dimensional constant $\eta=\eta(n)>0$ such that the following is true:
Let $F\subset \R^n$ satisfy 
$$|B\Delta F| \le \eta\qquad \text{and}\qquad \mean{F} \frac {x}{|x|}\, dx=x_1\quad \text{with }|x_1|\leq \eta.$$
 Then there exists   $x_0\in \mathbb R^n$, with $|x_0|\le 2|x_1|$, such that 
$$\int_{F+x_0} \frac {x}{|x|}\, dx=0.$$ 
\end{lem}
\begin{proof}
Consider the continuous map
$$\Psi:\R^n\to \R^n,\qquad  y\mapsto y- \mean{F+y} \frac {x}{|x|}\, dx.$$
By a change of variable and the fact that $|x|^{-1}\in L^1_{\loc}(\mathbb R^n), n\ge 2$, one can easily check that
$$\mean{F+t\theta} \frac {x}{|x|}\, dx=\mean{F} \frac {x+t\theta}{|x+t\theta|}\, dx=x_1+t\mean{F} \bigg(\frac{\theta}{|x|}-\frac{(x\cdot \theta) x}{|x|^3} \bigg)\, dx+o(t)$$
uniformly with respect to $\theta\in\mathbb S^{n-1}$.
Also, since
$$
\mean{B}\frac{1}{|x|}\left({\rm Id}-\frac{x\otimes x}{|x|^2}\right) \, dx= {\rm Id},\,
$$
if $|F\Delta B| \leq \eta$  then the symmetric matrix
$$
A_F:=\mean{F} \frac{1}{|x|}\left(Id-\frac{x\otimes x}{|x|^2}\right) \, dx \in \R^{n\times n}
$$
satisfies $|A_F - {\rm Id}| =o_\eta(1)$, where $o_\eta(1)\to 0$ as $\eta\to 0$.
Thus, if $|y|\leq 2 |x_1|\leq 2\eta \ll 1$, we get
$$
|\Psi(y)|=|y-x_1-A_Fy +o(|y|)|\leq |x_1|+|y -A_Fy|+o(|y|)\leq |x_1|+o_\eta(1) |y|+o(|y|)\leq 2|x_1|.
$$
This implies that
$\Psi$ maps $\overline{B_{2|x_1|}}$ to itself, so it follows from the Brouwer fixed-point theorem that there exists a point $x_0\in \overline{B_{2|x_1|}}$ such that
$\Psi(x_0)=x_0,$
or equivalently
$$\mean{F+x_0} \frac {x}{|x|}\, dx=0.$$
This proves the result.
\end{proof}

We shall use the following notation throughout the paper. If $F\subset \R^n$
has finite perimeter, we define the distributional curvature vector of
$\partial F$ as the distribution
\[
\Phi\mapsto \int_{\partial^*F}{\rm div}_\tau \Phi\,d\mathscr H^{n-1},
\qquad \Phi\in C_c^1(\R^n;\R^n).
\]
When this distribution is represented by an $\R^n$-valued Radon measure, we say
that $\mathbf H_{\partial F}\in\mathcal M^{\rm v}(\R^n)$, where
$\mathcal M^{\rm v}(\R^n)$ denotes the space of vector-valued Radon measures on
$\R^n$, and we define $\mathbf H_{\partial F}$ by
\begin{equation}\label{eq:vector curvature measure}
\int_{\partial^*F}{\rm div}_\tau \Phi\,d\mathscr H^{n-1}
=\int_{\R^n}\Phi\cdot d\mathbf H_{\partial F}
\qquad \forall\,\Phi\in C_c^1(\R^n;\R^n).
\end{equation}
If $\mathbf H_{\partial F}$ is absolutely continuous with respect to
$\mathscr H^{n-1}\llcorner\partial^*F$, then\footnote{If we write
$d\mathbf H_{\partial F}=h\,d\mathscr H^{n-1}\llcorner\partial^*F$, then
$h$ has no tangential component: this follows from the orthogonality
of the generalized mean-curvature vector of a rectifiable varifold to its
approximate tangent plane; see Allard~\cite[Section~4]{A1972}.}
\[
\mathbf H_{\partial F}
=\mathcal H_{\partial^*F}\nu_F\,\mathscr H^{n-1}\llcorner\partial^*F,
\]
where $\mathcal H_{\partial^*F}$ is the scalar distributional mean curvature.

We conclude this section with a technical result that will be used later to
control the vector-valued distributional curvature measure of a set obtained
by cutting $E$ with a ball.

\begin{lem}\label{singular mean curvature}
Let $E\subset \mathbb R^n$ be a set of finite perimeter whose scalar
distributional mean curvature is represented by a function
$\mathcal H_{\partial^*E}\in L^1(\partial^* E)$. Then, for a.e. $r>0$, the set
$$G_r=E\setminus B_r$$
has distributional curvature represented by a vector-valued Radon measure $\mathbf H_{\partial G_r}$, and
we have
$$|\mathbf H_{\partial G_r}|\le |\mathcal H_{\partial^*E}|\, \mathscr H^{n-1}\lfloor_{\partial^*E\setminus B_r}+\frac{n-1}{r}\,\mathscr H^{n-1}\lfloor_{E^{(1)}\cap\partial B_r}+ C \mathscr H^{n-2}\lfloor_{\partial^*E\cap \partial B_{r}},$$
and the constant $C$ is absolute. 
\end{lem}
\begin{proof}
By coarea formula, we may assume that
$$\mathscr H^{n-2}(\partial^*E\cap \partial B_{r})<\infty. $$
By the definition of scalar distributional mean curvature, for $X\in C_c^\infty(\mathbb R^n;\, \mathbb R^n)$ with $|X|\le 1$ and $\varphi\in C^\infty(\mathbb R^n)$, we have
$$\int_{\partial^* E}\big( \varphi\, {\rm div}_\tau X -  \varphi\, \mathcal H_{\partial^* E}  \, X\cdot \nu_E\big)\, d\mathscr H^{n-1}=- \int_{\partial^* E} X \cdot \nabla_\tau \varphi  \,d\mathscr H^{n-1}.$$
By approximation, we can consider the Lipschitz function
$$
\varphi(x):=\phi\left(\frac{|x|}{\ez}+1-\frac{r}{\ez}\right)\qquad \text{where }\quad \phi(t):=\left\{
\begin{array}{ll}
0 &\text{for }t \leq 1,\\
t-1&\text{for }1 \leq t \leq 2,\\
1&\text{for }t \geq 2,
\end{array}
\right.
\qquad 
$$
and $\ez>0$ is small. This, together with the coarea formula \cite[Theorem 2, Section 3.4.3]{EG1992}, yields
\begin{align*}
&\left|\int_{\partial^* E\setminus B_{r}} \varphi\, {\rm div}_\tau X -  \varphi\, \mathcal H_{\partial^* E}   X\cdot \nu_E \, d\mathscr H^{n-1}\right|\\
&= \left| - \frac 1 \ez \int_{\partial^* E\cap (B_{r+\ez}\setminus \overline{B}_r)}   \phi'\left(\frac{|x|}{\ez}+1-\frac{r}{\ez}\right) \left(\frac x {|x|} -\frac{x\cdot \nu_E}{|x|} \nu_E\right)\cdot X \, d\mathscr H^{n-1}\right|\\
&\le  C  \ez^{-1}\int_{r}^{r+\ez}\int_{\partial^* E\cap \partial B_s}     \frac {\left|\frac x {|x|} -\frac{x\cdot \nu_E}{|x|} \nu_E\right|} {\sqrt{1-\left(\frac{x}{|x|}\cdot \nu_E\right)^2}} \, d\mathscr H^{n-2}\, ds. 
\end{align*}
Thus, noticing 
$$\left(\frac x {|x|} -\frac{x\cdot \nu_E}{|x|} \nu_E\right)\cdot \frac x {|x|}=1-\left(\frac{x}{|x|}\cdot \nu_E\right)^2$$
and letting $\ez\to 0$, we obtain that
\begin{equation}\label{E part}
\left|\int_{\partial^* E\setminus B_{r}}   {\rm div}_\tau X -  \mathcal H_{\partial^* E}  X\cdot \nu_E    \, d\mathscr H^{n-1}\right|\le C  \mathscr H^{n-2}(\partial^*E\cap \partial B_{r}) \qquad \text{for a.e. } r>0. 
\end{equation}
Likewise, for $X\in C_c^\infty(\mathbb R^n;\, \mathbb R^n)$ with $|X|\le 1$ and $\varphi\in C^\infty(\mathbb R^n)$, we also have
$$\int_{\partial  B_r} \varphi \,{\rm div}_\tau X - \varphi  \,\mathcal H_{\partial B_r}  X\cdot \nu_{B_r}  \, d\mathscr H^{n-1}=- \int_{\partial  B_r} X \cdot \nabla_\tau \varphi  \, d\mathscr H^{n-1}.$$
Now we choose, for $\ez>0$ small,
$$\varphi(x)=\max\left\{1-\frac{\dist(x,E)}{\ez},\,0\right\},$$
and applying a similar argument on the smooth hypersurface $\partial B_r$, with
cutoffs approximating the trace of $E^{(1)}\cap\partial B_r$,  leads to
\begin{equation}\label{B part}
\left|\int_{\partial B_{r}\cap E^{(1)}}   {\rm div}_\tau X -  \mathcal H_{\partial  B_r}  X\cdot \nu_{B_r}    \, d\mathscr H^{n-1}\right|\le C  \mathscr H^{n-2}(\partial^*E\cap \partial B_{r}). 
\end{equation}
Since, up to $\mathscr H^{n-1}$-negligible sets,
$$
\partial^*G_r=\big(\partial^*E\setminus B_r\big)\cup \big(E^{(1)}\cap \partial B_r\big),
$$
and the outer unit normal of $G_r$ on $E^{(1)}\cap \partial B_r$ is $-\nu_{B_r}$, combining \eqref{E part} and \eqref{B part} gives the stated bound on the total variation of $\mathbf H_{\partial G_r}$.
\end{proof}

\section{A spectral gap estimate on sets of finite perimeter}\label{sec:spectral sets FP}
 
In his celebrated work \cite{Fuglede1989}, Fuglede proved a sharp stability result for the isoperimetric inequality on $\sigma$-spherical sets by exploiting a spectral gap for the Laplacian on the sphere for functions.
Our aim is to find a suitable generalization of his estimate for sets of finite perimeter contained in a small annulus around the sphere. The main result of this section, Proposition~\ref{stability sets}, is based on a delicate interplay between compactness estimates on BV functions and suitable rearrangement techniques.

Recall that a function $u\in L^1(\mathbb S^{n-1})$ is of bounded variation (BV) if its distributional tangential gradient $\nabla_\tau u$ is a finite measure.
By the decomposition of measure, one can write $D_\tau u=\nabla_\tau u\,d\mathscr H^{n-1}+D^s_\tau u$, where $D^s_\tau u$ is singular with respect to $\mathscr H^{n-1}$. Also, $\|D_\tau u\|(\mathbb S^{n-1})$ denotes the total mass of $D_\tau u$, that is,
$$
\|D_\tau u\|(\mathbb S^{n-1})=\int_{\mathbb S^{n-1}}|\nabla_\tau u|\,d\mathscr H^{n-1}+\int_{\mathbb S^{n-1}}|D^s_\tau u|.
$$
 
We begin by proving the following inequality.
\begin{lem}\label{strong l2}
Let $\xi \in {BV(\mathbb S^{n-1})}$ satisfy, for some  $0<\sz=\sz(n)\ll1$,
\begin{equation}\label{small average}
    \left|\frac{1}{ n|B|}\int_{\mathbb S^{n-1}}\xi\,  d\mathscr H^{n-1}\right|\le \sz \left(\int_{\mathbb S^{n-1}} \frac{|\xi|^2}{\sqrt{1+\ez^2|\xi|^2}} \,d\mathscr H^{n-1}\right)^{\frac 1 2}.
\end{equation}
 Then,  for any $0<\ez<\frac 1 2,$
$$   \int_{\mathbb S^{n-1}} \frac{|\xi|^2}{\sqrt{1+\ez^2|\xi|^2}} \,d\mathscr H^{n-1}\le C(n)\left( \int_{\mathbb S^{n-1}} \frac{|\nabla_\tau  \xi|^2}{\sqrt{1+\ez^2|\nabla_\tau \xi|^2}} \,
 d\mathscr H^{n-1}+\ez^{-1} |D^s_\tau \xi|(\mathbb S^{n-1})\right).$$
In addition, assume that there exists a sequence $\xi_i\in BV(\mathbb S^{n-1})$ satisfying, for every $i$,
\begin{equation}\label{small average seq}
    \left|\frac{1}{ n|B|}\int_{\mathbb S^{n-1}}\xi_i\,  d\mathscr H^{n-1}\right|\le \sz \left(\int_{\mathbb S^{n-1}} \frac{|\xi_i|^2}{\sqrt{1+\ez_i^2|\xi_i|^2}} \,d\mathscr H^{n-1}\right)^{\frac 1 2},
\end{equation}
and
$$\int_{\mathbb S^{n-1}} \frac{|\nabla_\tau  \xi_i|^2}{\sqrt{1+\ez_i^2|\nabla_\tau \xi_i|^2}} \,
 d\mathscr H^{n-1}+ \epsilon_i^{-1}|D^s_\tau \xi_i|(\mathbb S^{n-1})\le M$$
for some $\ez_i\to 0$ and $M>0$.
Then, up to passing to a subsequence,
$\xi_i \to \xi$   strongly in $L^p (\mathbb S^{n-1})$ for some $1<p=p(n)<2$ (below the BV-Sobolev compactness exponent on $S^{n-1}$),
and
\begin{equation}
    \label{eq:xi_i to xi}
\lim_{i\to \infty}\int_{\mathbb S^{n-1}} \frac{|\xi_i|^2}{\sqrt{1+\ez_i^2|\xi_i|^2}} \,d\mathscr H^{n-1}=  \int_{\mathbb S^{n-1}} {|\xi|^2}  \,d\mathscr H^{n-1}.
\end{equation}
\end{lem}

\begin{proof}
\noindent{\bf Step 1: Energy inequalities.}
Given $\xi\in BV(\mathbb S^{n-1})$,  we set for simplicity
$$
\mathcal A_\epsilon(\xi):=\int_{\mathbb S^{n-1}} \frac{|\xi|^2}{\sqrt{1+\epsilon^2|\xi|^2}}\,d\mathscr H^{n-1},
$$
$$
\mathcal E_\epsilon(\xi):=\int_{\mathbb S^{n-1}} \frac{|\nabla_\tau \xi|^2}{\sqrt{1+\epsilon^2|\nabla_\tau \xi|^2}}\,d\mathscr H^{n-1}
+\ez^{-1}|D^s_\tau \xi|(\mathbb S^{n-1}).
$$
We first prove the estimate for smooth functions. The passage to $BV$ follows by strict approximation on the sphere. Indeed, after subtracting constants, one may choose
$\xi_h\in C^\infty(\mathbb S^{n-1})$ with
$\int\xi_h=\int\xi$, $\xi_h\to\xi$ in $L^1$, and
$D_\tau\xi_h\stackrel{*}{\rightharpoonup}D_\tau\xi$ with
$|D_\tau\xi_h|(\mathbb S^{n-1})\to |D_\tau\xi|(\mathbb S^{n-1})$. Since the
recession function of $p\mapsto |p|^2(1+\epsilon^2|p|^2)^{-1/2}$ is
$\epsilon^{-1}|p|$, Reshetnyak's theorem gives
$\mathcal E_\epsilon(\xi_h) \to \mathcal E_\epsilon(\xi)$ and
$\mathcal A_\epsilon(\xi_h)\to\mathcal A_\epsilon(\xi)$. We may therefore work
in the smooth setup.

We first consider the case where $\sz=0$ in \eqref{small average}.

\medskip
\noindent{\bf Case 1}: $\|\ez \xi\|_{L^\infty(\mathbb S^{n-1})}\le 1$. Then
\begin{equation}\label{comparable}
 \frac 1 2 \int_{\mathbb S^{n-1}}  {|\xi|^2}  \,d\mathscr H^{n-1}\le \int_{\mathbb S^{n-1}} \frac{|\xi|^2}{\sqrt{1+\ez^2|\xi|^2}} \,d\mathscr H^{n-1}\le \int_{\mathbb S^{n-1}}  {|\xi|^2}  \,d\mathscr H^{n-1}.
\end{equation}
We claim that, for any $0<\bz<1$,
\begin{equation}\label{youngs}
|\xi||\nabla_\tau \xi|\le C(\bz) \frac{|\nabla_\tau  \xi|^2}{\sqrt{1+\ez^2|\nabla_\tau \xi|^2}} + \bz |\xi|^2.
\end{equation}
Indeed, if $\ez|\nabla_\tau \xi|\le 1$, then by Young's inequality,
$$|\xi||\nabla_\tau \xi|\le C(\beta)  |\nabla_\tau \xi|^2 +\bz |\xi|^2\le \sqrt{2} C(\bz) \frac{|\nabla_\tau  \xi|^2}{\sqrt{1+\ez^2|\nabla_\tau \xi|^2}} + \bz |\xi|^2. $$
When $\ez |\nabla_\tau \xi|\ge 1$, it follows from $\ez|\xi|\le 1$ that
$$|\xi||\nabla_\tau \xi|\le \ez^{-1} |\nabla_\tau \xi|\le \sqrt{2}  \frac{|\nabla_\tau  \xi|^2}{\sqrt{1+\ez^2|\nabla_\tau \xi|^2}}. $$
Thus, we conclude the validity of \eqref{youngs}.

Now, by applying a $(1,1)$-Poincar\'e inequality and \eqref{youngs} (recall that $\sz=0$ in \eqref{small average}), we get
\begin{align*}
\int_{\mathbb S^{n-1}} |\xi|^2 \,d\mathscr H^{n-1}& \leq  C(n) \int_{\mathbb S^{n-1}} |\xi \nabla_\tau \xi| \,d\mathscr H^{n-1}\\
& \leq  C(n,\bz) \int_{\mathbb S^{n-1}}  \frac{|\nabla_\tau  \xi|^2}{\sqrt{1+\ez^2|\nabla_\tau \xi|^2}}   + C(n)\bz \int_{\mathbb S^{n-1}}   |\xi|^2 \,d\mathscr H^{n-1}.
\end{align*}
Then by taking $\bz>0$ small enough, we obtain
\begin{equation}\label{xi2 L2}
 \int_{\mathbb S^{n-1}} |\xi|^2 \,d\mathscr H^{n-1} \le  C(n)  \int_{\mathbb S^{n-1}} \frac{|\nabla_\tau  \xi|^2}{\sqrt{1+\ez^2|\nabla_\tau \xi|^2}} \,d\mathscr H^{n-1},
\end{equation}
and also
\begin{equation}\label{xi2 W11}
    \int_{\mathbb S^{n-1}} |\xi \nabla_\tau \xi| \,d\mathscr H^{n-1}\le  C(n)  \int_{\mathbb S^{n-1}} \frac{|\nabla_\tau  \xi|^2}{\sqrt{1+\ez^2|\nabla_\tau \xi|^2}} \,d\mathscr H^{n-1}.
\end{equation}
Recalling \eqref{comparable}, this proves the first part of the lemma under the assumption of Case 1.

\medskip
\noindent{\bf Case 2}: General case.
Define the truncation map
$$
T_\epsilon(t):=\max\left\{-\frac1\epsilon,\min\Bigl\{t,\frac1\epsilon\Bigr\}\right\}
$$
and consider the function
$$
F(c):=\int_{\mathbb S^{n-1}} T_\epsilon(\xi-c)\,d\mathscr H^{n-1}.
$$
Since $T_\epsilon$ is continuous and nondecreasing, $F$ is continuous and nonincreasing.
Moreover,
$$
\lim_{c\to-\infty}F(c)=\frac{\mathscr H^{n-1}(\mathbb S^{n-1})}{\epsilon}>0,
\qquad
\lim_{c\to+\infty}F(c)=-\frac{\mathscr H^{n-1}(\mathbb S^{n-1})}{\epsilon}<0,
$$
hence the intermediate value theorem yields the existence of $c_\epsilon\in\mathbb R$ for which 
$$
\int_{\mathbb S^{n-1}} T_\epsilon(\xi-c_\epsilon)\,d\mathscr H^{n-1}=0.
$$
Set
$$
\tilde\xi:=\xi-c_\epsilon,\qquad
\psi:=T_\epsilon(\tilde\xi),\qquad
\eta:=\tilde\xi-\psi
=
\Bigl(\tilde\xi-\frac 1 \ez \Bigr)_+-\Bigl(\tilde\xi+\frac 1 \ez \Bigr)_-.
$$
Then
$$
\tilde\xi=\psi+\eta,\qquad \int_{\mathbb S^{n-1}}\psi\,d\mathscr H^{n-1}=0,
\qquad \|\epsilon\psi\|_{L^\infty(\mathbb S^{n-1})}\le 1.
$$
Also,
$$
\nabla_\tau\psi=\left\{
\begin{array}{ll}
\nabla_\tau\xi& \text{$\mathscr H^{n-1}$-a.e. on $\mathbb S^{n-1}\cap \{\ez|\tilde\xi|\leq  1\}$},\\
0& \text{$\mathscr H^{n-1}$-a.e. on $\mathbb S^{n-1}\cap \{\ez|\tilde\xi|> 1\}$},
\end{array}\right.
$$
and similarly
$$
\nabla_\tau\eta=\left\{
\begin{array}{ll}
0& \text{$\mathscr H^{n-1}$-a.e. on $\mathbb S^{n-1}\cap \{\ez|\tilde\xi|\leq  1\}$},\\
\nabla_\tau\xi& \text{$\mathscr H^{n-1}$-a.e. on $\mathbb S^{n-1}\cap \{\ez|\tilde\xi|> 1\}$}.
\end{array}\right.
$$
Since
$$
\left|\frac{d}{dt}\frac{|t|^2}{\sqrt{1+\epsilon^2|t|^2}}\right|
\le C\,\epsilon^{-1}\qquad\forall\,t\in\mathbb R,
$$
we have
$$
\mathcal A_\epsilon(\xi)\le \mathcal A_\epsilon(\tilde\xi)+C\epsilon^{-1}|c_\epsilon|.
$$
On the other hand, since $\int_{\mathbb S^{n-1}}\xi\,d\mathscr H^{n-1}=0$ and
$\int_{\mathbb S^{n-1}}\psi\,d\mathscr H^{n-1}=0$, we get
$$
-c_\epsilon\,\mathscr H^{n-1}(\mathbb S^{n-1})
=\int_{\mathbb S^{n-1}}\tilde\xi\,d\mathscr H^{n-1}
=\int_{\mathbb S^{n-1}}\eta\,d\mathscr H^{n-1},
$$
hence
$$
|c_\epsilon|\le C(n)\int_{\mathbb S^{n-1}}|\eta|\,d\mathscr H^{n-1}.
$$
Therefore
$$
\mathcal A_\epsilon(\xi)\le \mathcal A_\epsilon(\tilde\xi)+C(n)\epsilon^{-1}\int_{\mathbb S^{n-1}}|\eta|\,d\mathscr H^{n-1}.
$$
We claim that, pointwise,
\begin{equation}\label{pointwise tilde xi}
\frac{|\tilde\xi|^2}{\sqrt{1+\epsilon^2|\tilde\xi|^2}}
\le
C |\psi|^2 + C\epsilon^{-1}|\eta|.
\end{equation}
Indeed, if $|\eta|\le |\psi|$, then $|\tilde\xi|\le 2|\psi|$ and $|\psi|\le \epsilon^{-1}$, so
$$
\frac{|\tilde\xi|^2}{\sqrt{1+\epsilon^2|\tilde\xi|^2}}
\le |\tilde\xi|^2 \le 4|\psi|^2.
$$
If $|\eta|>|\psi|$, then $|\tilde\xi|\le 2|\eta|$, and hence
$$
\frac{|\tilde\xi|^2}{\sqrt{1+\epsilon^2|\tilde\xi|^2}}
\le \epsilon^{-1}|\tilde\xi|\le 2\epsilon^{-1}|\eta|,
$$
which concludes the validity of \eqref{pointwise tilde xi}. 
Thus
\begin{equation}\label{two terms}
    \begin{split}
     \int_{\mathbb S^{n-1}} \frac{|\xi|^2}{\sqrt{1+\ez^2|\xi|^2}} \,d\mathscr H^{n-1}
 & \leq  C \int_{\mathbb S^{n-1}}  {|\psi|^2}  \,d\mathscr H^{n-1}+ C \ez^{-1}\int_{\mathbb S^{n-1}} |\eta|  \, \,d\mathscr H^{n-1}.
    \end{split}
\end{equation}
Also, by applying the results obtained in Case 1 to $\psi$ (recall that $\int_{\mathbb S^{n-1}}\psi\,d\mathscr H^{n-1}=0$), we get
\begin{equation}
    \begin{split}
   \int_{\mathbb S^{n-1}} |\psi|^2 \,d\mathscr H^{n-1}
  &\le  C(n)  \int_{\mathbb S^{n-1}} \frac{|\nabla_\tau  \psi|^2}{\sqrt{1+\ez^2|\nabla_\tau \psi|^2}} \,d\mathscr H^{n-1}\\
 &=  C(n)  \int_{\mathbb S^{n-1}\cap \{\ez|\tilde\xi|\le 1\}} \frac{|\nabla_\tau  \xi|^2}{\sqrt{1+\ez^2|\nabla_\tau \xi|^2}} \,d\mathscr H^{n-1}\\
  &\le  C(n)  \int_{\mathbb S^{n-1} } \frac{|\nabla_\tau  \xi|^2}{\sqrt{1+\ez^2|\nabla_\tau \xi|^2}} \,d\mathscr H^{n-1}.  \label{first part}
    \end{split}
\end{equation}
and
\begin{equation}
    \begin{split}
\int_{\mathbb S^{n-1}} |\psi \nabla_\tau \psi| \,d\mathscr H^{n-1}& \le  C(n)  \int_{\mathbb S^{n-1}} \frac{|\nabla_\tau  \psi|^2}{\sqrt{1+\ez^2|\nabla_\tau \psi|^2}} \,d\mathscr H^{n-1}  \\
& \le  C(n)  \int_{\mathbb S^{n-1} } \frac{|\nabla_\tau  \xi|^2}{\sqrt{1+\ez^2|\nabla_\tau \xi|^2}} \,d\mathscr H^{n-1}. \label{psi W11}
    \end{split}
\end{equation}
Combining \eqref{first part} and \eqref{two terms}, we see that
it suffices to show that
\begin{equation} \label{eta controlled by xi}
\ez^{-1}\int_{\mathbb S^{n-1}} |\eta|  \, \,d\mathscr H^{n-1}\le C(n)\int_{\mathbb S^{n-1}}   \frac{|\nabla_\tau  \xi|^2}{\sqrt{1+\ez^2|\nabla_\tau \xi|^2}}  \, \,d\mathscr H^{n-1}.
\end{equation}
We now claim that $0$ is the median of $\eta$. Indeed, note that
$$
\{\eta>0\}=\{\psi= \epsilon^{-1}\} \quad \text{  and }\quad  \{\eta<0\}=\{\psi=-\epsilon^{-1}\}.
$$
Then, since $\psi\in [-\epsilon^{-1},  \epsilon^{-1}]$ and $\int_{\mathbb S^{n-1}}\psi\,d\mathscr H^{n-1}=0$,   
$$ \text{ both }  \ \mathscr H^{n-1}(\{\eta>0\}) \ \text{ and } \ \mathscr H^{n-1}(\{\eta<0\})$$ 
are no more than  $\frac 1 2\mathscr H^{n-1}(\mathbb S^{n-1})$, which proves the claim. Thus,  the $(1,1)$-Poincar\'e inequality yields
$$
\int_{\mathbb S^{n-1}} |\eta| \,d\mathscr H^{n-1}
\le C(n)\int_{\mathbb S^{n-1}} |\nabla_\tau\eta|\,d\mathscr H^{n-1}.
$$

\medskip
\noindent {\bf SubCase 2.1}:
We first assume that
\begin{equation}\label{eta small}
\int_{\mathbb S^{n-1}} |\eta|  \, \,d\mathscr H^{n-1}\le \ez^{-1}|\mathbb S^{n-1}\cap \{\ez|\tilde\xi|>1\}|.
\end{equation}
Therefore, as $|\psi|=\ez^{-1}$ on $\mathbb S^{n-1}\cap\{\ez |\tilde\xi|> 1\}$,
\begin{equation}\label{eta small2}
\ez^{-1}\int_{\mathbb S^{n-1}} |\eta|  \, \,d\mathscr H^{n-1}\leq
\int_{\mathbb S^{n-1}\cap\{\ez|\tilde\xi|>1\}}  \ez^{-2} \,d\mathscr H^{n-1}\leq \int_{\mathbb S^{n-1}}  {|\psi|^2}  \,d\mathscr H^{n-1}.
\end{equation}
Thus, \eqref{first part} directly gives \eqref{eta controlled by xi}.

\medskip
\noindent {\bf SubCase 2.2}:
When
\begin{equation}\label{eta large}
\int_{\mathbb S^{n-1}} |\eta|  \, \,d\mathscr H^{n-1}\ge \ez^{-1}|\mathbb S^{n-1}\cap\{\ez|\tilde\xi|>1\}|,
\end{equation}
we write
$$\mathcal E= \{\ez|\nabla_\tau\xi|<1\}\cap\{\ez|\tilde\xi|>1\}\cap\mathbb S^{n-1}, \quad  \mathcal F= \{\ez|\nabla_\tau\xi|\ge 1\}\cap\{\ez|\tilde\xi|>1\}\cap\mathbb S^{n-1}.$$
Then, since $\nabla_\tau\eta=0$ $\mathscr H^{n-1}$-a.e. on $\mathbb S^{n-1}\cap \{\ez|\tilde\xi|\leq  1\}$, the $(1,1)$-Poincar\'e inequality implies
\begin{equation}
    \begin{split}
\ez^{-1}\int_{\mathbb S^{n-1}} |\eta|  \, \,d\mathscr H^{n-1} & \leq  C(n) \ez^{-1} \int_{\mathbb S^{n-1} }  |\nabla_\tau \eta| \, \,d\mathscr H^{n-1}\\
& =  C(n) \ez^{-1} \int_{\mathcal E}  |\nabla_\tau \eta| \, \,d\mathscr H^{n-1}+ C(n) \ez^{-1} \int_{\mathcal F}  |\nabla_\tau \eta| \, \,d\mathscr H^{n-1}. \label{D eta}
    \end{split}
\end{equation}

\medskip
\noindent {\bf SubCase 2.2.1}:
Suppose first that
\begin{equation}\label{F large}
\int_{\mathcal F}  |\nabla_\tau \eta| \, \,d\mathscr H^{n-1}\ge \int_{\mathcal E}  |\nabla_\tau \eta| \, \,d\mathscr H^{n-1}.
\end{equation}
Then, by the definition of $\mathcal F$ and \eqref{D eta},
\begin{multline}\label{control 1}
\ez^{-1}\int_{\mathbb S^{n-1}} |\eta|  \, \,d\mathscr H^{n-1}  \leq  C(n) \ez^{-1} \int_{\mathbb S^{n-1} }  |\nabla_\tau \eta| \, \,d\mathscr H^{n-1}\\
\leq  2C(n) \ez^{-1} \int_{\mathcal F}  |\nabla_\tau \eta| \, \,d\mathscr H^{n-1}\le 2\sqrt{2} C(n)  \int_{\mathcal F} \frac{|\nabla_\tau  \xi|^2}{\sqrt{1+\ez^2|\nabla_\tau \xi|^2}} \, \,d\mathscr H^{n-1},
    \end{multline}
where we used that $\nabla_\tau\eta=\nabla_\tau  \xi$ $\mathscr H^{n-1}$-a.e. on $\mathbb S^{n-1}\cap \{\ez|\tilde\xi|>  1\}$ together with the definition of $\mathcal F$.
This gives \eqref{eta controlled by xi}.

\medskip
\noindent {\bf SubCase 2.2.2}: Now the remaining case is
\begin{equation}\label{E large}
\int_{\mathcal F}  |\nabla_\tau \eta| \, \,d\mathscr H^{n-1}\le \int_{\mathcal E }  |\nabla_\tau \eta| \, \,d\mathscr H^{n-1}.
\end{equation}
Notice that  \eqref{eta large} yields
$$|\mathbb S^{n-1}\cap\{\ez|\tilde\xi|>1\}|= \int_{\mathbb S^{n-1}\cap\{\ez|\tilde\xi|>1\}} 1\, d\mathscr H^{n-1} \le \ez \int_{\mathbb S^{n-1}} |\eta|  \, \,d\mathscr H^{n-1}.$$
Therefore, since $\tilde\xi=\psi+\eta$, the definition of $\mathcal E$ gives
\begin{equation}\label{control 2}
|\mathcal E|\le \int_{\mathcal E} \ez|\tilde\xi| \,  \,d\mathscr H^{n-1}= \int_{\mathbb S^{n-1}\cap\{\ez|\tilde\xi|>1\}} (1+\ez|\eta|) \,  \,d\mathscr H^{n-1}\le 2\ez \int_{\mathbb S^{n-1}} |\eta|  \, \,d\mathscr H^{n-1}.
\end{equation}
Thus, by \eqref{D eta}, \eqref{E large}, H\"older's inequality, the definition of $\mathcal E$, and \eqref{control 2}, we get
\begin{align*}
\ez^{-1}\int_{\mathbb S^{n-1}} |\eta|  \, \,d\mathscr H^{n-1}& \leq 2 C(n) \ez^{-1} \int_{\mathcal E}  |\nabla_\tau \eta| \, \,d\mathscr H^{n-1} =2C(n) \ez^{-1} \int_{\mathcal E}  |\nabla_\tau \xi| \, \,d\mathscr H^{n-1}\nonumber\\
& \leq 2 C(n)\ez^{-1} \left(\int_{\mathcal E}   \frac{|\nabla_\tau  \xi|^2}{\sqrt{1+\ez^2|\nabla_\tau \xi|^2}}  \, \,d\mathscr H^{n-1}\right)^{1/2} \left(\int_{\mathcal E}    {\sqrt{1+\ez^2|\nabla_\tau \xi|^2}}  \, \,d\mathscr H^{n-1}\right)^{1/2}\nonumber\\
& \leq 2\sqrt{2} C(n)\ez^{-1} |\mathcal E|^{1/2} \left(\int_{\mathcal E}   \frac{|\nabla_\tau  \xi|^2}{\sqrt{1+\ez^2|\nabla_\tau \xi|^2}}  \, \,d\mathscr H^{n-1}\right)^{1/2}\nonumber\\
& \leq 4 C(n) \ez^{-1/2}  \left(\int_{\mathbb S^{n-1}}   |\eta|  \, \,d\mathscr H^{n-1} \right)^{1/2}    \left(\int_{\mathcal E}   \frac{|\nabla_\tau  \xi|^2}{\sqrt{1+\ez^2|\nabla_\tau \xi|^2}}  \, \,d\mathscr H^{n-1}\right)^{1/2}
\end{align*}
which yields \eqref{eta controlled by xi}  and
\begin{equation}\label{eta W11}
    \ez^{-1} \int_{\mathcal E}  |\nabla_\tau \eta| \, \,d\mathscr H^{n-1} \le C(n)\int_{\mathcal E}   \frac{|\nabla_\tau  \xi|^2}{\sqrt{1+\ez^2|\nabla_\tau \xi|^2}}  \, \,d\mathscr H^{n-1}.
\end{equation}
This concludes the proof of the first part of the lemma when $\sz=0$.

\medskip
We now treat the general case in \eqref{small average}. Let
$$
m:=\frac{1}{n|B|}\int_{\mathbb S^{n-1}}\xi\,d\mathscr H^{n-1},
\qquad
\bar\xi:=\xi-m.
$$
Then $\int_{\mathbb S^{n-1}}\bar\xi\,d\mathscr H^{n-1}=0$ and $\mathcal E_\epsilon(\bar\xi)=\mathcal E_\epsilon(\xi)$.
Set
$$
g_\epsilon(t):=\min\{|t|^2,\epsilon^{-1}|t|\}.
$$
One easily checks that
$$
\frac1{\sqrt2}\,g_\epsilon(t)\le \frac{|t|^2}{\sqrt{1+\epsilon^2|t|^2}}\le g_\epsilon(t)
\qquad\forall\, t\in\mathbb R,
$$
and that
$$
g_\epsilon(a+b)\le 2g_\epsilon(a)+2g_\epsilon(b)\qquad\forall\, a,b\in\mathbb R.
$$
Hence
$$
\mathcal A_\epsilon(\xi)=\mathcal A_\epsilon(\bar\xi+m)\le C \mathcal A_\epsilon(\bar\xi)+C \mathcal A_\epsilon(m)\le C \mathcal A_\epsilon(\bar\xi)+C |m|^2.
$$
Applying the already proved case $\sz=0$ to $\bar\xi$, we get
$$
\mathcal A_\epsilon(\bar\xi)\le C(n)\mathcal E_\epsilon(\bar\xi)=C(n)\mathcal E_\epsilon(\xi).
$$
On the other hand, \eqref{small average} gives
$$
|m|^2\le \sz^2 \mathcal A_\epsilon(\xi).
$$
Choosing $\sz=\sz(n)>0$ sufficiently small, we can absorb the $|m|^2$ term into the left-hand side and conclude the first part of the lemma in full generality.

\medskip

\noindent{\bf Step 2: Energy convergence.} In this second part, we argue directly in $BV$.
By the first part, the quantities
$$
\mathcal  A_{\epsilon_i}(\xi_i)
$$
are uniformly bounded. In particular, \eqref{small average seq} yields a uniform bound on the averages
$$
m_i:=\frac{1}{n|B|}\int_{\mathbb S^{n-1}}\xi_i\,d\mathscr H^{n-1}.
$$
Also, since $\epsilon_i<1$ for $i$ large,
$$
|a|\le 1+\frac{2|a|^2}{\sqrt{1+|a|^2}}
\le 1+\frac{2|a|^2}{\sqrt{1+\epsilon_i^2|a|^2}}
\qquad\forall\, a\in\mathbb R^n,
$$
and therefore
$$
\int_{\mathbb S^{n-1}} |\nabla_\tau \xi_i|\,d\mathscr H^{n-1}
+|D^s_\tau \xi_i|(\mathbb S^{n-1})
\le C(n,M).
$$
Hence, up to passing to a subsequence,
$$
\xi_i\to \xi \qquad\text{strongly in }L^p(\mathbb S^{n-1})
$$
for some $1<p=p(n)<2$, and then a.e. on $\mathbb S^{n-1}$.

For each $i$, let $c_i\in\mathbb R$ be such that
$$
\int_{\mathbb S^{n-1}} T_{\epsilon_i}(\xi_i-c_i)\,d\mathscr H^{n-1}=0.
$$
Set
$$
\tilde\xi_i:=\xi_i-c_i,\qquad
\psi_i:=T_{\epsilon_i}(\tilde\xi_i),\qquad
\eta_i:=\tilde\xi_i-\psi_i.
$$
Then the argument of Case 2 above applies verbatim to $(\tilde\xi_i,\psi_i,\eta_i)$, and gives the analogues of \eqref{two terms}, \eqref{first part}, \eqref{psi W11}, \eqref{eta controlled by xi}, \eqref{control 1}, and \eqref{eta W11}, with constants independent of $i$.

In particular, from \eqref{first part}, \eqref{psi W11}, and the $BV$ chain
rule for the truncation, we obtain
$$
\sup_i \left(
\int_{\mathbb S^{n-1}} |\psi_i|^2\,d\mathscr H^{n-1}
+
\int_{\mathbb S^{n-1}} |\psi_i\nabla_\tau \psi_i|\,d\mathscr H^{n-1}
+\epsilon_i^{-1}|D^s_\tau\psi_i|(\mathbb S^{n-1})
\right)<\infty.
$$
Moreover $\|D_\tau\psi_i\|(\mathbb S^{n-1})\le
\|D_\tau\xi_i\|(\mathbb S^{n-1})$. Hence $\psi_i$ is compact in $L^1$, while
$\psi_i^2$ is compact in $L^1$, since
$|D_\tau(\psi_i^2)|\le C|\psi_i\nabla_\tau\psi_i|\,\HH^{n-1}
+C\epsilon_i^{-1}|D^s_\tau\psi_i|$. Thus, up to extracting a further subsequence,
\begin{equation}\label{strong L2 psi}
\text{$\psi_i$ converges strongly in $L^2(\mathbb S^{n-1})$ to some function $v$.}
\end{equation}
Since $\epsilon_i\psi_i\to 0$ a.e. on $\mathbb S^{n-1}$ and
$$\frac{|\psi_i|^2}{\sqrt{1+\ez_i^2|\psi_i|^2}}\le |\psi_i|^2, $$
Vitali's convergence theorem yields
\begin{equation}\label{strong L2 psi2}
\lim_{i\to \infty}\int_{\mathbb S^{n-1}}  \frac{|\psi_i|^2}{\sqrt{1+\ez_i^2|\psi_i|^2}} \,d\mathscr H^{n-1}=  \int_{\mathbb S^{n-1}} {|v|^2}  \,d\mathscr H^{n-1}.
\end{equation}

Let
$$
E_i:=\{\eta_i\neq 0\}=\{\epsilon_i|\tilde\xi_i|>1\}\subset \{|\psi_i|=\epsilon_i^{-1}\}.
$$
Then, by \eqref{strong L2 psi},
\begin{equation}\label{measure to 0}
\epsilon_i^{-2}|E_i|
=
\int_{E_i} |\psi_i|^2\,d\mathscr H^{n-1}
\le
\int_{\{|\psi_i|\ge \epsilon_i^{-1}\}} |\psi_i|^2\,d\mathscr H^{n-1}
\longrightarrow 0.
\end{equation}
Moreover, the analogue of \eqref{eta controlled by xi} gives
$$
\epsilon_i^{-1}\int_{\mathbb S^{n-1}}|\eta_i|\,d\mathscr H^{n-1}\le C(n,M),
$$
hence
$$
\int_{\mathbb S^{n-1}}|\eta_i|\,d\mathscr H^{n-1}\to 0.
$$
Since
$$
\int_{\mathbb S^{n-1}}\xi_i\,d\mathscr H^{n-1}
=
c_i\,\mathscr H^{n-1}(\mathbb S^{n-1})+\int_{\mathbb S^{n-1}}\eta_i\,d\mathscr H^{n-1},
$$
and the averages of $\xi_i$ are uniformly bounded, it follows that $\{c_i\}$ is bounded. Up to a subsequence,
$
c_i\to c\in\mathbb R.
$
Then, by \eqref{strong L2 psi},
$$
\psi_i+c_i \to v+c \qquad\text{strongly in }L^2(\mathbb S^{n-1}).
$$
Since $\eta_i\to 0$ in $L^1(\mathbb S^{n-1})$ and $\xi_i=\psi_i+c_i+\eta_i$, while $\xi_i\to\xi$ in measure, we conclude that
$$
\xi=v+c \qquad\text{a.e. on }\mathbb S^{n-1}.
$$

We now show that
$$
\lim_{i\to\infty} \mathcal  A_{\epsilon_i}(\xi_i)=\int_{\mathbb S^{n-1}} |\xi|^2\,d\mathscr H^{n-1}.
$$
On $\mathbb S^{n-1}\setminus E_i$, one has $\eta_i=0$, hence $\xi_i=\psi_i+c_i$ there. Therefore
$$
|\mathcal A_{\epsilon_i}(\xi_i)-\mathcal A_{\epsilon_i}(\psi_i+c_i)|
\le
\int_{E_i}\frac{|\xi_i|^2}{\sqrt{1+\epsilon_i^2|\xi_i|^2}}\,d\mathscr H^{n-1}
+
\int_{E_i}\frac{|\psi_i+c_i|^2}{\sqrt{1+\epsilon_i^2|\psi_i+c_i|^2}}\,d\mathscr H^{n-1}.
$$
Since $|\psi_i|=\epsilon_i^{-1}$ on $E_i$ and $c_i$ is bounded, for $i$ large enough
$$
|\psi_i+c_i|\le 2\epsilon_i^{-1}\qquad\text{on }E_i,
$$
and therefore
$$
\int_{E_i}\frac{|\psi_i+c_i|^2}{\sqrt{1+\epsilon_i^2|\psi_i+c_i|^2}}\,d\mathscr H^{n-1}
\le C\,\epsilon_i^{-2}|E_i|\to 0
$$
by \eqref{measure to 0}. Also, on $E_i$,
$$
\xi_i=(\psi_i+c_i)+\eta_i.
$$
Since $|\psi_i+c_i|\ge \frac{1}{2}\epsilon_i^{-1}$ on $E_i$ for $i$ large, one has
$$
\frac{|\xi_i|^2}{\sqrt{1+\epsilon_i^2|\xi_i|^2}}
\le C\frac{|\psi_i+c_i|^2}{\sqrt{1+\epsilon_i^2|\psi_i+c_i|^2}}+C\epsilon_i^{-1}|\eta_i|
\qquad\text{on }E_i.
$$
Hence
$$
\int_{E_i}\frac{|\xi_i|^2}{\sqrt{1+\epsilon_i^2|\xi_i|^2}}\,d\mathscr H^{n-1}
\le
C\int_{E_i}\frac{|\psi_i+c_i|^2}{\sqrt{1+\epsilon_i^2|\psi_i+c_i|^2}}\,d\mathscr H^{n-1}
+
C\epsilon_i^{-1}\int_{\mathbb S^{n-1}}|\eta_i|\,d\mathscr H^{n-1}.
$$
If SubCase 2.1 occurs along a subsequence, then by the analogue of \eqref{eta small2} and \eqref{strong L2 psi},
$$
\epsilon_i^{-1}\int_{\mathbb S^{n-1}}|\eta_i|\,d\mathscr H^{n-1}\to 0.
$$
If SubCase 2.2.1 or SubCase 2.2.2 occurs along a subsequence, then the analogues of \eqref{control 1} or \eqref{eta W11} imply that $\epsilon_i^{-1}\eta_i$ is uniformly bounded in $BV(\mathbb S^{n-1})$; since its support is contained in $E_i$ and $|E_i|\to 0$, compactness in $L^1$ implies
$$
\epsilon_i^{-1}\eta_i\to 0 \qquad\text{strongly in }L^1(\mathbb S^{n-1}).
$$
Thus, in all cases,
$$
\epsilon_i^{-1}\int_{\mathbb S^{n-1}}|\eta_i|\,d\mathscr H^{n-1}\to 0,
$$
and consequently
$$
\mathcal A_{\epsilon_i}(\xi_i)-\mathcal A_{\epsilon_i}(\psi_i+c_i)\to 0.
$$
Finally, since $\psi_i+c_i\to \xi$ strongly in $L^2(\mathbb S^{n-1})$ and
$$
\frac{|\psi_i+c_i|^2}{\sqrt{1+\epsilon_i^2|\psi_i+c_i|^2}}
\le |\psi_i+c_i|^2,
$$
Vitali's convergence theorem again yields
$$
\lim_{i\to \infty} \mathcal A_{\epsilon_i}(\psi_i+c_i)
=
\int_{\mathbb S^{n-1}} |\xi|^2\,d\mathscr H^{n-1}.
$$
Combining the last two formulae, we obtain \eqref{eq:xi_i to xi}. This concludes the proof.

\end{proof}

%

Now we are ready to generalize (and improve, as needed in our setting) the spectral gap result of Fuglede \cite{Fuglede1989}, originally proved for Lipschitz perturbations, to BV perturbations of the sphere.
In the following, $\{Y_0,\ldots,Y_n\}$ denote the first $n+1$ eigenfunctions of the Laplace--Beltrami operator on the sphere, namely
\begin{equation}
    \label{LB}
    Y_0(\theta)=\frac{1}{\sqrt{n|B|}},\qquad Y_i(\theta)=\frac{\theta_i}{\sqrt{|B|}}\quad i=1,\ldots,n,
\end{equation}
see for instance \cite[Proof of Theorem 3.1, Step 3]{F2015}.

\begin{lem}\label{spectrum}
Given any $N\ge 2n$, there exist $\zeta=\zeta(n)>0$, $\sz_0=\sz_0(n,N)>0$, and $\delta_0=\delta_0(n,N)>0$ so that the following holds.\\
For each $0<\sz<\sz_0$, any set Borel $K\subset \mathbb S^{n-1}$ with  $\mathscr H^{n-1}(K)\le \dz_0$, and any function $\xi \in BV(\mathbb S^{n-1})$ satisfying both
$$\int_{\mathbb S^{n-1}}  \frac{|\nabla_\tau  \xi|^2}{\sqrt{1+ |\nabla_\tau \xi|^2}} \,d\mathscr H^{n-1}+ \int_{\mathbb S^{n-1}}|D^s_\tau \xi| \le \sz,$$
and
\begin{equation}\label{almost orthogonal general}
\biggl|\int_{\mathbb S^{n-1}} Y_i\,\xi \,d\mathscr H^{n-1}\biggr| \leq C(n)\sigma \left(\int_{\mathbb S^{n-1}} \frac{|\xi|^2}{\sqrt{1+ |\xi|^2}} \,d\mathscr H^{n-1}\right)^{\frac 1 2}\qquad \text{for }i=0,\ldots,n,
\end{equation}
we have:\\
either
$$\int_{\mathbb S^{n-1}\setminus K}  \frac{|\nabla_\tau  \xi|^2}{\sqrt{1+ |\nabla_\tau \xi|^2}} \,d\mathscr H^{n-1} + \int_{\mathbb S^{n-1}\setminus K}|D^s_\tau \xi|\geq  \left(n-1+3 \zeta\right)  \int_{\mathbb S^{n-1}} \frac{|\xi|^2}{\sqrt{1+ |\xi|^2}} \,d\mathscr H^{n-1},$$ 
or
$$\int_{\mathbb S^{n-1}}  \frac{|\nabla_\tau  \xi|^2}{\sqrt{1+ |\nabla_\tau \xi|^2}} \,d\mathscr H^{n-1} + \int_{\mathbb S^{n-1}}|D^s_\tau \xi|\geq 2N  \int_{\mathbb S^{n-1}} \frac{|\xi|^2}{\sqrt{1+ |\xi|^2}} \,d\mathscr H^{n-1}.$$
In particular, if in addition $\|\xi\|_{L^{\infty}(\mathbb S^{n-1})}\le \sz$, then:\\
either
$$ \int_{\mathbb S^{n-1}\setminus K}  \frac{|\nabla_\tau  \xi|^2}{\sqrt{1+ |\nabla_\tau \xi|^2}} \,d\mathscr H^{n-1}+ \int_{\mathbb S^{n-1}\setminus K}|D^s_\tau \xi| \geq  \left(n-1+2 \zeta\right) \int_{\mathbb S^{n-1}} |\xi|^2 \,d\mathscr H^{n-1},$$ 
or
$$\int_{\mathbb S^{n-1}}  \frac{|\nabla_\tau  \xi|^2}{\sqrt{1+ |\nabla_\tau \xi|^2}} \,d\mathscr H^{n-1} + \int_{\mathbb S^{n-1}}|D^s_\tau \xi|\geq N  \int_{\mathbb S^{n-1}} |\xi|^2 \,d\mathscr H^{n-1}.$$
\end{lem}
\begin{proof}
First of all, it follows from \cite[Proof of Theorem 3.1, Step 3]{F2015} that, when $\sigma \ll1 $ in \eqref{almost orthogonal}, the Poincar\'e constant for  functions in $W^{1,\,2}(\mathbb S^{n-1})$ is arbitrarily close to $2n$. More precisely,  for any $\xi \in W^{1,\,2}(\mathbb S^{n-1})$ satisfying
\begin{equation}\label{almost orthogonal}
    \biggl|\int_{\mathbb S^{n-1}} Y_i\,\xi \,d\mathscr H^{n-1}\biggr| \leq C(n)\sigma \left(\int_{\mathbb S^{n-1}}  |\xi|^2   \,d\mathscr H^{n-1}\right)^{\frac 1 2}\qquad \text{for }i=0,\ldots,n,
\end{equation}
it holds that,
$$\int_{\mathbb S^{n-1}} {|\nabla_\tau \xi|^2} \,d\mathscr H^{n-1}\geq \left(2n-C(n)\sz\right) \int_{\mathbb S^{n-1}} \xi^2 \,d\mathscr H^{n-1},    $$
In particular, given $\zeta=\zeta(n) \in \left(0,\frac{n-1}{4}\right)$,  one can choose $\sigma$ sufficiently small so that
\begin{equation} \label{sharp Poincare}
\int_{\mathbb S^{n-1}} {|\nabla_\tau \xi|^2} \,d\mathscr H^{n-1}\geq  (n-1+4\zeta) \int_{\mathbb S^{n-1}} \xi^2 \,d\mathscr H^{n-1} \quad  \forall\, \xi \in W^{1,\,2}(\mathbb S^{n-1}) \text{ satisfying \eqref{almost orthogonal}}.
\end{equation}

We prove the result by contradiction. If it fails,  Lemma~\ref{strong l2} gives numbers $\sigma_i\downarrow0$,
sets $K_i\subset\mathbb S^{n-1}$ with
$\mathscr H^{n-1}(K_i)=:\delta_i\to0$, and functions
$\xi_i\in BV(\mathbb S^{n-1})$ satisfying \eqref{almost orthogonal general},
with
\[
\int_{\mathbb S^{n-1}}  \frac{|\nabla_\tau  \xi_i|^2}{\sqrt{1+ |\nabla_\tau \xi_i|^2}} \,d\mathscr H^{n-1}
+|D^s_\tau\xi_i|(\mathbb S^{n-1})\le\sigma_i,
\]
such that both inequalities
\begin{equation} \label{contradiction1}
\int_{\mathbb S^{n-1}\setminus K_i}  \frac{|\nabla_\tau  \xi_i|^2}{\sqrt{1+ |\nabla_\tau \xi_i|^2}} \,d\mathscr H^{n-1}
+|D^s_\tau\xi_i|(\mathbb S^{n-1}\setminus K_i)
\le  \left(n-1+ 3 \zeta\right)  \int_{\mathbb S^{n-1}} \frac{|\xi_i|^2}{\sqrt{1+ |\xi_i|^2}} \,d\mathscr H^{n-1}, \end{equation}
and
\begin{equation} \label{contradiction2}
\int_{\mathbb S^{n-1}}  \frac{|\nabla_\tau  \xi_i|^2}{\sqrt{1+ |\nabla_\tau \xi_i|^2}} \,d\mathscr H^{n-1}
+|D^s_\tau\xi_i|(\mathbb S^{n-1})
\leq 2N  \int_{\mathbb S^{n-1}} \frac{|\xi_i|^2}{\sqrt{1+ |\xi_i|^2}} \,d\mathscr H^{n-1}
\end{equation}
hold.

Define
$$
\alpha_i:=\left(\int_{\mathbb S^{n-1}}  \frac{|\nabla_\tau  \xi_i|^2}{\sqrt{1+ |\nabla_\tau \xi_i|^2}} \,d\mathscr H^{n-1}
+|D^s_\tau\xi_i|(\mathbb S^{n-1})\right)^{1/2} \to 0,
$$
We may assume $\alpha_i>0$ and consider
$$\hat \xi_i=\frac{\xi_i}{\alpha_i};$$
indeed, if $\alpha_i=0$, then
$\xi_i$ is constant, and the condition \eqref{almost orthogonal general} with $Y_0$ forces
$\xi_i\equiv0$ for $i$ large; then both alternatives are trivially true.
Then according to \eqref{almost orthogonal general}, $\hat \xi_i$ satisfies 
$$\biggl|\int_{\mathbb S^{n-1}} Y_i\,\hat\xi_i \,d\mathscr H^{n-1}\biggr| \leq C(n)\sigma_i \left(\int_{\mathbb S^{n-1}} \frac{|\hat\xi_i|^2}{\sqrt{1+ \alpha_i^2|\hat\xi_i|^2}} \,d\mathscr H^{n-1}\right)^{\frac 1 2}\qquad \text{for }i=0,\ldots,n,$$
and
\begin{equation} \label{norm 1}
\int_{\mathbb S^{n-1}}  \frac{|\nabla_\tau \hat \xi_i|^2}{\sqrt{1+ \alpha_i^2|\nabla_\tau \hat \xi_i|^2}} \,d\mathscr H^{n-1}
+\alpha_i^{-1}|D^s_\tau\hat\xi_i|(\mathbb S^{n-1})=1.
\end{equation}
Also, it follows from \eqref{contradiction1}, \eqref{contradiction2}, \eqref{norm 1}, and Lemma~\ref{strong l2} that
\begin{multline}
\label{other direction}
\max\left\{ \frac{1}{n-1+3 \zeta}\int_{\mathbb S^{n-1}\setminus K_i}  \frac{|\nabla_\tau \hat \xi_i|^2}{\sqrt{1+ \alpha_i^2|\nabla_\tau \hat \xi_i|^2}} \,d\mathscr H^{n-1},
 \frac{1}{2N}\right\}\\
 \le   \int_{\mathbb S^{n-1}} \frac{|\hat \xi_i|^2}{\sqrt{1+  \alpha_i^2 |\hat \xi_i|^2}} \,d\mathscr H^{n-1}\le C(n). 
\end{multline}
In particular, using Lemma~\ref{strong l2} again, up to passing to a subsequence, the functions  $\hat \xi_i$  converge strongly in $L^p(\mathbb S^{n-1})$, $p>1$,
to some function $\hat \xi_\infty$ satisfying \eqref{almost orthogonal}.
Furthermore, 
\begin{equation} \label{norm 2} \lim_{i\to \infty}\int_{\mathbb S^{n-1}} \frac{|\hat \xi_i|^2}{\sqrt{1+\alpha_i^2|\hat\xi_i|^2}} \,d\mathscr H^{n-1}=  \int_{\mathbb S^{n-1}} {|\hat \xi_\infty|^2}  \,d\mathscr H^{n-1}.\end{equation}
In particular $\hat\xi_\infty\not\equiv 0$ as a consequence of \eqref{other direction}. 

Now, given  $\ez>0$ small, define the sets
$$\mathcal R_{i,\ez}=\{x\in \mathbb S^{n-1}\colon |\nabla_\tau \hat \xi_i|\le \ez\alpha_i^{-1}\}, \qquad \mathcal S_{i,\ez}=\{x\in \mathbb S^{n-1}\colon |\nabla_\tau \hat \xi_i|\ge \ez\alpha_i^{-1}\}.$$
Since
\begin{equation}\label{R and S}
  1\ge  \int_{\mathbb S^{n-1}}  \frac{|\nabla_\tau \hat \xi_i|^2}{\sqrt{1+ \alpha_i^2|\nabla_\tau \hat \xi_i|^2}} \,d\mathscr H^{n-1}\ge \frac 1 2 \int_{\mathcal R_{i,\ez}} |\nabla_\tau \hat \xi_i|^2\, \,d\mathscr H^{n-1} +\ez \alpha_i^{-1} \int_{\mathcal S_{i,\ez}} |\nabla_\tau \hat \xi_i|\, \,d\mathscr H^{n-1}
\end{equation} 
and $|\nabla_\tau \hat \xi_i|$ is uniformly bounded away from $0$ inside $\mathcal S_{i,\ez}$, we deduce that 
\begin{equation}\label{no singular part}
\mathscr H^{n-1}(\mathcal S_{i,\ez})+  \int_{\mathcal S_{i,\ez}} |\nabla_\tau \hat \xi_i|\, \,d\mathscr H^{n-1}\le C(n) \ez^{-1}\alpha_i \to 0 \quad \text{ as }i \to \infty. 
\end{equation}
Moreover, \eqref{R and S} also gives
\begin{equation}\label{uniform L2 bound}
    \|\nabla_\tau \hat \xi_i\,\chi_{\mathcal R_{i,\ez}}\|_{L^2(\mathbb S^{n-1})}\le C(n).
\end{equation} 
Also $|D^s_\tau\hat\xi_i|(\mathbb S^{n-1})\le \alpha_i$ by \eqref{norm 1}.
Thanks to  \eqref{no singular part} and \eqref{uniform L2 bound}, the sequence
$\hat \xi_i$ is uniformly bounded in $BV$, the singular parts vanish, and
$\nabla_\tau \hat \xi_i$ is uniformly integrable. Hence, according to Dunford--Pettis theorem (see e.g.\ \cite{B1981}), $\hat {\xi}_\infty \in W^{1,1}(\mathbb S^{n-1})$ and 
$$\nabla_\tau \hat \xi_i  \, \text{ converges weakly to } \, \nabla_\tau \hat {\xi}_\infty \ \text{ in $L^1(\mathbb S^{n-1}, \mathbb  R^n)$}.$$
Combining the uniform integrability of $\nabla_\tau \hat \xi_i$
with the fact that $\mathscr H^{n-1}(\mathcal S_{i,\ez})+\mathscr H^{n-1}(K_i)\to 0$, we actually deduce that 
$$\nabla_\tau \hat \xi_i \chi_{\mathcal R_{i,\ez}\setminus K_i} \, \text{ converges weakly to } \, \nabla_\tau \hat {\xi}_\infty \ \text{ in $L^1(\mathbb S^{n-1}, \mathbb  R^n)$}.$$
Noticing now that the left-hand side above is uniformly bounded in $L^2(\mathbb S^{n-1})$ (see \eqref{uniform L2 bound}), it follows that $\nabla_\tau \hat {\xi}_\infty \in L^2(\mathbb S^{n-1}, \mathbb  R^n)$, and that the weak convergence above actually holds in $L^2$:
$$\nabla_\tau \hat \xi_i\chi_{\mathcal R_{i,\ez}\setminus K_i}    \, \text{ converges weakly to } \, \nabla_\tau \hat {\xi}_\infty\ \text{ in } L^2(\mathbb S^{n-1}, \mathbb  R^n).$$
Therefore, by weak lower-semicontinuity we obtain
\begin{align}
\liminf_{i\to \infty}\int_{\mathcal R_{i,\ez}\setminus K_i}  \frac{|\nabla_\tau \hat \xi_i|^2}{\sqrt{1+ \alpha_i^2|\nabla_\tau \hat \xi_i|^2}} \,d\mathscr H^{n-1}& \ge  (1+\ez^2)^{-1/2} \liminf_{i\to \infty} \int_{\mathcal R_{i,\ez}\setminus K_i} |\nabla_\tau \hat \xi_i|^2  \,d\mathscr H^{n-1} \nonumber\\
& \ge  (1+\ez^2)^{-1/2} \int_{\mathbb S^{n-1}} |\nabla_\tau \hat \xi_\infty|^2 \,d\mathscr H^{n-1}. \label{gradient convergence}
\end{align}
Thus, recalling \eqref{other direction} and \eqref{norm 2}, we get
$$
    \frac{1}{(1+\ez^2)^{1/2} \left(n-1+  3 \zeta\right)} \int_{\mathbb S^{n-1}} {|\nabla_\tau  \hat \xi_\infty|^2}  \,d\mathscr H^{n-1}\le  \int_{\mathbb S^{n-1}} |\hat \xi_\infty|^2  \,d\mathscr H^{n-1}.
$$
Since $\hat \xi_\infty$ satisfies \eqref{almost orthogonal},
this contradicts \eqref{sharp Poincare} for $\ez>0$ sufficiently small. This concludes the proof of the first part of the lemma.

The second part of the lemma follows immediately, since 
$$\frac {1}{\sqrt{1+|\xi|^2}}\ge 1-\sz$$
when  $\|\xi\|_{L^{\infty}(\mathbb S^{n-1})}\le \sz \ll 1.$ Hence, it suffices to  choose $\sz$ sufficiently small, depending on $n$ and $\zeta=\zeta(n)$.
\end{proof}

Our next step is to generalize Lemma~\ref{spectrum} to general sets of finite perimeter. We first prove two auxiliary lemmas. 
In the following, given $U\subset \mathbb R^n$, we denote by $$\mathcal{C}_{U}:=\big\{t\theta\colon t>0, \theta\in U\big\}$$ 
the cone over it.

\begin{lem}\label{small oscillation}
Let $0<\sz=\sz(n)<\frac 1 {4}$ small, and $G\subset \mathbb R^n$ be a set of finite perimeter with $|G|=|B|$ and
$$B_{1-\sigma}\subset G\subset B_{1+\sigma}.$$ 
Define
\begin{align*}
 \mathcal B:=\big\{\theta\in \mathbb S^{n-1}\colon \#\big\{(\partial^*G\setminus \big\{x \in \partial^*G\,:\, x\cdot \nu_{G}=0\big\})\cap \{t\theta\colon t>0\}\big\}\ge 2\big\}.
\end{align*}
Suppose that
 $B_{1/2}\subset G\subset B_2$.
Then, for $\mathscr H^{n-1}$-a.e. $\theta\in\mathcal B$, there are numbers
$$
0<a_0(\theta)<b_1(\theta)<a_1(\theta)<\cdots < b_{k(\theta)}(\theta)<a_{k(\theta)}(\theta),
$$
with $k(\theta)\ge1$, such that
\begin{equation}
\label{eq:ray-decomposition-leveling}
G_\theta:=G\cap\{t\theta:t>0\}
 =
 (0,a_0(\theta))\cup\bigcup_{j=1}^{k(\theta)}(b_j(\theta),a_j(\theta))
\end{equation}
up to one-dimensional negligible sets. 
Moreover,
  \begin{equation}\label{eq:exc G5}
\mathscr H^{n-1}(\big\{x \in \partial^*G\,:\, x\cdot \nu_{G}=0\big\} )+\mathscr H^{n-1}(\mathcal B ) +\mathscr H^{n-1}(\partial^*G \cap \mathcal{C}_{\mathcal B }) \leq C(n) {\rm Exc}(G ).
\end{equation} 
In particular, 
\begin{equation}\label{eq:exc G6}
  \mathscr H^{n-1}(\partial^*G \cap \mathcal{C}_{\mathcal B }) \le C(n) \int_{\partial^*G\cap \mathcal C_{\mathcal B}}  \left(1-\frac{x}{|x|}\cdot \nu_G \right)\, d\mathscr H^{n-1}.
\end{equation}
\end{lem}
\begin{proof}
By the definition of excess,
\begin{equation}\label{eq:exc G1}
\begin{split}
{\rm Exc}(G )&\geq \int_{\big\{x \in \partial^*G\,:\, x\cdot \nu_{G}=0\big\} } \left(1- \frac{x}{|x|}\cdot \nu_{G } \right) \, d\mathscr H^{n-1}=\mathscr H^{n-1}(\big\{x \in \partial^*G\,:\, x\cdot \nu_{G}=0\big\} ),
\end{split}
\end{equation}
\begin{equation}\label{eq:exc G2}
\begin{split}
{\rm Exc}(G )&\geq \int_{\partial^*G \cap \mathcal{C}_{U}} \left(1- \frac{x}{|x|}\cdot \nu_{G } \right) \, d\mathscr H^{n-1} \qquad \forall\,U \subset \mathbb S^{n-1} \text{ Borel.}
\end{split}
\end{equation}
Set
$$
\mathcal T_G:=\big\{x\in\partial^*G:\ x\cdot\nu_G=0\big\}.
$$
Applying the area formula \cite[Theorem 1, Section 3.3.3]{EG1992} to
$(\partial^*G\cap\mathcal C_{\mathcal B})\setminus\mathcal T_G$, it follows from
\eqref{eq:exc G2} with $U=\mathcal B$ that
\begin{align}
{\rm Exc}(G )&\geq \int_{\partial^*G\cap \mathcal C_{\mathcal B}}  \left(1-\frac{x}{|x|}\cdot \nu_G \right)\, d\mathscr H^{n-1} \geq  \int_{\mathcal B } \sum_{\substack{x\in \partial^*G,\\ x\cdot \nu_{G }\neq 0}} \left(1- \frac{x}{|x|}\cdot \nu_{G } \right)\frac{|x|^{n}}{ |x \cdot \nu_{G }|}\, d\mathscr H^{n-1}(\theta),\label{excess G in cone}
\end{align}
and
\begin{equation}\label{measure G in cone}
    \mathscr H^{n-1}((\partial^* G\cap \mathcal C_{\mathcal B})\setminus\mathcal T_G)=\int_{\mathcal B } \sum_{\substack{x\in \partial^*G,\\ x\cdot \nu_{G }\neq 0}} \frac{|x|^{n}}{ |x \cdot \nu_{G }|}\, d\mathscr H^{n-1}(\theta).
\end{equation}
We next observe that, thanks to \cite[Chapters 3.2 and 3.11]{AFP2000} and \cite[Theorem~G]{CCF2005}, for $\mathscr H^{n-1}$-a.e. $\theta\in \mathbb S^{n-1}$ the set 
$$\{t\theta\colon t>0\}\cap G $$
 is of finite perimeter in $(0,\infty)$. In particular, by the definition of $\mathcal B $, for $\mathscr H^{n-1}$-a.e. $\theta \in \mathcal B $  there exist  $\{a_i=a_i(\theta)\}_{0\leq i \leq k(\theta)}$ and $\{b_i=b_i(\theta)\}_{1\leq i \leq k(\theta)}$, with $0<a_0<b_1<a_1<\cdots<b_{k}<a_k$ and $k=k(\theta)\ge 1$, such that
\begin{equation*} 
    \{t\theta\colon t>0\}\cap G = (0,\,a_0) \cup \bigcup_{i=1}^{k}(b_i,\,a_i).
\end{equation*}  
This proves \eqref{eq:ray-decomposition-leveling}.

Moreover, since $\partial G\subset A_{1-\sz,\,1+\sz}$, we have
$$a_i,\,b_i\in (1-\sz,\,1+\sz).$$
Now let us fix such a $\theta$. By the structure of \eqref{eq:ray-decomposition-leveling}, we always have 
\begin{equation}\label{normal at b}
    x\cdot \nu_G\leq 0 \quad \text{ when $x=b_i\theta$},
\end{equation} 
and hence at this point
\begin{equation}\label{b case}
\left(1- \frac{x}{|x|}\cdot \nu_{G } \right)\frac{|x|^{n}}{ |x \cdot \nu_{G }|}\ge  \frac{|x|^{n}}{ |x \cdot \nu_{G }|}.
\end{equation}

Now we 
discuss in two cases depending on the normal of $G$ at $x=a_i\theta,$ $ 0\le i\le k$. 

{\noindent \bf Case 1: }  $\frac{x}{|x|}\cdot \nu_G(x)<\frac 7 8.$
Then 
$$ 1- \frac{x}{|x|}\cdot \nu_{G }(x) \ge \frac 1 8,$$
and hence
\begin{equation}\label{a case 1}
    \left(1- \frac{x}{|x|}\cdot \nu_{G }(x) \right)\frac{|x|^{n}}{ |x \cdot \nu_{G }(x)|}\ge \frac 1 8 \frac{|x|^{n}}{ |x \cdot \nu_{G }(x)|}.
\end{equation} 

{\noindent \bf Case 2:}  $\frac{x}{|x|}\cdot \nu_G(x)\ge \frac 7 8.$ 
In this case, since $|x|\in (1-\sz,\,1+\sz)$, 
$$\frac{|x|^n}{|x\cdot \nu|}\le \frac 8 7(1+C(n)\sz). $$
Let $\hat x = b_i \theta$ with $b_0:=b_1.$ Then by \eqref{normal at b} one has
$$\left(1- \frac{\hat x}{|\hat x|}\cdot \nu_{G }(\hat x) \right)\frac{|\hat x|^{n}}{ |\hat x \cdot \nu_{G }(\hat x)|}\ge 1-C(n)\sz.$$
Therefore, for $\sz=\sz(n)>0$ small we get
\begin{equation}\label{a case 2}
    \left(1- \frac{\hat x}{|\hat x|}\cdot \nu_{G }(\hat x) \right)\frac{|\hat x|^{n}}{ |\hat x \cdot \nu_{G }(\hat x)|}\ge c(n) \frac{|x|^n}{|x\cdot \nu_G(x)|}. 
\end{equation}
Combining now \eqref{b case} with \eqref{a case 1} and \eqref{a case 2}, we eventually conclude that, for  $\mathscr H^{n-1}$-a.e. $\theta \in \mathcal B$, 
$$
\sum_{\substack{x\in \{t\theta\colon t>0\}\cap \partial^* G,\\ x\cdot\nu_G\neq0}}\frac{|x|^{n}}{ |x \cdot \nu_{G }|}  \le C(n)   \sum_{\substack{x\in \{t\theta\colon t>0\}\cap \partial^* G,\\ x\cdot\nu_G\neq0}} \left(1- \frac{x}{|x|}\cdot \nu_{G } \right)\frac{|x|^{n}}{ |x \cdot \nu_{G }|}.
$$
Recalling \eqref{eq:exc G1}, \eqref{excess G in cone}, and
\eqref{measure G in cone}, we get
\begin{align}
\label{eq:exc G3}
     \mathscr H^{n-1}(\partial^* G\cap \mathcal C_{\mathcal B})
     &\le
     \mathscr H^{n-1}(\mathcal T_G\cap \mathcal C_{\mathcal B})  \notag\\
     &\quad+C(n)   \int_{\mathcal B } \sum_{\substack{x\in \partial^*G,\\ x\cdot \nu_{G }\neq 0}} \left(1- \frac{x}{|x|}\cdot \nu_{G } \right)\frac{|x|^{n}}{ |x \cdot \nu_{G }|}\, d\mathscr H^{n-1}(\theta) \notag\\
     &\le C(n) {\rm Exc}(G ),
\end{align}
which also yields \eqref{eq:exc G6}. 

The estimate on $\mathscr H^{n-1}(B)$ follows from a similar (but simpler) argument. Indeed,  for $\mathscr H^{n-1}$-a.e. $\theta \in \mathcal B$, choose the first non-tangential reentry point $b_\theta\theta$ in \eqref{eq:ray-decomposition-leveling}. Since $|b_\theta|>\frac 1 2$,
\begin{equation}\label{eq:exc G4}
{\rm Exc}(G )\geq \int_{\mathcal B } \frac{b_\theta^{n}}{ b_\theta |\theta \cdot \nu_{G }(b_\theta \theta)|}\, d\mathscr H^{n-1}(\theta)\geq   \int_{\mathcal B } b_\theta^{n-1}\, d\mathscr H^{n-1}(\theta) \geq 2^{-(n-1)}\mathscr H^{n-1}(\mathcal B ).
\end{equation}
Combining \eqref{eq:exc G1}, \eqref{eq:exc G3}, and \eqref{eq:exc G4}, we deduce the validity of \eqref{eq:exc G5}.

\end{proof}

To proceed further, given a set of finite perimeter $G$, we consider its star-shaped rearrangement $G^S$ defined as follows: for $\theta\in \mathbb S^{n-1}$, define $G_\theta:= G\cap \{t\theta\colon t>0 \}$
and 
$$h(\theta)=
\left\{
\begin{array}{cc}
\int_{G_\theta} r^{n-1}\, d\mathscr H^1(r) & \text{when } \ \mathscr H^{1}(G_\theta)>0, \\
0 & \text{when } \ \mathscr H^{1}(G_\theta)=0.
\end{array}
\right.
$$
Then, set $R(\theta)= (n h(\theta))^{\frac 1 n}$ and 
$$G^S=\{t\theta\colon \theta\in \mathbb S^{n-1},\,0\leq t<R(\theta)\}. $$
We observe that 
\begin{equation}
\label{eq:vol G GS}
|G^S|=|G|,
\end{equation}
see \cite{K1985, K1986}. 
In addition, since by construction
\begin{align*}
\int_{G_\theta} r^{n-1}\,dr= h(\theta)=\int_0^{R(\theta)} r^{n-1}\, dr,
\end{align*}
we deduce that 
$$
 \left|\int_{G^S} \frac{x}{|x|}\, dx - \int_{G} \frac{x}{|x|}\, dx\right| 
\le \int_{\mathbb S^{n-1}} \left|\int_{G_\theta} r^{n-1}\,dr - \int_0^{R(\theta)} r^{n-1}\, dr\right|\, d\mathscr H^{n-1}(\theta)=0,
$$
that is, $G$ and $G^S$ have the same weighted barycenter:
\begin{equation}
\label{eq:bar G GS}
\int_{G^S} \frac{x}{|x|}\, dx =\int_{G} \frac{x}{|x|}\, dx.
\end{equation}
However it is not generally true that 
$P(G^S)\le P(G),$
see \cite[Lemma 1.2]{K1986}.
Nevertheless, we can show the following: 

\begin{lem}\label{rearrange G}
For any $\sz \in (0,1/2)$, let $G\subset \mathbb R^n$ be a set of finite perimeter satisfying
$$B_{1-\sz}\subset G\subset B_{1+\sz}. $$
Then  
$$\mathscr H^{n-1}(\partial^*G^S\cap\mathcal C_U)\le (1+C(n)\sz) \mathscr H^{n-1}(\partial^*G\cap\mathcal C_U)
\qquad \forall\, U\subset \mathbb S^{n-1} \text{ Borel.}$$
In particular,   $G^S=B+u$ with $u\in BV(\mathbb S^{n-1})$. 
\end{lem}
\begin{proof}
Consider a  family of open disjoint subsets $\{\mathcal V_i\}_{1\leq i \leq M=M(n,\sigma)}$ whose closures cover  $\mathbb S^{n-1}$ and such that
\begin{equation}
    \label{eq:pG V}
    \mathscr{H}^{n-1}\big((\partial^*G \cup \partial^*G^S) \cap \partial \mathcal{C}_{\mathcal V_i}\big)=0\qquad \forall\,i=1,\ldots,M.
\end{equation}
We also assume that each set $\mathcal V_i$ has diameter at most $\sz$, 
so that it can be written as a $C(n)\sz$-Lipschitz graph over $\mathcal W_i\subset \Pi_i$, where $\Pi_i\subset \R^n$ is a $(n-1)$-dimensional linear subspace. 

Fix $i \in \{1,\ldots,M\}$ and assume, without loss of generality, that $\Pi_i=\R^{n-1}\times \{0\}$ and that $\mathcal C_{\mathcal V_i}$ points in the $e_n$ direction. Then, writing $x=(x',x_n) \in \R^{n-1}\times \R$, we define
$$\Psi_i : \left(B_{1+\sz}\setminus B _{1-\sz}\right)\cap \mathcal{C}_{\mathcal V_i} \to \mathcal W_i\times [0,3\sigma], \qquad x\mapsto \left(\frac{x'}{|x|},\int_{1-\sigma}^{|x|} r^{n-1}\, dr\right).$$
In this way
$$\Psi_i\colon   \left(B_{1+\sz}\setminus B _{1-\sz}\right)\cap \mathcal{C}_{\mathcal V_i} \to \mathcal W_i\times [0,3\sigma] $$
is a  $\big(1+C(n)\sz\big)$-biLipschitz map such that:\\
- $\Psi_i$ sends radial lines to lines parallel to $e_n$;\\
- $\mathcal W_i\times \{0\} =\Psi_i\left(\partial B_{1-\sigma}\cap \mathcal C_{\mathcal V_i}\right)$;\\
- composing $\Psi_i$ with the two-valued mapping
$$\mathcal R\colon \mathcal W_i\times [0,3\sigma] \to \mathcal W_i \times [-3\sigma,3\sigma], \qquad  \mathcal R(x', x_n)=\big\{(x', - x_n),(x',x_n)\big\} $$
our star-shaped rearrangement of $G$ corresponds, inside $\mathcal C_{\mathcal V_i}$, to the Steiner symmetrization of $\mathcal R\circ \Psi_i\big((G \cap \mathcal C_{\mathcal V_i})\setminus  B_{1-\sz}\big)$ with respect to the hyperplane $\Pi_i$. 

Since the Steiner symmetrization locally decreases the perimeter \cite[Theorem 14.4]{M2012}, it follows that
\begin{multline*}
2\mathscr H^{n-1}\left(\Psi_i(\partial^*G^S\cap \mathcal C_{\mathcal V_i \cap U})\right)=\mathscr H^{n-1}\left(\mathcal R\circ \Psi_i(\partial^*G^S\cap \mathcal C_{\mathcal V_i \cap U})\right)\\
\leq 
\mathscr H^{n-1}\left(\mathcal R\circ \Psi_i(\partial^*G\cap \mathcal C_{\mathcal V_i \cap U})\right)=2\mathscr H^{n-1}\left(\Psi_i(\partial^*G\cap \mathcal C_{\mathcal V_i \cap U})\right).
\end{multline*}
Note that the comparison is made for relative perimeters in the corresponding cylinders.
Thus, since $\Psi_i$ is $\big(1+C(n)\sz\big)$-biLipschitz,
$$
\mathscr H^{n-1}\left(\partial^*G^S\cap \mathcal C_{\mathcal V_i \cap U}\right) \leq \big(1+C(n)\sz\big)
\mathscr H^{n-1}\left(\partial^*G\cap \mathcal C_{\mathcal V_i \cap U}\right).
$$
Recalling \eqref{eq:pG V}, the result follows by summing the inequalities above over $i \in \{1,\ldots,M\}$
\end{proof}

\begin{lem}\label{lem:projection exceptional cone}
Let $G\subset\R^n$ be a set of finite perimeter with
$\partial^*G\subset A_{r,R}$ for some $0<r<R<\infty$, and set
\begin{equation}\label{eq:PiTG}
    \Pi(x):=\frac{x}{|x|},\qquad
\mathcal T_G:=\{x\in\partial^*G:\ x\cdot\nu_G(x)=0\}.
\end{equation}
Then
\[
\HH^{n-1}(\Pi(\mathcal T_G))=0.
\]
Moreover, if $Z\subset A_{r,R}$ is Borel with
$\HH^{n-1}(\Pi(Z))=0$, then
\[
\HH^{n-1}\big((\partial^*G\setminus \mathcal T_G)\cap
\Pi^{-1}(\Pi(\mathcal T_G\cup Z))\big)=0.
\]
\end{lem}

\begin{proof}
Since $\partial^*G$ is countably $(n-1)$-rectifiable and lies in an annulus,
$\Pi$ is Lipschitz on $\partial^*G$. Its tangential Jacobian is
\[
J_{\partial G}\Pi(x)=\frac{|x\cdot\nu_G(x)|}{|x|^n}
\qquad\text{for }\HH^{n-1}\text{-a.e. }x\in\partial^*G.
\]
Hence $J_{\partial G}\Pi=0$ on $\mathcal T_G$, and the area formula gives
\[
\HH^{n-1}(\Pi(\mathcal T_G))=0.
\]
Now let $Z$ be as in the statement and set
\[
A_m:=\big((\partial^*G\setminus \mathcal T_G)\cap\Pi^{-1}(\Pi(\mathcal T_G\cup Z))\big)
\cap\{J_{\partial G}\Pi\ge m^{-1}\}.
\]
By the area formula again,
\[
\frac1m\HH^{n-1}(A_m)
\le\int_{A_m}J_{\partial G}\Pi\,d\HH^{n-1}
=\int_{\Pi(A_m)}N(A_m,\theta)\,d\HH^{n-1}(\theta)=0,
\]
because $\Pi(A_m)\subset\Pi(\mathcal T_G\cup Z)$ and this last set is
$\HH^{n-1}$-negligible. Since
\[
(\partial^*G\setminus \mathcal T_G)\cap\Pi^{-1}(\Pi(\mathcal T_G\cup Z))
=\bigcup_{m=1}^\infty A_m,
\]
the second assertion follows.
\end{proof}

\begin{lem}\label{lem:star trace outside bad}
Let $G$ be a set of finite perimeter satisfying
$B_{1-\sigma}\subset G\subset B_{1+\sigma}$, and let
$G^S=B+u$ be its star-shaped rearrangement. Let $\mathcal B$ be as in
Lemma~\ref{small oscillation}, and $\mathcal T_G$ as in \eqref{eq:PiTG}.
Then, up to
$\HH^{n-1}$-negligible sets,
$$
\partial^*G\setminus(\mathcal C_{\mathcal B}\cup\mathcal T_G)
=\partial^*G^S\setminus\mathcal C_{\mathcal B},
\qquad \text{with}\quad
\nu_G=\nu_{G^S}\quad\text{there.}
$$
\end{lem}

\begin{proof}
By Lemma~\ref{lem:projection exceptional cone} applied with $Z=\emptyset$, we have
$\HH^{n-1}(\Pi(\mathcal T_G))=0$ and
\[
\HH^{n-1}\big((\partial^*G\setminus\mathcal T_G)\cap
\Pi^{-1}(\Pi(\mathcal T_G))\big)=0.
\]
It is therefore enough to argue away from this negligible set.
By the slicing theorem used in Lemma~\ref{small oscillation}, for a.e.
$\theta\notin\mathcal B \cup \Pi(\mathcal T_G)$ the one-dimensional section $G_\theta$ has exactly one
non-tangential boundary point. Since $B_{1-\sigma}\subset G\subset B_{1+\sigma}$,
this section is $(0,a(\theta))$, up to negligible sets. The rearranged radius is
therefore determined by
$$
\int_0^{R(\theta)}r^{n-1}\,dr=\int_0^{a(\theta)}r^{n-1}\,dr,
$$
hence $R(\theta)=a(\theta)$. The equality of the reduced boundaries and of the
normals follows from the area formula on the non-tangential part. 
\end{proof}

We are now ready to show the following stability result. 
\begin{prop}\label{stability sets}
Let $G\subset \mathbb R^n$ be a set of finite perimeter satisfying
\begin{equation}\label{normalize G}
|G|=|B| \quad \text{ and } \quad \int_G \frac{x}{|x|}\,dx=0. 
\end{equation}
Then, for any $\kappa>0$ there exist $\delta_1=\dz_1(n,\,\kappa)>0$ and $0<\sz_1=\sz_1(n,\,\kappa)<\frac 1 n$ such that the following hold: Given $\sigma \in (0,\sz_1)$, assume that
$${\rm Exc}(G)\le \dz_1 \qquad \text{and}\qquad B_{1-\sz}\subset G\subset B_{1+\sz}.$$
Then
\begin{multline}\label{eq:kappa G1}
 \kappa \mathscr H^{n-1}(\partial^* G\cap \mathcal{C}_{\mathcal B})+ \int_{\partial^* G  } \left(1-\frac{(x\cdot \nu_G)^2}{|x|^2}\right) \, d\mathscr H^{n-1}  \\
 \ge (n-1+\zeta)\int_{\partial^* G} \left(\frac{1}{|x|}+|x|-2\right)  \,  d\mathscr H^{n-1},
\end{multline}
and
$$ \int_{\partial^*  G}   \left(\frac{1}{|x|}+|x|-2\right)\, d\mathscr H^{n-1}\ge c(n)| G\Delta B|^2.$$
\end{prop}
\begin{proof}
Let $G^S$ denote the star-shaped rearrangement defined before. Then also $G^S$ satisfies the same inclusion. In particular,
$$G^S=B+u$$ 
for some $\|u\|_{L^\infty(\mathbb S^{n-1})}\le \sz. $ Moreover, $u\in BV(\mathbb S^{n-1})$ by Lemma~\ref{rearrange G}. Its graph energy is small when $\sigma$ and ${\rm Exc}(G)$ are small: indeed, Lemma~\ref{rearrange G} applied with $U=\mathbb S^{n-1}$ and \eqref{difference perimeter} give $P(G^S)-P(B)\le C(n)\big({\rm Exc}(G)+\sigma P(G)\big)$, and the area formula for $G^S=B+u$ yields
\[
\int_{\mathbb S^{n-1}}\frac{|\nabla_\tau u|^2}{\sqrt{1+|\nabla_\tau u|^2}}\,d\mathscr H^{n-1}
\,+\,|D^s_\tau u|(\mathbb S^{n-1})
\le \omega_n(\sigma,{\rm Exc}(G)),
\]
with $\omega_n(\sigma,t)\to0$ as $\sigma,t\to0$. Decreasing $\sigma_1$ and $\delta_1$ if necessary, we may therefore apply Lemma~\ref{spectrum} to $u$.

Note now that, as a consequence of
\eqref{eq:vol G GS} and \eqref{eq:bar G GS}, also $G^S$ satisfies the properties of $G$ in \eqref{normalize G}.
Using polar coordinates, these imply
$$
\int_{\mathbb S^{n-1}}\big[(1+u(\theta))^n-1\big]\,d\mathscr H^{n-1}(\theta)=0, \qquad \int_{\mathbb S^{n-1}}(1+u(\theta))^n\theta\,d\mathscr H^{n-1}(\theta)=0.
$$
Hence, since $\|u\|_{L^\infty(\mathbb S^{n-1})}\le \sz$ we get 
$$
\int_{\mathbb S^{n-1}}u(\theta)\,d\mathscr H^{n-1}(\theta)=O\left(\sigma \int_{\mathbb S^{n-1}}|u(\theta)|\,d\mathscr H^{n-1}(\theta)\right),$$
$$
\int_{\mathbb S^{n-1}}u(\theta)\theta_i\,d\mathscr H^{n-1}(\theta)=O\left(\sigma \int_{\mathbb S^{n-1}}|u(\theta)|\,d\mathscr H^{n-1}(\theta)\right)\qquad \forall\,i=1,\ldots,n.
$$
Recalling \eqref{LB}, this 
implies that $u$ satisfies \eqref{almost orthogonal general}. 

Now, let $\mathcal B\subset \mathbb S^{n-1}$ defined as in Lemma~\ref{small oscillation}. Then, with $N>2n$ to be determined, we apply Lemma~\ref{spectrum} to find 
$\sz_0,\,\dz_0>0$,
and then we choose $\delta_1$ small enough to ensure that  $\mathscr H^{n-1}(\mathcal B)\leq \delta_0$
(recall \eqref{eq:exc G5} and \eqref{exc symm diff}).

Our first goal is to prove \eqref{eq:kappa G1}. For this, we first do some preliminary estimates and then distinguish between the two cases provided by Lemma~\ref{spectrum}.

Recall from Lemma~\ref{lem:star trace outside bad} that, up to $\mathscr H^{n-1}$-negligible sets,
$$\partial^*G\setminus(\mathcal C_{\mathcal B}\cup\big\{x \in \partial^*G\,:\, x\cdot \nu_{G}=0\big\})
=\partial^*G^S\setminus\mathcal C_{\mathcal B},$$
and note that $\frac1{1+t}+(1+t)-2\leq t^2+O(t^3)$. On the tangency set
$\{x\in\partial^*G:\ x\cdot\nu_G=0\}\setminus\mathcal C_{\mathcal B}$, the
left-hand angular integrand in \eqref{eq:kappa G1} is equal to $1$, while
$\frac1{|x|}+|x|-2\le 2\sigma^2$; hence, for $\sigma$ sufficiently small, this
part is absorbed pointwise. Thus, recalling that $\|u\|_{L^\infty(\mathbb S^{n-1})}\le \sz$, when $\sigma$ is sufficiently small, we get from a change of variable that
\begin{equation}\label{volume term F}
    \begin{split}
 &\int_{\partial^* G^S\setminus \mathcal{C}_{\mathcal B}} \left(\frac{1}{|x|}+|x|-2\right) \,   d\mathscr H^{n-1}\\
 & \leq  (1+C(n) \sz)\int_{\mathbb S^{n-1}\setminus \mathcal B} |u|^2 \sqrt{1+(1+u)^{-2} |\nabla_\tau u|^2}\,  d\mathscr H^{n-1} + C(n) \sz^2 \int_{\mathbb S^{n-1}\setminus \mathcal B}|D^s_\tau u|  \\
& \leq (1+C(n) \sz)\int_{\mathbb S^{n-1}} |u|^2  \,  d\mathscr H^{n-1} +C(n) \sz^2 \left[\int_{\mathbb S^{n-1}\setminus \mathcal B} \frac{|\nabla_\tau u|^2}{\sqrt{1+|\nabla_\tau u|^2}}    \,  d\mathscr H^{n-1} + \int_{\mathbb S^{n-1}\setminus \mathcal B}|D^s_\tau u|\right],
    \end{split}
\end{equation}
where we used that, as $|u|\le \sz<\frac 1 2$,
$$
\sqrt{1+(1+t)^{-2} |p|^2} -1\leq C\frac{|p|^2}{\sqrt{1+|p|^2}} \qquad \forall\,t \in [-1/2,1/2],\,p\in\R^n.
$$
Analogously,  by the assumption $\partial^* G\subset A_{1-\sz,\,1+\sz}$, for $\sigma$ small we have
\begin{equation}
\label{eq:sigma G}
0\le  \frac{1}{|x|}+|x|-2 \le 2\sz^2  \quad \text{ for any } \ x\in \partial^* G,
\end{equation}
therefore
\begin{equation}\label{error volume term}
 \int_{\partial^* G \cap \mathcal{C}_{\mathcal B}} \left(\frac{1}{|x|}+|x|-2\right) \,  d\mathscr H^{n-1}
\le 2 \sz^2 \mathscr H^{n-1}(\partial^* G\cap \mathcal{C}_{\mathcal B}). 
\end{equation}
We can now prove \eqref{eq:kappa G1}.

\medskip
\noindent{\bf Case 1}: 
Suppose that
\begin{equation}\label{case 1}
    \int_{\mathbb S^{n-1}\setminus \mathcal B}  \frac{|\nabla_\tau  u|^2}{\sqrt{1+ |\nabla_\tau u|^2}} \,d\mathscr H^{n-1} + \int_{\mathbb S^{n-1}\setminus \mathcal B}|D^s_\tau u|\geq  \left(n-1+2 \zeta\right)  \int_{\mathbb S^{n-1}} {|u|^2} \,d\mathscr H^{n-1}.
\end{equation} 
Using the preceding identification of $\partial^*G$ and $\partial^*G^S$ outside
$\mathcal C_{\mathcal B}$ and outside the tangency set
$\{x\cdot\nu_G=0\}$, together with the pointwise absorption of the tangency
contribution explained above,   the $BV$ area formula for radial graphs, and the assumption
$\|u\|_{L^\infty(\mathbb S^{n-1})}\le \sz$ give
\begin{align}
 &  \int_{\partial^* G }  \left(1-\frac{(x\cdot \nu_G)^2}{|x|^2}\right) \, d\mathscr H^{n-1} 
\ge  \int_{\partial^* G^S  \setminus \mathcal{C}_{\mathcal B}}  \left(1-\frac{(x\cdot \nu_{G^S} )^2}{|x|^2}\right) \, d\mathscr H^{n-1}\nonumber \\ 
&= \int_{\mathbb S^{n-1}\setminus \mathcal B}  \left(1- \frac{1}{ 1+(1+u)^{-2}|\nabla_\tau u|^2 }\right)(1+u)^{n-1}\sqrt{1+(1+u)^{-2} |\nabla_\tau u|^2} \, d\mathscr H^{n-1} \nonumber\\
&\quad + \int_{\mathbb S^{n-1}\setminus \mathcal B}(1+u)^{n-2}\,d|D^s_\tau u|\nonumber \\
& \ge (1-C(n)\sz) \left[\int_{\mathbb S^{n-1}\setminus \mathcal B} \frac{|\nabla_\tau u|^2}{\sqrt{1+|\nabla_\tau u|^2}}   \, d\mathscr H^{n-1} + \int_{\mathbb S^{n-1}\setminus \mathcal B}|D^s_\tau u|\right], \label{gradient term F 1}
\end{align}
where we used that, as $|u|\le \sz<\frac 1 2$,
$$\left|1-\frac{\sqrt{1+(1+u)^{-2} |\nabla_\tau u|^2}}{\sqrt{1+ |\nabla_\tau u|^2}}\right|\le C \sz.$$
Thus, choosing $\sz_1=\sz_1(n,\kz)>0 $ small enough so that
$$
\frac{\bigl(1-C(n)\sz\bigr)(n-1+2\zeta)-C(n)\sigma^2}{1+C(n)\sz}\ge n-1+\zeta  \quad \text{ and } \quad 2\sz^2\le \kz,
$$
\eqref{eq:kappa G1} follows by combining \eqref{case 1}, \eqref{gradient term F 1}, and \eqref{volume term F}.

\medskip
\noindent{\bf Case 2}: 
Now we assume that
\begin{equation}\label{case 2}
    \int_{\mathbb S^{n-1} }  \frac{|\nabla_\tau  u|^2}{\sqrt{1+ |\nabla_\tau u|^2}} \,d\mathscr H^{n-1} + \int_{\mathbb S^{n-1}}|D^s_\tau u|\geq N  \int_{\mathbb S^{n-1}} {|u|^2} \,d\mathscr H^{n-1}.
\end{equation}
We note that (cp. \eqref{gradient term F 1})
\begin{multline}
  \int_{\partial^* G^S }  \left(1-\frac{(x\cdot \nu_{G^S} )^2}{|x|^2}\right) \, d\mathscr H^{n-1}\\ 
\ge (1-C(n)\sz) \left[\int_{\mathbb S^{n-1}} \frac{|\nabla_\tau u|^2}{\sqrt{1+|\nabla_\tau u|^2}}   \, d\mathscr H^{n-1} + \int_{\mathbb S^{n-1}}|D^s_\tau u|\right]. \label{gradient term F 2}
\end{multline}
Then it follows from \eqref{gradient term F 2}, \eqref{volume term F}, and \eqref{error volume term}, 
\begin{align*}
   & \int_{\partial^* G}\left(\frac{1}{|x|}+|x|-2\right)  \,  d\mathscr H^{n-1}\\
   &\leq (1+C(n) \sz)\int_{\mathbb S^{n-1}} |u|^2  \,  d\mathscr H^{n-1} +C(n)\sz^2 \left[\int_{\mathbb S^{n-1}\setminus \mathcal B} \frac{|\nabla_\tau u|^2}{\sqrt{1+|\nabla_\tau u|^2}}    \,  d\mathscr H^{n-1} + \int_{\mathbb S^{n-1}\setminus \mathcal B}|D^s_\tau u|\right]\\
   &\qquad + 2\sz^2 \mathscr H^{n-1}(\partial^* G\cap \mathcal{C}_{\mathcal B}) \\
    & \leq  \left(\frac{1+C(n)\sz}{N} + C(n)\sz^2\right)   \int_{\partial^* G^S }  \left(1-\frac{(x\cdot \nu_{G^S} )^2}{|x|^2}\right) \, d\mathscr H^{n-1} + 2\sz^2 \mathscr H^{n-1}(\partial^* G\cap \mathcal{C}_{\mathcal B}).  
\end{align*}
Moreover, Lemma~\ref{rearrange G} yields
$$\int_{\partial^* G^S\cap \mathcal{C}_{\mathcal B} }  \left(1-\frac{(x\cdot \nu_{G^S} )^2}{|x|^2}\right) \, d\mathscr H^{n-1}\le \mathscr H^{n-1}(\partial^* G^S\cap \mathcal{C}_{\mathcal B})\le (1+C(n)\sz) \mathscr H^{n-1}(\partial^* G\cap \mathcal{C}_{\mathcal B}). $$
Therefore, choosing $N>2n$ large enough and $\sz_1=\sz_1(n,\kz)>0 $ sufficiently small so that
$$
(n-1+\zeta) \left(\frac{1+C(n)\sz}{N} + C(n)\sz\right)\le 1 \quad \text{ and } \quad \frac{1+C(n)\sz}{N} +C(n)\sigma^2 +2\sigma^2 \le \kz,
$$
we again obtain \eqref{eq:kappa G1}.

\bigskip

Towards the second part,  note that
$$\frac{1}{|x|}+|x|-2 =0 \quad \text{ when } x\in \partial B.$$
Also, 
\begin{align*}
{\rm div}\left(\left(\frac{1}{|x|}+|x|-2\right) \frac{x}{|x|}\right)
= \frac{\big(2+n(|x|-1)\big)(|x|-1)}{|x|^2} \geq ||x|-1| \qquad \text{for $1<|x|<1 + \frac 1 n$}.
\end{align*}
Hence, recalling that $\frac{1}{|x|}+|x|-2\geq 0$, which is $0$ on $\partial B$, it follows from  the divergence theorem that
\begin{align*}
\int_{\partial^*  G}   \left( \frac{1}{|x|}+|x|-2\right)\, d\mathscr H^{n-1} &\ge \int_{\partial^*G \setminus B}\left(\frac{1}{|x|}+|x|-2\right)\, d\mathscr H^{n-1}\\
&=  \int_{\partial^*  G\setminus B}   \left(\frac{1}{|x|}+|x|-2\right) \left(\frac{x}{|x|}\cdot \nu_G\right)\, d\mathscr H^{n-1} \\
&\qquad +\int_{\partial^*  G \setminus B}   \left(\frac{1}{|x|}+|x|-2\right) \left(1-\frac{x}{|x|}\cdot \nu_G\right)\, d\mathscr H^{n-1}\\
& \ge \int_{\partial^*  G\setminus B}   \left(\frac{1}{|x|}+|x|-2\right) \left(\frac{x}{|x|}\cdot \nu_{G\setminus B}\right)\, d\mathscr H^{n-1}\\
&= \int_{\partial^*  (G\setminus B)}   \left(\frac{1}{|x|}+|x|-2\right) \left(\frac{x}{|x|}\cdot \nu_{G\setminus B}\right)\, d\mathscr H^{n-1} \\
& \ge  \int_{ G\setminus B }   \left| |x|-1\right|\, dx \ge c(n)|G\Delta B|^2,
\end{align*}
where the last inequality follows from   Lemma~\ref{lem:1 x}.
\end{proof}

 \section{Proof of Theorem~\ref{main thm}}\label{main proof}
 
We first reduce the proof to the case where the weighted barycenter vanishes.
Assume that the estimate has already been proved for all bounded $C^1$ open
sets $F$ satisfying
$$
|F|=|B|,\qquad \int_F\frac{x}{|x|}\,dx=0,\qquad
\operatorname{Exc}(F)\le\delta_*,
$$
where $\delta_*=\delta_*(n)>0$.  Let now $E$ satisfy the hypotheses of
Theorem~\ref{main thm} with $\operatorname{Exc}(E)\le\delta$, where $\delta$
will be chosen below. By \eqref{exc symm diff},
$|E\Delta B|\le C(n)\operatorname{Exc}(E)^{1/2}$. Since
$\int_B x/|x|\,dx=0$, this gives
$$
\left|\mean{E}\frac{x}{|x|}\,dx\right|
\le C(n)|E\Delta B|
\le C(n)\operatorname{Exc}(E)^{1/2}.
$$
Thus, after decreasing $\delta$ if necessary, Lemma~\ref{translation} gives a
vector $y_0\in\R^n$ with
$$
|y_0|\le C(n)\operatorname{Exc}(E)^{1/2},
\qquad
\int_{E+y_0}\frac{x}{|x|}\,dx=0.
$$
Moreover, for $|y_0|\ll1$,
$$
\big|\operatorname{Exc}(E+y_0)-\operatorname{Exc}(E)\big|
\le C(n)|y_0|^{\frac{n-1}{n}}.
$$
Indeed, using
$\operatorname{Exc}(F)=P(F)-\int_F (n-1)/|x|\,dx$, this follows by splitting
the integral of
$\left||x+y_0|^{-1}-|x|^{-1}\right|$ on $E$ into $B_{2|y_0|}$ and its
complement. Hence, choosing $\delta=\delta(n)>0$ small enough, the translated
set $F:=E+y_0$ satisfies $\operatorname{Exc}(F)\le\delta_*$. Applying the
normalized estimate to $F$, and using the invariance of the $L^2$-oscillation
of the mean curvature under translations, proves \eqref{eq:main thm} for the
original set with $x_0=-y_0$.

It remains to prove the normalized estimate.  In the rest of the proof we
relabel $F$ as $E$ and $\delta_*$ as $\delta$.

Let $E\subset\mathbb R^n$ be a bounded $C^1$ open set, and write
\begin{equation}\label{mean curvature E}
\mathcal H_{\partial^*E}=\mu +R,
\end{equation}
where $\mu \in \R$ and $R\in L^2(\partial^*E)$.  Recall that, by assumption, $|E|=|B|$,  
$\int_{E} \frac{x}{|x|}\,dx =0, $
and
${\rm Exc}(E)\leq \delta$.
Also, thanks to \eqref{difference perimeter}, $|E\Delta B|\leq C(n)\delta^{1/2}$.
Since bounded $C^1$ open sets have finite perimeter and $\partial E=\partial^*E$ up to $\mathscr H^{n-1}$-negligible sets, we shall keep using reduced boundary notation throughout. 

Given $0<r<\rho<\infty$, we recall the notation
$$A_{r,\rho}:=B_{\rho}\setminus \overline{B}_{r}. $$
Note that equation \eqref{mean curvature E} is equivalent, in weak form, to saying that
\begin{equation}\label{eq:mean div}
\int_{\partial^*E} {\rm div}_\tau \big(\Phi(x)\big) \,d\mathscr H^{n-1}=\mu \int_{\partial^*E}  \Phi(x)\cdot \nu_E \,d\mathscr H^{n-1} +  \int_{\partial^*E} R(x)  \Phi(x)\cdot \nu_E \,d\mathscr H^{n-1} 
\end{equation}
for every $\Phi \in C^1_c(\mathbb R^n;\mathbb R^n)$.
We shall consider test functions of the form $\Phi(x)=\phi(|x|)x$, and by approximation we can consider test functions that are Lipschitz and integrable on the domain of integration. For later use, we note that
\begin{equation}\label{eq:div}
\begin{split}
{\rm div}_\tau \big(\phi(|x|)x\big)&= {\rm div}_\tau(x) \phi(|x|) + \phi'(|x|) x\cdot \nabla_\tau |x|\\
&=(n-1)\phi(|x|)+ \phi'(|x|)|x|\left(1-\frac{(x\cdot \nu_E)^2}{|x|^2}\right).
\end{split}
\end{equation}

\subsection{Reducing to $R$ small in $L^2(\partial^* E)$}
Fix $\mu\in\mathbb R$ and write $R=\mathcal H_{\partial^*E}-\mu$.
All constants below are independent of this choice of $\mu$.
We first reduce the problem to the case where
\begin{equation}\label{R small}
    \|R\|_{L^2(\partial^*E)} \leq c_0
\end{equation} 
for some $c_0=c_0(n)>0$ sufficiently small. 

Indeed, 
choose $\Phi(x)=\varphi(|x|)x$,
where $\varphi \in C^\infty_c(\R)$ is a cut-off function supported on $\left(\frac 1 4,\, 2\right)$ that coincides with $1$ on $\left(\frac 1 2,\, \frac3 2\right)$, and satisfies
$\varphi'\ge 0$  on $\left(\frac 1 4,\, \frac 1 2\right)$.
Then \eqref{eq:mean div}, \eqref{eq:div}, and the divergence theorem, imply that
\begin{align*}
&\bigg|\mu \int_{E} \big( n\varphi(|x|) +\varphi'(|x|) |x| \big)\,dx\bigg| =\bigg|\mu \int_{E} {\rm div}\big(\Phi(x)\big) \,dx\bigg|=\bigg| \mu \int_{\partial^*E}  \Phi(x)\cdot \nu_E \,d\mathscr H^{n-1}\bigg|\\
&=\bigg| \int_{\partial^*E}  \bigg((n-1) \varphi(|x|) + \varphi'(|x|)|x| \left(1-\frac{(x\cdot \nu_E)^2}{|x|^2}\right) \bigg)\,d\mathscr H^{n-1} - \int_{\partial^*E} R(x)  \Phi(x)\cdot \nu_E \,d\mathscr H^{n-1} \bigg|\\
&\leq C(n)\left(P(E) + \|R\|_{L^1(\partial^*E)}\right)\leq C(n)\left(P(E) + P(E)^{1/2}\|R\|_{L^2(\partial^*E)}\right),
\end{align*}
where, for the last inequality, we used H\"older's inequality.
Thanks to \eqref{difference perimeter} and the smallness of the excess, it follows that $P(E)\leq C(n)$, thus
\begin{equation}\label{bounded mu integral}
\bigg|\mu \int_{E} \big( n\varphi(|x|) +\varphi'(|x|) |x| \big)\,dx\bigg| \leq C(n)\big(1+\|R\|_{L^2(\partial^*E)}\big).
\end{equation}
Also, by the construction of $\varphi$, 
$$
n\varphi(|x|) +\varphi'(|x|) |x|\geq \left\{
\begin{array}{ll}
0&\text{in }B_{1 /2}\\
n&\text{in }A_{1 /2 , 3 /2}\\
-C(n)&\text{outside }B_{3/ 2}.
\end{array}
\right.
$$
Therefore, choosing $|E\Delta B|\le \delta=\delta(n)$ sufficiently small, as $|E|=|B|$ we have
\begin{multline*}
\int_{E} \big( n\varphi(|x|) +\varphi'(|x|) |x| \big)\,dx \geq n\left|E\cap A_{1/2,3/2}\right|-C(n)|E\setminus  B_{3/2}| \\
 \geq n\left|E\cap A_{1/2,3/2}\right|-C(n)|E\Delta B| \ge \frac 1 2  \left|A_{1/2,3/2}\right|\ge c(n).
\end{multline*}
Combining this bound with \eqref{bounded mu integral}, we conclude that 
\begin{equation}\label{bounded mu}
|\mu|\le C(n) \big(1+\|R\|_{L^2(\partial^*E)}\big).
\end{equation}
Note that, when $\|R\|_{L^2(\partial^*E)} \geq c_0(n)$ for some dimensional
constant $c_0(n)$, \eqref{bounded mu} and ${\rm Exc}(E)\le\delta(n)\le1$ give
\[
{\rm Exc}(E)+|E\Delta B|^2+|\mu-(n-1)|^2
\le C(n)\bigl(1+\|R\|_{L^2(\partial^*E)}^2\bigr)
\le C(n)c_0(n)^{-2}\|R\|_{L^2(\partial^*E)}^2 .
\]
Thus the normalized estimate follows in the case
$\|R\|_{L^2(\partial^*E)} \geq c_0(n)$.

Hence, from now on, we shall assume that $\|R\|_{L^2(\partial^*E)} \leq c_0(n)$ for some small dimensional constant to be fixed later. In particular, thanks to \eqref{bounded mu}, $|\mu|$ is bounded by a dimensional constant.

\subsection{Cutting tentacles}

We now prove that, if the excess is sufficiently small, then we can control the area of potential ``tentacles'' of $E$ outside the annulus $A_{1-\sigma,1+\sigma}$ with $\|R\|_{L^2(\partial^*E)}^{\frac{2}{1-\eta}}$ for some $\eta=\eta(n)>0$.

\begin{lem}\label{small tentacle}
Let $E\subset\mathbb R^n$ be a set of finite perimeter with $|E|=|B|$ and
scalar distributional mean curvature
\[
\mathcal H_{\partial^*E}=\mu+R,\qquad R\in L^2(\partial^*E),
\]
where $|\mu|\le C_0(n)$. For every $\sz \in (0,1/8)$ there exists
$\delta=\delta(n,\sz)\ll 1$ such that, whenever ${\rm Exc}(E)\le \delta,$
then
\begin{align*}
&P(E\setminus B_{r})\le C(n)\|R\|_{L^2(\partial^*E)}^{\frac{2}{1-\eta}} \qquad \text{for all }r \geq 1+\sigma/2,\\
&P(B_r\setminus E)\le C(n)\|R\|_{L^2(\partial^*E)}^{\frac{2}{1-\eta}} \qquad \text{for all }r \in (1/2,1-\sigma/2).
\end{align*}
Here, $\eta=\eta(n) \in (0,1)$ is a dimensional constant.
\end{lem}
\begin{proof}
We first consider the part of $\partial^*E$ outside the unit ball.

Set $E_r:=E\setminus B_{r}$ and apply Lemma~\ref{small tentacle excess} with $\gamma=\sigma/4$ to ensure that, if $\delta=\delta(n,\sigma)$ is small enough, then
$P(E_r)\ll 1$ for all $r \geq 1+\sigma/4$.
Note now that, by the coarea formula (see for instance \cite[Theorem 18.8]{M2012}),
\begin{equation}
\label{eq:coarea angle}
-\frac{d}{dr}\mathscr H^{n-1}\left(\partial^*E \setminus B_{r}\right)= \int_{\partial^*E\cap \partial B_{r}}\frac{1}{\sqrt{1-\big|\frac{x}{|x|}\cdot \nu_E\big|^2}}\,d\mathscr H^{n-2} \geq  \mathscr H^{n-2}(\partial^*E\cap \partial B_{r})
\qquad \text{for a.e. }r.
\end{equation}
In addition, by Lemma~\ref{singular mean curvature}, the vector-valued
distributional curvature measure $\mathbf H_{\partial E_r}$ satisfies
$$|\mathbf H_{\partial E_r}|\le |\mathcal H_{\partial^*E}|\, \mathscr H^{n-1}\lfloor_{\partial^*E\setminus B_r}+\frac{n-1}{r}\,\mathscr H^{n-1}\lfloor_{E^{(1)}\cap\partial B_r}+ C \mathscr H^{n-2}\lfloor_{\partial^*E\cap \partial B_{r}},$$
which implies in particular that
\begin{equation}
\label{eq:H Er}
|\mathbf H_{\partial E_r}|\leq \big(|\mu|+|R|\big)\mathscr H^{n-1}\lfloor_{\partial^*E\setminus B_{r}}+\frac{n-1}{r}\mathscr H^{n-1}\lfloor_{E\cap \partial B_{r}}+C \mathscr H^{n-2}\lfloor_{\partial^*E\cap \partial B_{r}}.
\end{equation}
Note now that, by the Michael--Simon inequality \cite{MS1973} (see also \cite{A1972})\footnote{The Michael--Simon inequality states that, if $n \geq 3$, then for any smooth bounded set $F$ it holds
\begin{equation*}
P(F)^{\frac{n-2}{n-1}}\leq C(n) |\mathbf H_{\partial F}|(\R^n).
\end{equation*}
This extends to arbitrary sets of finite perimeter, where the curvature is understood in the distributional sense of \eqref{eq:vector curvature measure}, see e.g. \cite[Theorem 5.7]{S2014}  and also the proof of \cite[Theorem 1.2]{CM2022}.
When $n=2$, it follows from the Gauss--Bonnet theorem for curves that
$$
|\mathbf H_{\partial F}|(\R^n)\geq 2\pi,
$$
so the inequality $P(F)^{\alpha}\leq  |\mathbf H_{\partial F}|(\R^n)$ holds for any $\alpha>0$ provided $P(F)\leq 1$. This applies in our case to the set $F=E_r$, since $P(E_r)\ll 1$. }, there exists a dimensional constant $\eta=\eta(n)\in (0,1)$ such that
\begin{equation}\label{eq:MS}
\mathscr H^{n-1}(\partial^*E_r)^{1-\eta}\le C(n)|\mathbf H_{\partial E_r}|(\R^n).    
\end{equation}
Thus, combining \eqref{eq:MS}, \eqref{eq:H Er}, \eqref{eq:coarea angle}, and 
H\"older's inequality, 
we get
\begin{align*}
\mathscr H^{n-1}(\partial^*E_r)^{1-\eta}
&\leq    C(n)\int_{\partial^*E\setminus B_{r}}\big(|\mu|+|R|\big)\, d\mathscr H^{n-1}+  C(n)\mathscr H^{n-1}(E\cap \partial B_{r})\\
&\qquad+ C  \mathscr H^{n-2}(\partial^*E\cap \partial B_{r}) \\
&\le  C(n)|\mu|\mathscr H^{n-1}(\partial^*E\setminus B_{r})+\|R\|_{L^2(\partial^*E\setminus B_{r})}\mathscr H^{n-1}(\partial^*E\setminus B_{r})^{1/2} \\
&\qquad+ C(n)\mathscr H^{n-1}(E\cap \partial B_r)-   C   \frac{d}{dr}\mathscr H^{n-1}\left(\partial^*E \setminus B_{r}\right)\\
&\leq C(n) \mathscr H^{n-1}(\partial^*E_r)+ \frac{1}2 \|R\|_{L^2(\partial^*E)}^2 -   C   \frac{d}{dr}\mathscr H^{n-1}\left(\partial^*E \setminus B_{r}\right),
\end{align*}
where we used that $|\mu|\leq C(n)$ and that
\begin{align*}
\|R\|_{L^2(\partial^*E\setminus B_{r})}\mathscr H^{n-1}(\partial^*E\setminus B_{r})^{1/2}
&\leq \frac12 \|R\|_{L^2(\partial^*E\setminus B_{r})}^2+\frac12 \mathscr H^{n-1}(\partial^*E\setminus B_{r})\\
& \leq 
\frac12 \|R\|_{L^2(\partial^*E)}^2+\frac12
\mathscr H^{n-1}(\partial^*E_r).
\end{align*}
Since $\mathscr H^{n-1}(\partial^*E_r)=P(E_r)\ll 1$, we can absorb the first term in the right hand side to finally obtain
\begin{align*}
\mathscr H^{n-1}(\partial^*E_r)^{1-\eta}\leq \|R\|_{L^2(\partial^*E\setminus B_{r})}^2 -   C\frac{d}{dr}\mathscr H^{n-1}\left(\partial^*E \setminus B_{r}\right).
\end{align*}
In particular, if we consider the non-increasing function $F(r):=\mathscr H^{n-1}\left(\partial^*E \setminus B_{r}\right)$, we proved that, for  some $\rho\in (1+\sz/8,\,1+\sz/4)$, 
$$F(r)^{1-\eta}\le\mathscr H^{n-1}(\partial^*E_r)^{1-\eta}\le  \|R\|^2_{L^2(\partial^*E)} -  C  F'(r)
\qquad \text{for almost every } \ r \geq \rho.$$

Recalling that $\rho \leq 1+\frac\sigma{4}$, we make the following:

\smallskip

\noindent
{\bf Claim:} {\it If $Exc(E)\leq \delta$ is small enough, then
\begin{equation}\label{F}
F(r)^{1-\eta}\le 2\|R\|^2_{L^2(\partial^*E)} \qquad \text{for  almost every } \ r \geq 1+\sigma/2.
\end{equation}}
Indeed, if not, since $F$ is non-increasing we would have
$$F(r)^{1-\eta}\le \frac12 F(r)^{1-\eta} - C F'(r) \qquad \text{for  almost every } \ r\in (1+ {\sz} /4,\,1+\sz/2),$$
or equivalently
$$
\frac{d}{dr}F(r)^\eta= \eta F'(r)F(r)^{\eta-1}\le -\frac{\eta}{2C} \qquad \text{for  almost every } r \in (1+  {\sz}/ 4,\,1+\sz/2).
$$
Integrating this ODE over $(1+ \sz/ 4,\,1+\sz/2)$ one obtains
$$\frac{\eta}{8C}\sz \le  F\left(1+\frac \sz 4\right)^{\eta}- F\left(1+ \frac \sz 2 \right)^{\eta}\le F\left(1+\frac \sz 4\right)^{\eta} = \mathscr H^{n-1}\left(\partial^*E \setminus B_{1+\sigma/4}\right)^{\eta}, $$
a contradiction to \eqref{eq:bound coarea1} when $\delta>0$ is sufficiently small.

This proves \eqref{F}, which, combined with \eqref{eq:bound div1}, implies the bound
$$\mathscr H^{n-1}(\partial^*(E\setminus B_{r}))\le C(n)\|R\|_{L^2(\partial^*E)}^{\frac{2}{1-\eta}} \qquad \text{for  almost every } \ r  \geq 1+\sigma/2,
$$
concluding the estimate for the part of $\partial^*E$
outside the unit ball.

 The interior estimate is similar, with the signs reversed. Indeed, set
$\hat E_r:=B_r\setminus E$ and
$F_{\rm in}(r):=\HH^{n-1}(\partial^*E\cap B_r)$. For a.e. $r$,
\[
\frac{d}{dr}F_{\rm in}(r)
=\int_{\partial^*E\cap\partial B_r}
\frac{1}{\sqrt{1-\big|\frac{x}{|x|}\cdot\nu_E\big|^2}}\,d\HH^{n-2}
\ge \HH^{n-2}(\partial^*E\cap\partial B_r),
\]
\[
|\mathbf H_{\partial \hat E_r}|
\le |\mathcal H_{\partial^*E}|\,\HH^{n-1}\lfloor_{\partial^*E\cap B_r}
+\frac{n-1}{r}\,\HH^{n-1}\lfloor_{E^{(0)}\cap\partial B_r}
+C\HH^{n-2}\lfloor_{\partial^*E\cap\partial B_r}.
\]
Thus, Michael--Simon, H\"older, the preceding coarea identity, and
Lemma~\ref{small tentacle excess} give
\[
F_{\rm in}(r)^{1-\eta}
\le \|R\|_{L^2(\partial^*E)}^2+C F_{\rm in}'(r)
\qquad\text{for a.e. }r\in(1-\sigma/2,1-\sigma/4).
\]
If the desired bound failed at some $r\le1-\sigma/2$, the monotonicity of
$F_{\rm in}$ would give
$F_{\rm in}^{1-\eta}\le \frac12F_{\rm in}^{1-\eta}+C F_{\rm in}'$
on $(1-\sigma/2,1-\sigma/4)$. Hence
$\frac{d}{dr}F_{\rm in}(r)^\eta\ge c(n)$ there, contradicting the smallness of
$F_{\rm in}(1-\sigma/4)$ from Lemma~\ref{small tentacle excess}. 
Finally the
same divergence estimate as in Lemma~\ref{small tentacle excess}, together with
the relative isoperimetric inequality in $B_r$, controls
$\HH^{n-1}(E^{(0)}\cap\partial B_r)$ by $F_{\rm in}(r)$.
This proves
\[
P(B_r\setminus E)
\le C(n)\|R\|_{L^2(\partial^*E)}^{\frac{2}{1-\eta}}
\qquad\text{for all }r\in(1/2,1-\sigma/2),
\]
after decreasing $\delta(n,\sigma)$ if necessary.
The bounds for all radii follow by approximating an arbitrary radius by regular radii and using lower semicontinuity of perimeter under $L^1_{\rm loc}$ convergence of the truncations.
\end{proof}

\subsection{Approximation results}\label{sec:approx}

We collect here the approximation results for the original set. The first
lemma records the local second-order regularity that follows from the weak
mean-curvature equation.

\begin{lem}\label{smooth C1 approximation}
Let $E\subset \mathbb R^n$ be a bounded Lipschitz open set, and assume that its scalar distributional mean curvature satisfies
$\mathcal H_{\partial E}\in L^2(\partial E).$
Then $\partial E$ is of class $W^{2,2}$ and there exists a sequence $\{E_j\}_{j \in \mathbb N}$ of bounded open sets with smooth boundary
such that
$$
E_j\supset E,\qquad |E_j\Delta E|\to 0,\qquad P(E_j)\to P(E),
$$
and 
$$
\int_{\partial E_j} |\mathcal H_{\partial E_j}|^2\, d\mathscr H^{n-1} \to  \int_{\partial E }|\mathcal H_{\partial E}|^2\, d\mathscr H^{n-1}.
$$
\end{lem}

\begin{proof}
Since $\partial E$ is compact and Lipschitz, we may cover it by finitely many coordinate cylinders in which $\partial E$ is represented as the graph of a function
$$
u\colon B'_R\subset \mathbb R^{n-1}\to \mathbb R
$$
with $u\in W^{1,\infty}(B'_R)$. In such a chart, writing the mean curvature with respect to the graph orientation, we have
\begin{equation}\label{eq:graph mean curvature weak}
-{\rm div}\bigg(\frac{Du}{\sqrt{1+|Du|^2}}\bigg)=f
\qquad \text{in } B'_R
\end{equation}
in the weak sense, where $f$ is the mean curvature read in the chart. Since the surface measure on the graph is comparable to Lebesgue measure on $B'_R$, our assumption $\mathcal H_{\partial E}\in L^2(\partial E)$ yields $f\in L^2(B'_R)$.

Set
$$
A(p):=\frac{p}{\sqrt{1+|p|^2}},\qquad p\in \mathbb R^{n-1},
$$
so that \eqref{eq:graph mean curvature weak} is of the form $-{\rm div}(A(Du))=f$. Since $u$ is Lipschitz, $DA$ is uniformly elliptic on the range of $Du$. Thus, by \cite[Theorem~2.1]{CiMa2018} we obtain
$$
A(Du)\in W^{1,2}_\loc(B'_R;\mathbb R^{n-1}).
$$
Since $A^{-1}: B_1\to \R^{n-1}$ is a diffeomorphism and $A(Du)$ is contained in a compact subset of $B_1$, it follows that
$$
Du=A^{-1}(A(Du))\in W^{1,2}_\loc(B'_R).
$$
Thus $\partial E$ is a $W^{2,2}$-boundary in the sense of \cite[Theorem~2(3)]{Antonini2024}. Applying that result, we obtain smooth inner and outer approximating domains $E^-_j\subset\subset E \subset\subset E^+_j$ such that, in each fixed chart,
$$
u^\pm_j\to u \qquad \text{in }W^{2,2}_\loc(B'_{R}).
$$
Moreover, the approximating charts remain uniformly Lipschitz, and therefore, up to a subsequence,
$$
Du^\pm_j\to Du \qquad \text{a.e. on }B'_{R}.
$$
Setting $E_j:=E_j^+$, we obtain
$$
|E_j\Delta E|\to 0,\qquad P(E_j)\to P(E).
$$
Let $a_{kl}(p):=\partial_{p_l}A_k(p)$. In each chart, with the convention of
\eqref{eq:graph mean curvature weak}, the mean curvature of the graph of $u$ is given by
$$
\mathcal H(u)= -{\rm div}(A(Du))=-a_{kl}(Du)\,\partial_{kl}u,
$$
and similarly for $u^\pm_j$. Since $u^\pm_j\to u$ in $W^{2,2}$ on every smaller ball, we have
\begin{multline*}
\|\mathcal H(u^\pm_j)-\mathcal H(u)\|_{L^2(B'_{R-\epsilon})}
\le \ \|a_{kl}(Du^\pm_j)(\partial_{kl}u^\pm_j-\partial_{kl}u)\|_{L^2(B'_{R-\epsilon})}\\
\ + \|(a_{kl}(Du^\pm_j)-a_{kl}(Du))\partial_{kl}u\|_{L^2(B'_{R-\epsilon})}
\to \ 0,
\end{multline*}
because the coefficients $a_{kl}$ are bounded on bounded sets, $D^2u^\pm_j\to D^2u$ in $L^2$, and $Du^\pm_j\to Du$ a.e. with a uniform $L^\infty$-bound. Since also
$$
\sqrt{1+|Du^\pm_j|^2}\to \sqrt{1+|Du|^2}
$$
pointwise a.e. and with a uniform $L^\infty$-bound, the area formula yields the convergence of the $L^2$-norms on each chart. Summing over the finite atlas gives the $L^2$-convergence. 
\end{proof}

We shall use the following polyhedral approximation result.
We use the notation introduced in \eqref{eq:vector curvature measure}; in
particular, for polyhedral boundaries the measure $\mathbf H_{\partial F}$
contains the codimension-two contributions.

\begin{cor}\label{cor:smooth polyhedral approximation}
Let $E\subset \R^n$ be a bounded open set of $C^1$-class and assume that
$\mathcal H_{\partial E}\in L^2(\partial E)$. Then there exists a sequence of
bounded open polyhedral sets $\{E_h\}_{h\in\mathbb N}$ such that
\[
E_h\to E\quad\text{in measure},\qquad P(E_h)\to P(E),
\]
and
\[
|\mathbf H_{\partial E_h}|(\R^n)\to
\int_{\partial E} |\mathcal H_{\partial E}|\,d\mathscr H^{n-1},
\]
where  $|\mathbf H_{\partial E_h}|$ denotes the total variation of the vector-valued measure $\mathbf H_{\partial E_h}$.
Moreover,
\[
\mathbf H_{\partial E_h}\stackrel{*}{\rightharpoonup}
\mathcal H_{\partial E}\nu_E\,\mathscr H^{n-1}\llcorner\partial E
\]
as vector-valued Radon measures. In addition, away from the
$(n-2)$-skeletons of the polyhedral boundaries, the corresponding graphs and
normals converge uniformly to those of $\partial E$.
\end{cor}

\begin{proof}
By the local regularity proved in Lemma~\ref{smooth C1 approximation}, the $C^1$ graph functions representing $\partial E$ belong to $W^{2,2}_{\loc}$. We fix a finite atlas of coordinate cylinders and then shrink the cylinders so that the smaller cylinders still cover $\partial E$ and the $W^{2,2}$ regularity holds up to the boundary of each cylinder used below. Approximating the graph functions on the larger cylinders by smooth functions converging both in $W^{2,2}$ on the smaller cylinders and in $C^1$, and gluing the local normal displacements by a partition of unity, gives smooth domains $\widetilde E_j$ with
$$
|\widetilde E_j\Delta E|\to 0,\qquad P(\widetilde E_j)\to P(E),
$$
and
$$
 \mathcal H_{\partial \widetilde E_j}\nu_{\widetilde E_j}\,\mathscr H^{n-1}\llcorner\partial \widetilde E_j
 \to
 \mathcal H_{\partial E}\nu_E\,\mathscr H^{n-1}\llcorner\partial E
$$
as vector-valued Radon measures, with convergence of the total variations. 
Then, for each fixed smooth $\widetilde E_j$, using \cite{BGM2010} (see also \cite{Fu1993,CSM2006}) we can find polyhedral domains $P_{j,h}$ such that
$$
|P_{j,h}\Delta \widetilde E_j|\to 0,\qquad P(P_{j,h})\to P(\widetilde E_j),
$$
$$
\mathbf H_{\partial P_{j,h}}\stackrel{*}{\rightharpoonup}
\mathcal H_{\partial \widetilde E_j}\nu_{\widetilde E_j}\,\mathscr H^{n-1}\llcorner\partial \widetilde E_j,
\qquad
|\mathbf H_{\partial P_{j,h}}|(\R^n)\to
\int_{\partial \widetilde E_j}|\mathcal H_{\partial \widetilde E_j}|\,d\mathscr H^{n-1},
$$
and the normals are close to those of $\partial \widetilde E_j.$
Choosing $h=h(j)$ sufficiently large gives a diagonal sequence; after
relabeling it as $E_h$, the corollary follows.
\end{proof}

\subsection{The fixed truncation and leveling}\label{sec:truncate}

Throughout this subsection, we say that $F$ is a \emph{truncation} if there
exist a bounded $C^1$ open set $E$ with
$\mathcal H_{\partial E}\in L^2(\partial E)$, a point $x_0\in\R^n$, and radii
$0<\rho_0<\rho_1$ such that
\[
F=(E\cap B_{\rho_1}(x_0))\cup B_{\rho_0}(x_0),
\]
and
\[
\HH^{n-2}\bigl(\partial E\cap\partial B_{\rho_0}(x_0)\bigr)
+\HH^{n-2}\bigl(\partial E\cap\partial B_{\rho_1}(x_0)\bigr)<\infty .
\]
For such a set we denote by $\mathcal S_F$ the union of the two interfaces
above. In this case, Lemma~\ref{singular mean curvature} implies that the measure $\mathbf H_{\partial F}$ defined by
\eqref{eq:vector curvature measure} satisfies
\[
\mathbf H_{\partial F}
=\mathcal H_F\nu_F\,\HH^{n-1}\llcorner\partial^*F+\mathbf H^s_{\partial F},
\]
where the singular part $\mathbf H^s_{\partial F}$ of $\mathbf H_{\partial F}$ is concentrated on
$\mathcal S_F$.

We now isolate the approximation property required of the two cutting radii.
Write $d_{x_0}(x):=|x-x_0|$. A truncation
\[
F=(E\cap B_{\rho_1}(x_0))\cup B_{\rho_0}(x_0)
\]
is called \emph{admissible} if there are smooth domains $\widetilde E_j$
approximating $E$ as in the proof of
Corollary~\ref{cor:smooth polyhedral approximation}, and regular values
$\rho_{0,j}\to\rho_0$, $\rho_{1,j}\to\rho_1$ of $d_{x_0}$ on
$\partial\widetilde E_j$, such that $\rho_{0,j}<\rho_{1,j}$ and, setting
\[
F_j=(\widetilde E_j\cap B_{\rho_{1,j}}(x_0))\cup B_{\rho_{0,j}}(x_0),
\]
one has $|F_j\Delta F|\to 0$, $P(F_j)\to P(F)$, and
the vector-valued curvature measures and their total variations converge
weakly:
\[
\mathbf H_{\partial F_j}\stackrel{*}{\rightharpoonup}\mathbf H_{\partial F},
\qquad
|\mathbf H_{\partial F_j}|\stackrel{*}{\rightharpoonup}
|\mathbf H_{\partial F}|
\]
as Radon measures on $\R^n$.
We call such a sequence $\{F_j\}$ a good smooth cutting approximation of $F$.

The next lemma shows that almost every pair of radii gives an admissible
truncation.

\begin{lem}\label{lem:good-cutting-levels}
Let $E\subset\R^n$ be a bounded $C^1$ open set with
$\mathcal H_{\partial E}\in L^2(\partial E)$, let $x_0\in\R^n$, and let
$\widetilde E_j$ be a smooth approximation of $E$ as in the proof of
Corollary~\ref{cor:smooth polyhedral approximation}. Then, for a.e. pair
$0<\rho_0<\rho_1$, the truncation
\[
F=(E\cap B_{\rho_1}(x_0))\cup B_{\rho_0}(x_0)
\]
is admissible.
\end{lem}

\begin{proof}
We first compute the density of the curvature measure carried by a single
spherical interface. Let
$M$ be a smooth hypersurface with unit normal $\nu$, and let $t$ be a regular
value of $d_{x_0}\llcorner M$. On $M\cap\partial B_t(x_0)$ set
\[
\omega:=\frac{x-x_0}{|x-x_0|},\qquad a:=\nu\cdot\omega,\qquad
s:=\sqrt{1-a^2}.
\]
For the cut $E\cap B_t(x_0)$, the exterior conormal of the $M$-piece is
\[
\eta_M=\frac{\omega-a\nu}{s},
\]
while the exterior conormal of the spherical piece is
\[
\eta_B=\frac{\nu-a\omega}{s}.
\]
Thus the total variation density of this interface contribution is
\begin{equation}\label{eq:cut-interface-density}
q(\nu,\omega):=|\eta_M+\eta_B|
=\left|\frac{\omega-a\nu}{s}+\frac{\nu-a\omega}{s}\right|
=\sqrt{2(1-\nu\cdot\omega)} .
\end{equation}
For the cut $E\cup B_t(x_0)$ both conormals are reversed, so the interface
measure changes sign but its total variation density is again
\eqref{eq:cut-interface-density}. We therefore use the continuous extension
$q(\nu,\omega)=\sqrt{2(1-\nu\cdot\omega)}$ for all pairs of unit vectors. Thus,
if
\[
m_j(t):=\int_{\partial\widetilde E_j\cap\partial B_t(x_0)}
q(\nu_{\widetilde E_j},\omega)\,d\HH^{n-2},
\]
then $m_j(t)$ is the total variation of the curvature measure carried by the
interface of the approximating cut at radius $t$. The analogous function $m(t)$
for $\partial E$ gives the total variation of the corresponding interface
contribution in the limiting cut for a.e. $t$.

We claim that the measures $m_j(t)\,dt$ converge weakly to $m(t)\,dt$ on
compact intervals of radii. Indeed, by the coarea formula on
$\partial\widetilde E_j$, for every $\varphi\in C_c((0,\infty))$,
\[
\int\varphi(t)m_j(t)\,dt
=\int_{\partial\widetilde E_j}
\varphi(d_{x_0}(x))q(\nu_{\widetilde E_j},\omega)
|\nabla_{\tau_j}d_{x_0}|\,d\HH^{n-1}.
\]
The right-hand side converges to the corresponding integral on $\partial E$,
because the approximating hypersurfaces converge in $C^1$, with uniformly
converging normals, on the finite atlas used in
Corollary~\ref{cor:smooth polyhedral approximation}.
Applying the coarea formula on $\partial E$ gives the claimed weak convergence
of the slice measures.

Analogously, the mean curvature measures of the spherical pieces are
controlled by the same slicing argument. On compact intervals of radii contained
in $(0,\infty)$ set
\[
a_{1,j}(t):=\frac{n-1}{t}\HH^{n-1}
\bigl(\partial B_t(x_0)\cap\widetilde E_j\bigr),
\qquad
a_{0,j}(t):=\frac{n-1}{t}\HH^{n-1}
\bigl(\partial B_t(x_0)\setminus\widetilde E_j\bigr),
\]
and define $a_1,a_0$ analogously with $E$ in place of
$\widetilde E_j$. These are precisely the total variations of the spherical
curvature contributions for the outer and inner cuts. Since
$\chi_{\widetilde E_j}\to\chi_E$ in $L^1_{\rm loc}$, the coarea formula in
$\R^n$ gives
\[
a_{i,j}(t)\,dt\stackrel{*}{\rightharpoonup}a_i(t)\,dt,\qquad i=0,1,
\]
on every such compact interval.

Set
\[
\lambda_j:=|\mathcal H_{\partial\widetilde E_j}|\,
\HH^{n-1}\llcorner\partial\widetilde E_j,\qquad
\lambda:=|\mathcal H_{\partial E}|\,
\HH^{n-1}\llcorner\partial E .
\]
By Corollary~\ref{cor:smooth polyhedral approximation},
$\lambda_j\stackrel{*}{\rightharpoonup}\lambda$. Now, let
$0<\rho_0<\rho_1$ be such that $\rho_i$ is a Lebesgue point of both $m$ and $a_i$, $i=0,1$, 
 and such that
\[
\lambda(\partial B_{\rho_0}(x_0))=\lambda(\partial B_{\rho_1}(x_0))=0.
\]
This holds for a.e. pair $0<\rho_0<\rho_1$. Choose shrinking intervals
$I_{0,h}\downarrow\{\rho_0\}$ and $I_{1,h}\downarrow\{\rho_1\}$ with
$\sup I_{0,h}<\inf I_{1,h}$. Since
$(m_j+a_{i,j})(t)\,dt$ converges weakly to $(m+a_i)(t)\,dt$, and since
$\rho_i$ is a Lebesgue point of $m+a_i$, a diagonal argument gives indices
$h(j)\to\infty$ such that, setting $I_{i,j}=I_{i,h(j)}$,
\[
\frac1{|I_{i,j}|}\int_{I_{i,j}}(m_j(t)+a_{i,j}(t))\,dt
\to m(\rho_i)+a_i(\rho_i).
\]
Since the regular values of $d_{x_0}\llcorner\partial\widetilde E_j$ have full
measure, we may then choose regular values $\rho_{i,j}\in I_{i,j}$ with
\[
m_j(\rho_{i,j})+a_{i,j}(\rho_{i,j})
\le \frac1{|I_{i,j}|}\int_{I_{i,j}}(m_j(t)+a_{i,j}(t))\,dt+\frac1j,
\qquad i=0,1 .
\]
In this way $\rho_{i,j}\to\rho_i$, $\rho_{0,j}<\rho_{1,j}$ for all large $j$,
and
\begin{equation}
    \label{eq:limsup cut}
    \limsup_{j\to\infty}\sum_{i=0}^1
\bigl(m_j(\rho_{i,j})+a_{i,j}(\rho_{i,j})\bigr)
\le \sum_{i=0}^1\bigl(m(\rho_i)+a_i(\rho_i)\bigr).
\end{equation}
Now set
\[
F_j=(\widetilde E_j\cap B_{\rho_{1,j}}(x_0))\cup B_{\rho_{0,j}}(x_0),
\qquad
F=(E\cap B_{\rho_1}(x_0))\cup B_{\rho_0}(x_0),
\]
\[
K_j:=\overline B_{\rho_{1,j}}(x_0)\setminus B_{\rho_{0,j}}(x_0),
\qquad
K:=\overline B_{\rho_1}(x_0)\setminus B_{\rho_0}(x_0),
\]
and, for $\varepsilon>0$ small enough,
\[
K_\varepsilon:=\overline B_{\rho_1+\varepsilon}(x_0)
\setminus B_{\rho_0-\varepsilon}(x_0).
\]
For every such $\varepsilon$ and all $j$ large enough,
$K_j\subset K_\varepsilon$. Hence, by the weak
convergence $\lambda_j\stackrel{*}{\rightharpoonup}\lambda$,
\[
\limsup_{j\to\infty}\lambda_j(K_j)
\le
\limsup_{j\to\infty}\lambda_j(K_\varepsilon)
\le \lambda(K_\varepsilon).
\]
Letting $\varepsilon\downarrow0$ and using
$\lambda(\partial B_{\rho_i}(x_0))=0$, $i=0,1$, gives
\[
\limsup_{j\to\infty}\lambda_j(K_j)\le \lambda(K).
\]
Combining this with \eqref{eq:limsup cut}, we proved
\begin{equation}
    \label{eq:limsup}
\limsup_{j\to\infty}|\mathbf H_{\partial F_j}|(\R^n)
\le |\mathbf H_{\partial F}|(\R^n).
\end{equation}
On the other hand, since
$\partial F_j\to\partial F$ as varifolds, 
\eqref{eq:vector curvature measure} gives
$\mathbf H_{\partial F_j}\stackrel{*}{\rightharpoonup}\mathbf H_{\partial F}$.
The opposite inequality follows from lower semicontinuity of total variation
under weak-$*$ convergence, combined with \eqref{eq:limsup}. Hence the total
masses converge, which implies that
\[
|\mathbf H_{\partial F_j}|\stackrel{*}{\rightharpoonup}
|\mathbf H_{\partial F}|,
\]
as desired.
\end{proof}

With the good pairs of cutting levels available, we now choose the two cuts
used in the main proof. The small-tentacle estimate supplies annuli where the
interfaces have small $(n-2)$-measure, and Lemma~\ref{lem:good-cutting-levels}
lets us choose the corresponding radii as a good pair.

\begin{lem}\label{lem:good truncation}
There exists $\sz_2=\sz_2(n)>0$ such that the following holds.  Let
$0<\sz<\sz_2$, $|x_0|\le\sz/8$, and let $E$ be a $C^1$-domain satisfying the
hypotheses of Lemma~\ref{small tentacle}. Then one can find
$s^\pm\in[0,\sz/8]$ so that, setting
$$
G:=(E\cap B_{1+\sz/2+s^+}(x_0))\cup B_{1-\sz/2-s^-}(x_0),
$$
one has $\partial G\subset A_{1-2\sz,1+2\sz}$ and
\begin{equation}\label{eq:singular-mean}
|\mathbf H^s_{\partial G}|(\partial G)
\le C(n,\sz)\|R\|_{L^2(\partial^*E)}^{\frac{2}{1-\eta}} .
\end{equation}
In addition, $G$ is an admissible truncation.
\end{lem}

\begin{proof}
The inclusion $\partial G\subset A_{1-2\sz,1+2\sz}$ follows from
$|x_0|\le\sz/8$ and the definitions of the two cutting radii.  We choose the
radii by coarea. Since $|\nabla_\tau |x-x_0||\le1$ on $\partial E$ and the
translation by $x_0$ is absorbed by the margins in the annuli, Lemma~\ref{small tentacle}
applied with parameter $3\sigma/4$ gives
\begin{align*}
\int_0^{\sz/8}\HH^{n-2}\bigl(\partial E\cap\partial B_{1+\sz/2+s}(x_0)\bigr)\,ds
&\le C(n)\HH^{n-1}\bigl(\partial E\setminus B_{1+3\sz/8}\bigr)\\
&\le C(n,\sz)\|R\|_{L^2(\partial^*E)}^{\frac{2}{1-\eta}},
\end{align*}
and, analogously,
\begin{align*}
\int_0^{\sz/8}\HH^{n-2}\bigl(\partial E\cap\partial B_{1-\sz/2-s}(x_0)\bigr)\,ds
&\le C(n)\HH^{n-1}\bigl(\partial E\cap B_{1-3\sz/8}\bigr)\\
&\le C(n,\sz)\|R\|_{L^2(\partial^*E)}^{\frac{2}{1-\eta}} .
\end{align*}
Since the good pairs of Lemma~\ref{lem:good-cutting-levels} have full measure
in the square $[0,\sz/8]^2$, removing the bad pairs does not affect the
average of the sum of the two interface terms over that square. Thus we may
choose $s^\pm\in[0,\sz/8]$ so that the two radii below form a good pair and
\begin{equation}\label{eq:good-levels}
\HH^{n-2}\bigl(\partial E\cap\partial B_{1+\sz/2+s^+}(x_0)\bigr)
+
\HH^{n-2}\bigl(\partial E\cap\partial B_{1-\sz/2-s^-}(x_0)\bigr)\le C(n,\sz)\|R\|_{L^2(\partial^*E)}^{\frac{2}{1-\eta}} .
\end{equation}
Let
\[
\rho_0=1-\sz/2-s^-,
\qquad
\rho_1=1+\sz/2+s^+ .
\]
By Lemma~\ref{lem:good-cutting-levels}, the corresponding truncation $G$ is
admissible.
Also, by the proof of
Lemma~\ref{singular mean curvature} with the radial cutoff centered at $x_0$,
\[
|\mathbf H^s_{\partial G}|(\partial G)
\le C(n)\HH^{n-2}(\partial E\cap \partial B_{1+\sz/2+s^+}(x_0))
+C(n)\HH^{n-2}(\partial E\cap \partial B_{1-\sz/2-s^-}(x_0)).
\]
Combining this with
\eqref{eq:good-levels} gives \eqref{eq:singular-mean}.
\end{proof}

The fixed-truncation estimate will first be proved for polyhedra and then
passed to admissible truncations. We therefore record that the good smooth
cutting approximations in the definition of admissibility can be replaced by
polyhedral ones without losing the estimates needed later.

\begin{cor}\label{cor:piecewise polyhedral approximation}
Let $F$ be an admissible truncation. Then its good smooth cutting
approximations can be replaced by bounded open polyhedral sets $F_h$ such that
\[
F_h\to F\quad\text{in measure},\qquad P(F_h)\to P(F),
\]
and, away from the $(n-2)$-skeletons of $\partial F_h$, the corresponding graphs
and facet normals converge uniformly on compact subsets of
$\partial^*F\setminus\mathcal S_F$. Moreover, the curvature measures converge
strictly away from the cutting interfaces: if
$U\supset\mathcal S_F$ is open with
$(|\mathcal H_F|\HH^{n-1}\llcorner\partial^*F)(\partial U)=0$,\footnote{Since $U\supset\mathcal S_F$,
the condition $(|\mathcal H_F|\HH^{n-1}\llcorner\partial^*F)(\partial U)=0$ is equivalent to $|\mathbf H_{\partial F}|(\partial U)=0$.} then
\[
\mathbf H_{\partial F_h}\llcorner(\R^n\setminus U)
\stackrel{*}{\rightharpoonup}
\mathcal H_F\nu_F\,\HH^{n-1}\llcorner(\partial^*F\setminus U),
\qquad
|\mathbf H_{\partial F_h}|(\R^n\setminus U)
\to
\int_{\partial^*F\setminus U}|\mathcal H_F|\,d\HH^{n-1}.
\]
Finally, if $\mathcal K_h^{\rm cut}$ denotes the union of the polyhedral
codimension-two faces approximating the two cutting interfaces, then
\[
\limsup_{h\to\infty}|\mathbf H^s_{\partial F_h}|(\mathcal K_h^{\rm cut})
\le |\mathbf H^s_{\partial F}|(\partial F).
\]
\end{cor}

\begin{proof}
Let $G_j$ be a good smooth cutting approximation of $F$, and denote by
$\mathcal S_j$ the union of its two cutting interfaces. Then $\partial G_j$ is
the union of finitely many smooth pieces and two spherical pieces, meeting
transversely along compact smooth $(n-2)$-manifolds. By construction,
$G_j\to F$ in measure, $P(G_j)\to P(F)$, and the graphs and normals converge
uniformly on compact subsets of $\partial^*F\setminus\mathcal S_F$. Moreover,
by admissibility,
\[
\mathbf H_{\partial G_j}\stackrel{*}{\rightharpoonup}\mathbf H_{\partial F},
\qquad
|\mathbf H_{\partial G_j}|\stackrel{*}{\rightharpoonup}
|\mathbf H_{\partial F}|.
\]

We now approximate each fixed $G_j$ by polyhedra. Since the intersections are
transverse, we first approximate the $(n-2)$-dimensional interfaces
$\mathcal S_j$ by a polyhedral complex $\mathcal K_{j,k}^{\rm cut}$. We then
approximate each adjacent smooth or spherical piece independently, with this
complex prescribed as boundary. Applying the same polyhedral approximation
result used in Corollary~\ref{cor:smooth polyhedral approximation}, namely
\cite{BGM2010} (see also \cite{Fu1993,CSM2006}), on each piece gives bounded
open polyhedral sets $P_{j,k}$ which converge to $G_j$ in the same sense as in
Corollary~\ref{cor:smooth polyhedral approximation}, with the traces agreeing
on $\mathcal K_{j,k}^{\rm cut}$. The pieces glue without producing extra
boundary, because their traces agree and the orientations are induced by $G_j$.
Choosing $k(j)$ by a diagonal argument and setting
\[
F_h:=P_{h,k(h)},\qquad
\mathcal K_h^{\rm cut}:=\mathcal K_{h,k(h)}^{\rm cut}, \qquad h \in \mathbb N,
\]
gives the desired sequence. In particular, the convergence of the
codimension-two curvature measures gives the stated limsup bound on
$|\mathbf H^s_{\partial F_h}|(\mathcal K_h^{\rm cut})$.
Finally, the asserted strict convergence away from $\mathcal S_F$ follows by
restricting the weak convergences above to $\R^n\setminus U$, since the
hypotheses on $U$ imply $|\mathbf H_{\partial F}|(\partial U)=0$.
\end{proof}

In the rest of the subsection, $\Pi(x)=x/|x|$ denotes the radial projection.
For an admissible truncation $F$, set
\[
\Sigma_F:=\{x\in\partial^*F:\ x\cdot\nu_F(x)=0\}\cup\mathcal S_F .
\]
For a polyhedron $F$, let $\mathcal K_F$ denote the union of the
$(n-2)$-faces of $\partial F$, and set
\[
\Sigma_F:=\{x\in\partial^*F:\ x\cdot\nu_F(x)=0\}\cup\mathcal K_F .
\]
For either class of sets, recalling Lemma~\ref{small oscillation}, set
\[
 \mathcal B_F=\Big\{\theta\in \mathbb S^{n-1}:
 \#\big((\partial^*F\setminus\Sigma_F)\cap\{t\theta:t>0\}\big)\ge 2\Big\}.
\]

The radial projection is cleanest when the tangencies and interfaces project to
a negligible cone. The next lemma gives this after an arbitrarily small generic
translation, without changing the curvature quantities that enter the final
estimate.

\begin{lem}\label{lem:generic-sigma}
Let $F$ be an admissible truncation. Then, for a.e. vector
$a\in\R^n$, the set of radial tangencies
\[
        T_a:=\{x\in\partial^*(F-a):\ x\cdot\nu_{F-a}(x)=0\}
\]
satisfies $\HH^{n-1}(T_a)=0$, the set
$\Sigma_{F-a}=T_a\cup\mathcal S_{F-a}$ is closed, and
\begin{equation}\label{eq:exceptional-cone-zero}
 \HH^{n-1}\big(\partial^* (F-a)\cap \Pi^{-1}(\Pi(\Sigma_{F-a}))\big)=0 .
\end{equation}
\end{lem}

\begin{proof}
Indeed, for every $M>0$, Fubini's theorem gives
\[
\int_{B_M}\HH^{n-1}(T_a)\,da
=\int_{\partial^*F}
\mathscr L^n\bigl(\{a\in B_M:\ (x-a)\cdot\nu_F(x)=0\}\bigr)
\,d\HH^{n-1}(x)=0,
\]
because the set inside the parenthesis is contained in an affine hyperplane.
Letting $M\to\infty$ gives $\HH^{n-1}(T_a)=0$ for a.e. $a\in\R^n$.
The closedness of $\Sigma_{F-a}$ follows from the $C^1$ and spherical pieces
and from the inclusion of the interface set $\mathcal S_{F-a}$.

It remains to prove \eqref{eq:exceptional-cone-zero}. Since $\partial^*(F-a)$ is
away from the origin for a.e. $a$, it lies in some annulus
$A_{r(a),R(a)}$. Also, $\mathcal S_{F-a}$ has finite $\HH^{n-2}$-measure, and,
because $\Pi$ is Lipschitz on the annulus containing $\partial(F-a)$,
\[
\HH^{n-1}\big(\Pi(\mathcal S_{F-a})\big)=0.
\]
Applying Lemma~\ref{lem:projection exceptional cone} to $G=F-a$ with
$Z=\mathcal S_{F-a}$ gives
\[
\HH^{n-1}\big((\partial^*(F-a)\setminus T_a)\cap
\Pi^{-1}(\Pi(\Sigma_{F-a}))\big)=0.
\]
Since $\HH^{n-1}(T_a)=0$, \eqref{eq:exceptional-cone-zero} follows.
\end{proof}

Once the exceptional cone is negligible, each direction in $\mathcal B_F$
contains alternating radial intervals inside and outside the set. The next
lemma packages the resulting finite-perimeter layers and the multiplicity
information needed in the leveling argument.

\begin{lem}\label{lem:radial-layer-package}
Let $0<\sz<1/4$, and let $F$ be either an admissible truncation or a polyhedron,
with
$B_{1-2\sz}\subset F\subset B_{1+2\sz}$. Assume
\[
\HH^{n-1}\bigl(\partial^*F\cap\Pi^{-1}(\Pi(\Sigma_F))\bigr)=0.
\]
For a.e. $\theta\in\mathcal B_F$, write
\[
F_\theta=(0,a_0(\theta))\cup
\bigcup_{\ell=1}^{k(\theta)}(b_\ell(\theta),a_\ell(\theta))
\]
and set $r_0=a_0$, $r_{2\ell-1}=b_\ell$, $r_{2\ell}=a_\ell$,
$D_m=\{\theta:2k(\theta)\ge m\}$, and
$$L_m=\{t\theta:\theta\in D_m,\ r_{m-1}(\theta)<t<r_m(\theta)\}.$$
In the annulus, set
\[
\partial_F L_m:=\partial^*L_m\cap\partial^*F,
\qquad
\partial_{\rm rad}L_m:=\partial^*L_m\setminus\partial^*F.
\]
Then each $L_m$ has finite perimeter,
\[
\mathbf 1_{(\partial^*F\cap\mathcal C_{\mathcal B_F})\setminus\Sigma_F}
\le
\sum_m\mathbf 1_{(\partial^*L_m\cap\partial^*F)\setminus\Sigma_F}
\le
2\,\mathbf 1_{(\partial^*F\cap\mathcal C_{\mathcal B_F})\setminus\Sigma_F},
\]
and $\partial_{\rm rad}L_m$ is contained in $\Pi^{-1}(\Pi(\Sigma_F))$ up to
$\HH^{n-1}$-null sets.
\end{lem}

\begin{proof}
The finite-perimeter slicing theorem in angular coordinates
\cite[Theorems B, D, E and G]{CCF2005} gives the one-dimensional
decomposition, the $BV$ regularity of the endpoint functions $r_m$, and the
finite perimeter of the layers. Since $\nu_{L_m}=\pm\nu_F$ on
$\partial_F L_m$, while $x\cdot\nu_{L_m}=0$ on $\partial_{\rm rad}L_m$, the
multiplicity estimate follows. Moreover, the radial boundary can only be
created over directions in $\Pi(\Sigma_F)$, giving the inclusion
$\partial_{\rm rad}L_m\subset\Pi^{-1}(\Pi(\Sigma_F))$ up to
$\HH^{n-1}$-null sets.
\end{proof}

Applying the curvature identity to one radial layer leaves a boundary term on
the interface where the radial sides meet $\partial F$. In the polyhedral case
this boundary term is controlled by the curvature measure carried by the
codimension-two faces.
We shall use the next estimate only for polyhedra with no radial faces, namely
such that the affine hyperplane containing each $(n-1)$-face does not contain
the origin. This condition can be achieved by an arbitrarily small perturbation of
the vertices.

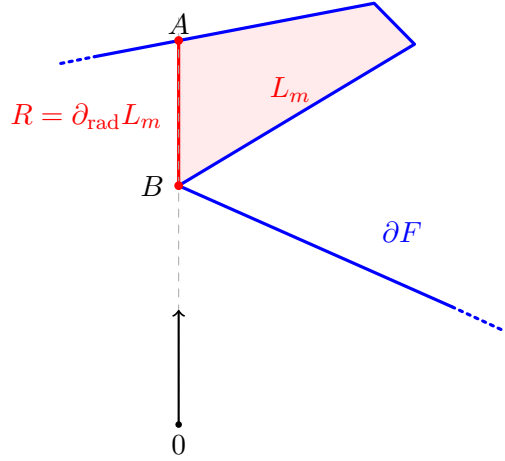
\begin{figure}[h]
\centering
\begin{tikzpicture}[scale=.82, line join=round, line cap=round]
\coordinate (O) at (1.55,-3.05);
\coordinate (A) at (1.55,3.14);
\coordinate (B) at (1.55,.80);
\coordinate (P0) at (.20,2.88);
\coordinate (P1) at (4.70,3.74);
\coordinate (P2) at (5.35,3.08);
\coordinate (P3) at (5.95,-1.15);

\fill[red!8] (A) -- (P1) -- (P2) -- (B) -- cycle;
\draw[blue, very thick, dotted] (-.35,2.77) -- (P0);
\draw[blue, very thick] (P0) -- (P1) -- (P2) -- (B) -- (P3);
\draw[blue, very thick, dotted] (P3) -- (6.75,-1.52);
\draw[red, very thick] (B) -- (A);
\draw[gray!65, dashed] (O) -- (1.55,3.55);
\draw[->, thick] (O) -- (1.55,-1.20);

\fill (O) circle (1.5pt) node[below] {$0$};
\fill[red] (A) circle (2pt);
\fill[red] (B) circle (2pt);

\node[blue] at (5.15,.05) {$\partial F$};
\node[red] at (3.35,2.35) {$L_m$};
\node[red, left] at (1.45,1.92) {$R=\partial_{\rm rad}L_m$};
\node[above] at (1.55,3.08) {$A$};
\node[left] at (1.48,.80) {$B$};
\end{tikzpicture}
\caption{Planar picture for the boundary term
$|x|^{-n}x\cdot\nu^\tau_{\partial F\cap\partial H}$ in Lemma~\ref{lem:polyhedral-edge}. The radial side $R$ of
the layer $L_m$ may meet $\partial F$ at a point $A$ in the interior of a side
of the polygon that is not a curvature point. The term at $A$ is bounded by $\rho_A^{-1}$, while
the term at the corner point $B$ is bounded by $-\rho_B^{-1}$ plus the corner
angle. Since $\rho_A\ge\rho_B$, the non-curvature part is non-positive, while
the corner error is controlled by the curvature of $\partial F$ at $B$.}
\end{figure}

\begin{lem}\label{lem:polyhedral-edge}
Let $F$ be a polyhedron with no radial faces and with
$\partial F\subset A_{1/2,2}$, and let $H=L_m$ be one of the sets constructed in
Lemma~\ref{lem:radial-layer-package}.
Set
\[
\partial_{\rm rad}H:=\partial^*H\setminus\partial^*F,
\qquad
\Gamma_H:=\partial F\cap\partial H\cap
\overline{\partial_{\rm rad}H},
\]
and let $\mathcal E(F)$ be the family of $(n-2)$-faces of $\partial F$.
Then
\[
\int_{\Gamma_H}|x|^{-n}x\cdot\nu^\tau_{\partial F\cap\partial H}\,
d\HH^{n-2}
\le
C(n)|\mathbf H_{\partial F}|\bigg(\Gamma_H\cap
\bigcup_{S\in\mathcal E(F)}S\bigg).
\]
\end{lem}

\begin{proof}
Set $X(x)=|x|^{-n}x$. We argue on one regular piece $R$ of
$\partial_{\rm rad}H$, and write $M=\Pi(R)$. Up to negligible sets, $R$ is
parametrized by
\[
        x=t\theta,\qquad \theta\in M,\qquad
        \rho_-(\theta)<t<\rho_+(\theta).
\]
Let $A(\theta)=\rho_+(\theta)\theta$ and
$B(\theta)=\rho_-(\theta)\theta$. We decompose $M$ so that each endpoint lies
either in a fixed face of $\partial F$ or in a fixed element of
$\mathcal E(F)$. If $\beta_\pm$ denotes the angle appearing in the coarea
formula for the radial projection at $A(\theta)$ and $B(\theta)$, then, by the area formula,
\begin{align}\label{eq:edge-coarea-pairing}
\int_{\Gamma_H\cap\overline R}|x|^{-n}x\cdot
\nu^\tau_{\partial F\cap\partial H}\,d\HH^{n-2}
&=
\int_M \rho_+^{-1}
\frac{\theta\cdot\nu^\tau_{\partial F\cap\partial H}(A(\theta))}
{\cos\beta_+(\theta)}\,d\HH^{n-2}(\theta)\notag\\
&\quad+
\int_M \rho_-^{-1}
\frac{\theta\cdot\nu^\tau_{\partial F\cap\partial H}(B(\theta))}
{\cos\beta_-(\theta)}\,d\HH^{n-2}(\theta).
\end{align}
At an endpoint contained in a codimension-two face, the quotient in
\eqref{eq:edge-coarea-pairing} denotes the sum of the two one-sided quotients
coming from the adjacent faces.
Indeed, if an endpoint is $x=\rho(\theta)\theta$, the Jacobian of $\Pi$ on the
corresponding codimension-two piece is
$\rho^{-(n-2)}\cos\beta$, whereas
$|x|^{-n}x=\rho^{1-n}\theta$.
Equivalently, we can introduce the angles $\gamma_\pm(\theta)$ by
\[
\frac{\theta\cdot\nu^\tau_{\partial F\cap\partial H}(A(\theta))}
{\cos\beta_+(\theta)}
=\cos\gamma_+(\theta),\qquad
\frac{\theta\cdot\nu^\tau_{\partial F\cap\partial H}(B(\theta))}
{\cos\beta_-(\theta)}
=\cos\gamma_-(\theta).
\]
We simply bound $\cos\gamma_+(\theta)$ by $1$. For $B(\theta)$ we observe that since no face is radial, it belongs to a codimension-two face $S=Q^+\cap Q^-\in\mathcal E(F)$. Also, if $\eta_S^\pm$ are the
exterior unit conormals to the two adjacent faces, the defect angle in the
two-dimensional normal section is
\[
\delta_S:=\arccos(-\eta_S^+\cdot\eta_S^-)
=2\arcsin\frac{|\eta_S^++\eta_S^-|}{2}.
\]
Since $\pi-\gamma_-(\theta)\le\delta_S$, we get
\[
\cos\gamma_-(\theta)
\le -\cos\delta_S
=-1+\frac{|\eta_S^++\eta_S^-|^2}{2}
\le -1+|\eta_S^++\eta_S^-|.
\]
This proves that
\begin{align*}
\int_{\Gamma_H\cap\overline R}|x|^{-n}x\cdot
\nu^\tau_{\partial F\cap\partial H}\,d\HH^{n-2}
&\le
\int_M(\rho_+^{-1}-\rho_-^{-1})\,d\HH^{n-2}
\\
&\quad +C(n)\sum_S
\int_{\Pi(\Gamma_H\cap\overline R\cap S)}
|\eta_S^++\eta_S^-|\,d\HH^{n-2}.
\end{align*}
Note that,  since $\rho_- \leq \rho_+$, the first term is non-positive. Also, since $\partial F\subset A_{1/2,2}$, by the area formula we can bound the second term by  
\[
C(n)\sum_S\int_{\Gamma_H\cap\overline R\cap S}
|\eta_S^++\eta_S^-|\,d\HH^{n-2}.
\]
Noticing that
\begin{equation}\label{eq:poly-edge-curvature}
\mathbf H_{\partial F}\llcorner S
=\pm(\eta^+_S+\eta^-_S)\,\HH^{n-2}\llcorner S,
\qquad
|\mathbf H_{\partial F}|\llcorner S
=|\eta^+_S+\eta^-_S|\,\HH^{n-2}\llcorner S,
\end{equation}
summing over the regular pieces $R$ gives the result.
\end{proof}

We can now prove the fixed-truncation leveling estimate. The preceding
polyhedral estimate handles the approximating polyhedra, and
Corollary~\ref{cor:piecewise polyhedral approximation} passes the result back
to the admissible truncation.

\begin{prop}\label{leveling oscillation}
There exists $\sz_2=\sz_2(n)>0$ such that, if $0<\sz<\sz_2$ and
$F$ is an admissible truncation satisfying
$$
B_{1-2\sz}\subset F\subset B_{1+2\sz},
$$
then, for every constant $\lambda\in\R$,
\begin{multline}\label{eq:leveling-fixed-F}
\frac1{4n}\HH^{n-1}(\partial^*F\cap\mathcal C_{\mathcal B_F})
+\int_{\partial^*F\setminus\mathcal C_{\mathcal B_F}}
\left(1-\frac{(x\cdot\nu_F)^2}{|x|^2}\right)\,d\HH^{n-1}\\
\le
\int_{\partial^*F}
\left(1-\frac{(x\cdot\nu_F)^2}{|x|^2}\right)\,d\HH^{n-1}
+C(n)\|\mathcal H_F-\lambda\|_{L^2(\partial^*F)}^2
+C(n)|\mathbf H^s_{\partial F}|(\partial F).
\end{multline}
\end{prop}

\begin{proof}
We first reduce to generic translations. Choose
$a_j\to0$ among the translation vectors given by
Lemma~\ref{lem:generic-sigma}, and set $F_j:=F-a_j$. 
All the objects below,
including the tangency set, the directions in $\mathcal B_{F_j}$, and the
radial layers, are understood for this translated set; in particular
\begin{equation}\label{eq:zero-projected-tangencies}
\HH^{n-1}\bigl(\partial^*F_j\cap\Pi^{-1}(\Pi(\Sigma_{F_j}))\bigr)=0.
\end{equation}
Moreover $F_j$ satisfies the same annular inclusion with
$\sigma$ replaced by $\sigma+|a_j|$, and this only changes the estimates below
by $C(n)|a_j|$.
It is enough to prove the estimate for $F_j$ and then let $j\to\infty$. Indeed,
translation preserves the curvature terms:
\[
\|\mathcal H_{F_j}-\lambda\|_{L^2(\partial^*F_j)}
=\|\mathcal H_F-\lambda\|_{L^2(\partial^*F)},\qquad
|\mathbf H^s_{\partial F_j}|(\partial F_j)
=|\mathbf H^s_{\partial F}|(\partial F).
\]
Also, the integrals of
\[
1-\frac{(x\cdot\nu_{F_j})^2}{|x|^2}
\]
over $\partial^*F_j$ converge by dominated convergence on the $C^1$ and
spherical pieces of the admissible truncation. 

For the terms containing $\mathcal C_{\mathcal B_{F_j}}$, fix an
open set $U\supset\Pi(\Sigma_F)$. Since $a_j\to0$, the tangencies and the
interfaces of $F_j$ project inside $U$ for $j$ large. Since on
$\mathbb S^{n-1}\setminus U$ all intersections of the rays with $\partial F$
are transverse, the implicit function theorem allows us to represent
$\partial F_j\cap\Pi^{-1}(\mathbb S^{n-1}\setminus U)$ by the same number of
$C^1$ radial graphs as for $F$, converging in $C^1$ and with the same order
along each ray. Thus
\[
\mathbf 1_{\mathcal C_{\mathcal B_{F_j}}}(x-a_j)
\to
\mathbf 1_{\mathcal C_{\mathcal B_F}}(x)
\qquad\hbox{in }L^1\bigl(\partial F\setminus\Pi^{-1}(U)\bigr).
\]
By the area formula applied to the radial projection on each $C^1$ piece,
\[
\HH^{n-1}\bigl((\partial^*F\setminus\Sigma_F)
\cap\Pi^{-1}(\Pi(\Sigma_F))\bigr)=0.
\]
Thus, up to a negligible set, the remaining part of the limiting cone is
contained in $\Sigma_F$, where
$1-(x\cdot\nu_F)^2/|x|^2=1$. Letting first $j\to\infty$ and then
$U\downarrow\Pi(\Sigma_F)$ gives the desired estimate for $F$. We may therefore
assume \eqref{eq:zero-projected-tangencies} and omit the index $j$.

Now, let $L_m$ be the consecutive radial layers given by
Lemma~\ref{lem:radial-layer-package}. Then each $L_m$ has finite perimeter,
the boundary $\partial_{\rm rad}L_m$ is radial, and
\[
\mathbf 1_{(\partial^*F\cap\mathcal C_{\mathcal B_F})\setminus\Sigma_F}
\le
\sum_{m\ge1}\mathbf 1_{(\partial^*L_m\cap\partial^*F)\setminus\Sigma_F}
\le
2\,\mathbf 1_{(\partial^*F\cap\mathcal C_{\mathcal B_F})\setminus\Sigma_F}.
\]
Also, let $H$ be one of these consecutive radial layers, and let
\[
\partial_{\rm rad}H:=\partial^*H\setminus\partial^*F,
\qquad
\Gamma_H:=\partial F\cap\partial H\cap
\overline{\partial_{\rm rad}H}
\]
be the codimension-two interface where $\partial^*H\cap\partial^*F$ meets
$\partial_{\rm rad}H$. We
also set
\[
\mathfrak R_F^\lambda(H):=
\left|
\int_{\partial F\cap\partial H}|x|^{-n}\langle x,d\mathbf H_{\partial F}\rangle
-\lambda\int_{\partial^*H\cap\partial^*F}|x|^{-n}x\cdot\nu_F\,d\HH^{n-1}
\right|.
\]
We first prove that, for $F$ polyhedral with no radial faces,
\begin{multline}\label{mean curvature control polyhedron}
\left(\frac{n-1}{n}-C(n)\sz\right)\HH^{n-1}(\partial^*H\cap\partial^*F)\\
\le
\int_{\partial^*H\cap\partial^*F}
\left(1-\frac{(x\cdot\nu_F)^2}{|x|^2}\right)\,d\HH^{n-1}
+C(n)\mathfrak R_F^\lambda(H)\\
+C(n)|\mathbf H_{\partial F}|(\Gamma_H).
\end{multline}
To see this, test the vector-valued curvature identity
\eqref{eq:vector curvature measure} for $F$ with $|x|^{-n}x\varphi_\ez$,
where $\varphi_\ez$ is a cut-off on $\partial F$
which equals $1$ on $\partial H\cap\partial F$ and whose tangential gradient is
supported in an $\ez$-neighborhood of $\Gamma_H$.  Letting
$\ez\to0$ gives
\begin{multline}
\int_{\partial F\cap\partial H}
\left[(n-1)|x|^{-n}
+n|x|^{-n}\left(\frac{(x\cdot\nu_F)^2}{|x|^2}-1\right)\right]\,d\HH^{n-1}
\\
 =
\int_{\Gamma_H}
|x|^{-n}x\cdot\nu^\tau_{\partial F\cap\partial H}\,d\HH^{n-2}
+\int_{\partial F\cap\partial H}|x|^{-n}\langle x,d\mathbf H_{\partial F}\rangle .
\label{x test}
\end{multline}
On $\partial_{\rm rad}H$ one has
$x\cdot\nu_H=0$, and
$\int_{\partial^*H}|x|^{-n}x\cdot\nu_H\,d\HH^{n-1}=0$.  On
$\partial^*H\cap\partial^*F$ one has $\nu_H=\pm\nu_F$, with a fixed sign for
the layer.  Therefore
\[
\int_{\partial^*H\cap\partial^*F}|x|^{-n}x\cdot\nu_F\,d\HH^{n-1}=0.
\]
Thus the $\lambda$-part may be subtracted from the curvature term before taking
absolute values, and the curvature contribution in \eqref{x test} is bounded
by $C(n)\mathfrak R_F^\lambda(H)$.

It remains to estimate the boundary term in \eqref{x test}. By
Lemma~\ref{lem:polyhedral-edge},
\[
\int_{\Gamma_H}
|x|^{-n}x\cdot\nu^\tau_{\partial F\cap\partial H}\,d\HH^{n-2}
\le
C(n)|\mathbf H_{\partial F}|(\Gamma_H).
\]
Combining this with \eqref{x test} and using
$\partial F\subset A_{1-2\sz,1+2\sz}$ proves
\eqref{mean curvature control polyhedron}.

We pass from polyhedra to the present admissible truncation after
summing over the layers. Let $F_k$ be the polyhedral approximations given by
Corollary~\ref{cor:piecewise polyhedral approximation}. Since the condition
that no face be radial is generic under arbitrarily small perturbations of the
vertices, we perturb $F_k$, if necessary, by $o_k(1)$ and keep the notation
$F_k$; the perturbation is chosen small enough to preserve the convergences and
the limsup bound in Corollary~\ref{cor:piecewise polyhedral approximation}.

Choose open neighborhoods $U_\ez\supset\Pi(\Sigma_F)$, decreasing to
$\Pi(\Sigma_F)$, and set
\[
K_\ez:=A_{1-2\sz,1+2\sz}\cap\Pi^{-1}(U_\ez).
\]
By choosing the neighborhoods slightly, we may assume
$|\mathbf H_{\partial F}|(\partial K_\ez)=0$ for every $\ez$.

Let $L_{m,k}$ be the consecutive radial
layers of $F_k$. We shall use only their behavior away from $K_\ez$. Indeed,
on $\mathbb S^{n-1}\setminus U_\ez$ the boundary $\partial F$ is given by
finitely many transverse $C^1$ radial graphs, with positive separation. The
uniform convergence of the polyhedra and of their normals gives the
corresponding graphs for $\partial F_k$, with the same order. Thus, after
relabeling the layers that meet $A_{1-2\sz,1+2\sz}\setminus K_\ez$ and passing
to a diagonal subsequence, for every fixed such $m$,
\[
|(L_{m,k}\Delta L_m)\setminus K_\ez|\to0,
\]
and
\[
\HH^{n-1}\llcorner(\partial_{F_k}L_{m,k}\setminus K_\ez)
\stackrel{*}{\rightharpoonup}
\HH^{n-1}\llcorner(\partial_F L_m\setminus K_\ez)
\]
with convergence of the masses. Moreover, for $k$ large, all interfaces
between $\partial_{F_k}L_{m,k}$ and $\partial_{\rm rad}L_{m,k}$ are contained
in $K_\ez$. In particular, outside $K_\ez$ only finitely many indices $m$
occur, uniformly for $k$ large.

We now apply \eqref{mean curvature control polyhedron} to all consecutive radial
layers of $F_k$ and then sum. After summing over $m$, the portions
$\partial^*L_{m,k}\cap\partial^*F_k$ cover
$\partial^*F_k\cap\mathcal C_{\mathcal B_{F_k}}$ with multiplicity between
$1$ and $2$, whereas the interface terms have bounded overlap and are contained in
$K_\ez$. Let $\mathcal K_k^{\rm cut}$ denote the union of the
codimension-two faces of $\partial F_k$ approximating the two interfaces
$\mathcal S_F$. Passing first $k\to\infty$ outside $K_\ez$, and using the
limsup bound on
$|\mathbf H^s_{\partial F_k}|(\mathcal K_k^{\rm cut})$ from
Corollary~\ref{cor:piecewise polyhedral approximation} inside $K_\ez$, gives
\begin{multline*}
\left(\frac{n-1}{n}-C(n)\sz\right)
\HH^{n-1}(\partial^*F\cap\mathcal C_{\mathcal B_F})\\
\le
2\int_{\partial^*F\cap\mathcal C_{\mathcal B_F}}
\left(1-\frac{(x\cdot\nu_F)^2}{|x|^2}\right)\,d\HH^{n-1}
+C(n)\int_{\partial^*F\cap\mathcal C_{\mathcal B_F}}
|\mathcal H_F-\lambda|\,d\HH^{n-1}\\
+C(n)\int_{\partial^*F\cap K_\ez}|\mathcal H_F|\,d\HH^{n-1}
+C(n)|\mathbf H^s_{\partial F}|(\partial F)
+C(n)\HH^{n-1}(\partial^*F\cap K_\ez).
\end{multline*}
Since 
$\HH^{n-1}(\partial^*F\cap\Pi^{-1}(\Pi(\Sigma_F)))=0$ (by \eqref{eq:zero-projected-tangencies}), the
absolutely continuous part of $\mathbf H_{\partial F}$ gives no mass to the
limiting cone.
Thus, letting $\ez\to0$, we obtain
\begin{multline}\label{mean curvature control H}
\left(\frac{n-1}{n}-C(n)\sz\right)
\HH^{n-1}(\partial^*F\cap\mathcal C_{\mathcal B_F})\\
\le
2\int_{\partial^*F\cap\mathcal C_{\mathcal B_F}}
\left(1-\frac{(x\cdot\nu_F)^2}{|x|^2}\right)\,d\HH^{n-1}
+C(n)\int_{\partial^*F\cap\mathcal C_{\mathcal B_F}}
|\mathcal H_F-\lambda|\,d\HH^{n-1}\\
+C(n)|\mathbf H^s_{\partial F}|(\partial F).
\end{multline}
Finally, by Young's
inequality,
$$
C(n)\int_{\partial^*F\cap\mathcal C_{\mathcal B_F}}|\mathcal H_F-\lambda|\,d\HH^{n-1}
\le
\frac1{8n}\HH^{n-1}(\partial^*F\cap\mathcal C_{\mathcal B_F})
+C(n)\|\mathcal H_F-\lambda\|_{L^2(\partial^*F)}^2.
$$
Choosing $\sz_2=\sz_2(n)$ small enough, dividing by $2$, and adding the integral over
$\partial^*F\setminus\mathcal C_{\mathcal B_F}$ to both sides gives
\eqref{eq:leveling-fixed-F}, concluding the proof.
\end{proof}

\subsection{$\mu$ is quantitatively close to $n-1$}
Consider, by approximation, the Lipschitz test vector-field $\Phi(x)=\phi(|x|)x$ with
$$\phi(t)=\left\{
\begin{array}{cl}
  1 & \text{ when } \ 0<t<2 \\
\frac{1}{t-1} &  \text{ when } \ t>2
\end{array}\right. .$$ 
Recalling \eqref{eq:div}, it holds
$$ {\rm div}_\tau \big(\Phi(x)\big)=\left\{
\begin{array}{cl}
  n-1 & \text{ when } \ 0<|x|<2 \\
  \frac{n-1}{|x|-1} - \frac{1}{(|x|-1)^2}\left(|x| -\frac{(x\cdot \nu_E)^2}{|x|}\right)&  \text{ when } \ |x|>2
\end{array}\right.$$
and
$$ {\rm div} \big(\Phi(x)\big)=\left\{
\begin{array}{cl}
  n & \text{ when } \ 0<|x|<2 \\
  \frac{n}{|x|-1} - \frac{|x|}{(|x|-1)^2}&  \text{ when } \ |x|>2.
\end{array}\right. $$
Hence, using \eqref{eq:mean div} and the divergence theorem inside $E$, we get
\begin{align*}
&(n-1)\mathscr H^{n-1}(\partial^*E\cap B_2)+ \int_{\partial^*E\setminus B_2} \bigg(\frac{n-1}{|x|-1} - \frac{1}{(|x|-1)^2}\left(|x| -\frac{(x\cdot \nu_E)^2}{|x|}\right)\bigg)\,d\mathscr H^{n-1}\\
&= \mu \int_{\partial^*E} \Phi(x)\cdot \nu_E\,d\mathscr H^{n-1} + \int_{\partial^*E} R(x)  \Phi(x)\cdot \nu_E \,d\mathscr H^{n-1}\\
&=\mu n |E\cap B_2| -\mu \int_{E\setminus B_2} \bigg(\frac{n}{|x|-1} - \frac{|x|}{(|x|-1)^2}\bigg)\,dx + \int_{\partial^*E} R(x)  \Phi(x)\cdot \nu_E \,d\mathscr H^{n-1}.
\end{align*}
Rearranging terms, this gives
\begin{align*}
(n-1)P(E)-\mu n|E|&=
 \int_{\partial^*E\setminus B_2} \bigg((n-1)-\frac{n-1}{|x|-1} + \frac{1}{(|x|-1)^2}\left(|x| -\frac{(x\cdot \nu_E)^2}{|x|}\right)\bigg)\,d\mathscr H^{n-1}\\
 &-\mu \int_{E\setminus B_2} \bigg(n+\frac{n}{|x|-1} - \frac{|x|}{(|x|-1)^2}\bigg)\,dx + \int_{\partial^*E} R(x)  \Phi(x)\cdot \nu_E \,d\mathscr H^{n-1}.
\end{align*}
Note now that $n|E|=n|B|=P(B)$.
Also, since by assumption $\|R\|_{L^2(\partial^*E)}$ is small, \eqref{bounded mu} implies that $|\mu|$ is bounded by a dimensional constant. Hence, recalling \eqref{difference perimeter} and noticing that all the integrands involving $\Phi$ and its divergence are bounded, using the isoperimetric inequality we obtain
\begin{align*}
|(n-1) - \mu|&\leq C(n)\Big(\big(P(E)-P(B)\big)+\mathscr H^{n-1}(\partial^*E \setminus B_2)+|E\setminus B_2|+\|R\|_{L^1(\partial^*E)}\Big) \\
&\leq C(n)\Big({\rm Exc }(E)+P(E\setminus B_2)+P(E\setminus B_2)^{\frac{n}{n-1}}+\|R\|_{L^2(\partial^*E)}\Big).
\end{align*}
Recalling Lemma~\ref{small tentacle}, this proves that
 \begin{equation} \label{general mu}
|(n-1) - \mu|\leq C(n)\Big({\rm Exc }(E)+\|R\|_{L^2(\partial^*E)}\Big).
\end{equation}

 \subsection{Construct a suitable replacement of $E$}
 Let 
$$ G:=(  E\cap B_{\rho_1})\cup B_{\rho_0},$$
where 
$$\rho_0=1-\frac\sz 2-s^-\in \left(1-\sz,\,1-\frac \sz 2\right),\quad \rho_1=1+\frac\sz 2+s^+\in \left(1+\frac \sz 2,\,1+\sz\right)$$ 
 are given by  Lemma~\ref{lem:good truncation} with center at the origin. 
By Lemma~\ref{small tentacle} and the isoperimetric inequality, for $\delta=\delta(n,\sigma)$ small enough we have
 $$|G\Delta E|^{\frac{n-1}n}\le C(n) \mathscr H^{n-1}(\partial^*(G\Delta E))\le C(n)\|R\|_{L^2(\partial^*E)}^{\frac{2}{1-\eta}},$$
therefore
$$ \left|\mean{G} \frac{x}{|x|}\, dx- \mean{E} \frac{x}{|x|}\,dx\right|\le C(n)|E\Delta G| \leq  C(n)\|R\|_{L^2(\partial^*E)}^{\frac{2}{1-\eta}\cdot \frac{n}{n-1}}.$$
Set
\[
r:=\left(\frac{|G|}{|B|}\right)^{1/n}.
\]
Thus, by Lemma~\ref{translation} applied to the volume-normalized set $G/r$,
there exists $y_0\in \mathbb R^n$ such that
$\int_{(G+y_0)/r}x/|x|\,dx=0$.  Setting $x_0:=-y_0$, we have
\begin{equation}\label{r x0}
|r-1|+|x_0|\le C(n) |G\Delta E|\le C(n)\|R\|_{L^2(\partial^*E)}^{\frac{2}{1-\eta}\cdot \frac{n}{n-1}},
\end{equation}
so that
$$\hat E:=\frac {E-x_0}{r}  \quad \text{ and } \quad \hat G=\frac {G-x_0}{r}$$
 satisfy
\begin{equation}\label{mean curvature hat}
  |\hat G|=|B|,\quad \mean{\hat G}\frac{x}{|x|}\,dx=0,\quad \mathcal H_{\partial^* \hat E}(x)=r\mathcal H_{\partial^* E}(x_0+rx)= r \big(\mu+\hat R(x)\big),
\end{equation}
with $\hat R(x):=R(x_0+rx)$.
In addition, provided $\|R\|_{L^2(\partial^*E)}\le c(n)$ is sufficiently small,  we can ensure that
\[
B_{1-2\sigma}\subset \hat G\subset B_{1+2\sigma},
\qquad
\partial \hat G\subset A_{1-\sz,\,1+\sz},
\qquad
{\rm Exc}(\hat G)\le\delta_1(n,\kappa),
\]
and, since $G$ is an admissible truncation by Lemma~\ref{lem:good truncation}
and admissible truncations are invariant under the preceding translation and
dilation, $\hat G$ is admissible. Thus Proposition~\ref{stability sets} and
Proposition~\ref{leveling oscillation} apply to $\hat G$.
Moreover, $\partial^*\hat G$ is contained in the union of
$\partial^*\hat E$ and of two translated spheres. We decompose the
vector-valued curvature measure of $\hat G$ as
\begin{equation*}
\mathbf H_{\partial\hat G}
=H_{\hat G}\nu_{\hat G}\,\mathscr H^{n-1}\llcorner\partial^*\hat G
+\mathbf H^s_{\partial\hat G},
\qquad
\hat R^a_G:=H_{\hat G}-r\mu .
\end{equation*}

The scaling $x\mapsto (x-x_0)/r$ sends vector curvature measures to
$r^{2-n}$ times their pull-backs, while scalar mean curvatures are multiplied
by $r$. Since \eqref{r x0} gives $r\sim1$, Lemma~\ref{lem:good truncation}
therefore gives, after the harmless rescaling and translation used to define
$\hat G$,
\begin{equation}\label{RGhat singular bound}
|\mathbf H^s_{\partial\hat G}|(\partial\hat G)
\le C(n,\,\sz)\|R\|_{L^2(\partial^*E)}^{\frac{2}{1-\eta}}.
\end{equation}
Moreover, $\hat R^a_G=r\hat R$ on $\partial^*\hat G\cap \partial^*\hat E$. On $\partial^*\hat G\setminus \partial^*\hat E$, the boundary consists of pieces of spheres with radii comparable to $1$, and therefore, by \eqref{general mu} and \eqref{r x0},
$$|\hat R^a_G|\le C(n)\qquad \text{on }\partial^*\hat G\setminus \partial^*\hat E.$$
All uses of Lemma~\ref{small tentacle} for $\hat E$ below are understood after
pulling the estimate back to $E$; the relevant annuli are preserved because
$|x_0|+|r-1|\ll\sigma$. In particular
$\mathscr H^{n-1}(\partial^*\hat G\setminus \partial^*\hat E)\le C(n)\|R\|_{L^2(\partial^*E)}^{\frac{2}{1-\eta}}$, and we obtain
\begin{equation}\label{RGhat bound}
\|\hat R^a_G\|_{L^2(\partial^*\hat G)}^2 \le C(n)\|r\hat R\|_{L^2(\partial^*\hat E)}^2 + C(n)\|R\|_{L^2(\partial^*E)}^{\frac{2}{1-\eta}}.
\end{equation}

We also note that, since\footnote{For $n\geq 3$, the integral $\int_{B_{1/2}} \frac{1}{|x||x-x_0|}\,dx$ is uniformly bounded as $|x_0| \to 0$, while in dimension $n=2$ it diverges as $\big|\log|x_0|\big|$ as $|x_0|\to 0$. Thus
$$
|x_0|\int_{E} \frac{1}{|x||x-x_0|}\, dx \leq C(n)\left\{
\begin{array}{cl}
|x_0| & \text{for $n \geq 3$},\\
|x_0|\big|\log|x_0|\big| & \text{for $n =2$},
\end{array}
\right.
\leq C(n)|x_0|^{\frac{n-1}{n}} \qquad \text{for $|x_0|\ll 1$}.
$$
}
\begin{align*}
  \left|\int_{E} \frac{n-1}{|x|} - \frac{n-1}{|x-x_0|}\, dx\right| 
\le (n-1)|x_0|\int_{E} \frac{1}{|x||x-x_0|}\, dx \leq C(n)|x_0|^{\frac{n-1}{n}},
\end{align*}
it follows from the divergence theorem and \eqref{r x0} that
\begin{align}
\big|{\rm Exc}(\hat E)-{\rm Exc}(E)\big|&= \bigg| P(\hat E)- \int_{\hat E} \frac{n-1}{|x|}\, dx -P(E)+\int_{E} \frac{n-1}{|x|}\, dx\bigg|\nonumber\\
&=\bigg| r^{1-n}P(E)-r^{1-n} \int_{E-x_0} \frac{n-1}{|x|}\, dx -P(E)+\int_{E} \frac{n-1}{|x|}\, dx\bigg|  \nonumber\\
&\leq C(n)|r-1|+ C(n) \left|\int_{E} \frac{1}{|x|} - \frac{1}{|x-x_0|}\, dx\right|\leq C(n)\| R\|^{\frac 2{1-\eta}}_{L^2(\partial^*E)}. \label{difference 1}
\end{align}
Similarly, defining $\hat B=\frac{B-x_0}{r},$ by the triangle inequality and \eqref{r x0} we get
\begin{align}
|\hat E\Delta B| &\ge  |\hat E\Delta \hat B|  - |\hat B\Delta B|= r^{-n}|E\Delta B|  - C(n)\big(|r-1|+|x_0|\big) \nonumber\\
 &\ge  |E\Delta B|  - C(n)\big(|r-1|+|x_0|\big) \geq |E\Delta B| -C(n)\| R\|^{\frac 2{1-\eta}}_{L^2(\partial^*E)}. \label{difference 2}
\end{align}

 \subsection{Stability estimates}
 
Consider the test function $\Phi(x) :=x-\psi(|x|)\frac{x}{|x|}$, where
$$\psi(t)=\left\{
\begin{array}{cl}
(1-\sz)^{-1} t& \text{ when } \ 0<t<1-\sz \\
1 &  \text{ when } \ 1-\sz\le t\le 1+\sz\\
t-\frac{\sz}{t-\sz} & \text{ when } \ t>1+\sz
\end{array}\right. .$$ 
We note that
\begin{equation}
\label{psi}
\left\{\begin{array}{cl}
|\Phi(x)|=\big|\psi(|x|)-|x|\big|\le \sz&\text{for all }x \in \R^n,\\
t^{-1}\psi(t)=\psi'(t)=(1-\sigma)^{-1}  &\text{for } t<1-\sz,\\
1-\frac{\psi(t)}{t}\le C\sz,\ 1-\psi'(t)\le C  \sz, \  \left|t^{-1}\psi(t)-\psi'(t)\right|\le C\sz  &
 \text{for }   t>1+\sigma.
\end{array}
\right.
\end{equation}
Using $\Phi(x)=x-\psi(|x|)\frac{x}{|x|}$ as a test function in the analogue of \eqref{eq:mean div} for $\hat E$, and applying \eqref{eq:div} together with \eqref{mean curvature hat} we get
\begin{equation}
\begin{split}
& \int_{\partial^* \hat E}  (n-1)\left(1-\frac{\psi(|x|)}{|x|}+r\big[\psi(|x|)-|x|\big]\frac{x}{|x|}\cdot \nu_{\hat E}\right)   \, d\mathscr H^{n-1} \\
&\qquad + \int_{\partial^* \hat E}\left(\frac{\psi(|x|)}{|x|} -  \psi'(|x|)\right)\left(1-\frac{(x\cdot \nu_{\hat E})^2}{|x|^2}\right)  \, d\mathscr H^{n-1} \\
&= r(\mu-n+1)  \int_{\partial^* \hat E}  \big[|x|-\psi(|x|)\big] \frac{x}{|x|}\cdot \nu_{\hat E} \, d\mathscr H^{n-1} +  r\int_{\partial^* \hat E} \hat R(x)  \Phi\cdot \nu_{\hat E} \,d\mathscr H^{n-1} \label{test general}
\end{split}
\end{equation}

We now estimate the terms in \eqref{test general}. Set
\begin{align*}
&I:=\int_{\partial^* \hat E}  (n-1)\left(1-\frac{\psi(|x|)}{|x|}+r\big[\psi(|x|)-|x|\big]\frac{x}{|x|}\cdot \nu_{\hat E}\right)   \, d\mathscr H^{n-1} ,\\
&II:=\int_{\partial^* \hat E}\left(\frac{\psi(|x|)}{|x|} -  \psi'(|x|)\right)\left(1-\frac{(x\cdot \nu_{\hat E})^2}{|x|^2}\right)  \, d\mathscr H^{n-1},\\
&III:=r(\mu-n+1)  \int_{\partial^* \hat E}  \big[|x|-\psi(|x|)\big] \frac{x}{|x|}\cdot \nu_{\hat E} \, d\mathscr H^{n-1}, \\
&IV:=r \int_{\partial^* \hat E} \hat R(x)  \Phi\cdot \nu_{\hat E} \,d\mathscr H^{n-1}.
\end{align*}

\subsubsection*{Estimating I}
By  \eqref{r x0}, \eqref{psi}, and Lemma~\ref{small tentacle}, and using that $\nu_{\hat E}=\nu_{\hat G}$ on $\partial^*\hat E\cap \partial^*\hat G$, while the integrand is uniformly bounded on $A_{1-\sigma,\,1+\sigma}$, we can write 
\begin{align}
I&\geq (n-1)\int_{\partial^* \hat G}  \left(1-\frac{1}{|x|}+r(1-|x|)\frac{x}{|x|}\cdot \nu_{\hat G}\right) \, d\mathscr H^{n-1}- C(n)\mathscr H^{n-1}(\partial^*  \hat G\Delta \partial^*  \hat E)\nonumber \\ 
& =   -(n-1)\int_{\partial^* \hat G}  \left(\frac{1}{|x|}+|x|-2\right)\, d\mathscr H^{n-1} +(n-1)\int_{\partial^* \hat G}\left(|x|-1\right)\left(1-r\frac{x}{|x|}\cdot \nu_{\hat G}\right) \, d\mathscr H^{n-1}\nonumber\\
&\qquad - C(n)\mathscr H^{n-1}(\partial^*  \hat G\Delta \partial^*  \hat E)\nonumber\\
&\geq   -(n-1)\int_{\partial^* \hat G}  \left(\frac{1}{|x|}+|x|-2\right)\, d\mathscr H^{n-1}-C(n)\int_{\partial^* \hat G}\big||x|-1\big| \left(1-\frac{x}{|x|}\cdot \nu_{\hat G}\right) \, d\mathscr H^{n-1}\nonumber\\
&\qquad -C(n)|r-1|- C(n)\mathscr H^{n-1}(\partial^*  \hat G\Delta \partial^*  \hat E)\nonumber\\
& \ge  - (n-1)\int_{\partial^* \hat G}  \left(\frac{1}{|x|}+|x|-2\right)\, d\mathscr H^{n-1} -C(n)\sz{\rm Exc}(\hat G)- C(n)  \| R\|^{\frac 2{1-\eta}}_{L^2(\partial^*E)}, \label{1st term}
\end{align}
where, in the last inequality, we used that $\partial^* \hat G \subset A_{1-\sigma,1+\sigma}$. Note that the boundary portions created or removed by the truncation are controlled by Lemma~\ref{small tentacle}.

\subsubsection*{Estimating II}

Note that $\psi\equiv 1$ on $\partial^* \hat G.$ 
The part of $\partial^*\hat E\cap\hat G$ not contained in $\partial^*\hat G$
contributes with a nonnegative integrand, so it may be discarded in the lower bound. Hence,
applying Lemma~\ref{small tentacle} to $\partial\hat E\setminus \hat G$ and
recalling that $\partial^* \hat G \subset A_{1-\sigma,1+\sigma}$, we have
\begin{align}
II &\ge \int_{\partial^* \hat G}  \frac{1}{|x|}  \left(1-\frac{(x\cdot \nu_{\hat G})^2}{|x|^2}\right) \, d\mathscr H^{n-1} +  \int_{\partial^*  \hat E\setminus \hat G}  \left(\frac{\psi(|x|)}{|x|} -  \psi'(|x|)\right)\left(1-\frac{(x\cdot \nu_{\hat E})^2}{|x|^2}\right) \, d\mathscr H^{n-1}\nonumber\\
& \quad\quad -  \int_{\partial^*  \hat G\setminus\partial^* \hat E}  \frac{1}{|x|}  \left(1-\frac{(x\cdot \nu_{\hat G})^2}{|x|^2}\right) \, d\mathscr H^{n-1} \nonumber\\
&\ge  (1-\sz) \int_{\partial^* \hat G}   \left(1-\frac{(x\cdot \nu_{\hat G})^2}{|x|^2}\right) \, d\mathscr H^{n-1}  - C(n) \|\hat  R\|_{L^2(\partial^* \hat E)}^{\frac{2}{1-\eta}}-C(n)\|R\|_{L^2(\partial^*E)}^{\frac{2}{1-\eta}\cdot \frac{n}{n-1}}, \label{2nd term}
\end{align}
where we applied \eqref{r x0} to get
$$ \left|1-\frac{(x\cdot \nu_{\hat G})^2}{|x|^2}\right| \le C(n)|x_0|\le C(n)\|R\|_{L^2(\partial^*E)}^{\frac{2}{1-\eta}\cdot \frac{n}{n-1}}, \quad x\in \partial^*  \hat G\setminus\partial^* \hat E.$$

\subsubsection*{Estimating III}
Observe that 
$$|x|=1=\psi(|x|)  \quad  \text{ for } x\in \mathbb S^{n-1}.$$
Then, the divergence theorem together with \eqref{psi} give
\begin{align*} 
  &\left|\int_{\partial^* \hat E}  \big[|x|-\psi(|x|)\big] \frac{x}{|x|}\cdot \nu_{\hat E} \, d\mathscr H^{n-1}\right| \\
&\le \left|\int_{\hat E\setminus B} (n-1) \left(1-\frac{\psi(|x|)}{|x|}\right)+1-\psi'(|x|)\, dx\right|\\
&\quad\quad + \left|\int_{B\setminus \hat E} (n-1) \left(1-\frac{\psi(|x|)}{|x|}\right)+1-\psi'(|x|)\, dx\right|\\
&\le C(n)|\hat E\Delta B|.
\end{align*}
Thus, recalling \eqref{general mu}, \eqref{difference 1}, and that $r \sim 1$,
it follows by Young's inequality that, for any $0<\ez<1$,
\begin{multline} 
  \left|r(\mu-n+1)  \int_{\partial^* \hat E}  \big[|x|-\psi(|x|)\big] \frac{x}{|x|}\cdot \nu_{\hat E} \, d\mathscr H^{n-1}\right|
 \le   C(n )   |\hat E\Delta B| \left( {\rm Exc}(\hat E)  +  \| R\|_{L^2(\partial^* E)}\right) \\
  \le  C(n )     |\hat E\Delta B| {\rm Exc}(\hat E)  +C(n)\ez |\hat E\Delta B|^2+   C(n)\ez^{-1}\| R\|^2_{L^2(\partial^*E)}. \label{mu term}
\end{multline}

\subsubsection*{Estimating IV}
We employ  Lemma~\ref{small tentacle} and \eqref{psi} to obtain that
\begin{align*}
&\int_{\partial^*\hat  E} [|x|-\psi(|x|)]^2 \left(\frac{x}{|x|}\cdot \nu_{\hat E}\right)^2 \, d\mathscr H^{n-1}\\
& \leq  \int_{\partial^*\hat G} (|x|-1)^2 \left(\frac{x}{|x|}\cdot \nu_{\hat G}\right)^2 \, d\mathscr H^{n-1}   + C(n)  \|\hat  R\|_{L^2(\partial^* \hat E)}^{\frac{2}{1-\eta}}\\
& \leq  (1+ \sz) \int_{\partial^* \hat G}  \left(\frac{1}{|x|}+|x|-2\right)\, d\mathscr H^{n-1}  +  C(n) \|\hat  R\|_{L^2(\partial^* \hat E)}^{\frac{2}{1-\eta}},
\end{align*}
where we applied that, for $\sz>0$ small, as $\partial \hat G\subset A_{1-\sz,\,1+\sz}$, 
\begin{align*}
    (|x|-1)^2 \left(\frac{x}{|x|}\cdot \nu_{\hat G}\right)^2  \le  (|x|-1)^2 &  = |x|  \left(\frac{1}{|x|}+|x|-2\right)\\
   &  \le (1+\sz)   \left(\frac{1}{|x|}+|x|-2\right) \quad\quad  \text{ for any }  x\in \partial \hat G.
\end{align*} 
Then directly via Young's inequality, for $0<\ez<1$, 
\begin{multline}
  \int_{\partial^* \hat E} r \hat R(x)  \Phi\cdot \nu_{\hat E} \,d\mathscr H^{n-1} 
\le  \ez \int_{\partial^*\hat  E} [|x|-\psi(|x|)]^2 \left(\frac{x}{|x|}\cdot \nu_{\hat E}\right)^2 \, d\mathscr H^{n-1} + \ez^{-1} \|r\hat R\|^2_{L^2(\partial^* \hat E)}   \\
\le  \ez^{-1}  \|r \hat  R\|^2_{L^2(\partial^* \hat  E)} +C(n,\sz) \ez  \|\hat R\|_{L^2(\partial^* \hat E)}^{\frac{2}{1-\eta}} +\ez (1+ \sz) \int_{\partial^* \hat G}  \left( \frac{1}{|x|}+|x|-2\right)\, d\mathscr H^{n-1}.    \label{R term}
\end{multline}

\subsubsection*{Conclusion}

By plugging \eqref{1st term}, \eqref{2nd term}, \eqref{mu term} and \eqref{R term} into \eqref{test general} and applying \eqref{R small}, we conclude
\begin{align}
 & (1-\sz)  \int_{\partial^* \hat G}   \left(1-\frac{(x\cdot \nu_{\hat G})^2}{|x|^2}\right) \, d\mathscr H^{n-1} -C(n)\sz{\rm Exc}(\hat G) \nonumber \\
 & \quad -(n-1+C(n)\sz)(1+\ez)\int_{\partial^* \hat G}   \left(\frac{1}{|x|}+|x|-2\right)\, d\mathscr H^{n-1} \nonumber  \\ 
& \quad \quad
- C(n )  |\hat E\Delta B| {\rm Exc}(\hat E) -C(n)\ez |\hat E\Delta B|^2  \nonumber \\
&\le    C(\ez) \|r \hat R\|^2_{L^2(\partial^* \hat  E)} +C(n,\sz)  \|r \hat R\|_{L^2(\partial^* \hat E)}^{\frac{2}{1-\eta}} + C(n,\,\ez)  \| R\|^2_{L^2(\partial^*E)}. \label{final estimate}
\end{align}
In the rest of the proof we write
$\mathcal B:=\mathcal B_{\hat G}$ and
$\mathcal C_{\mathcal B}:=\mathcal C_{\mathcal B_{\hat G}}$.

We explain the parameter order in the argument below as follows: First $c_0(n)$  is chosen small in the reduction $\|R\|_{L^2}\le c_0(n)$. Then we fix $\kappa=\kappa(n)$ and choose $\sigma=\sigma(n)$ small. Next $\epsilon=\epsilon(n)$ is chosen small. Finally the initial excess $\delta(n)$ is chosen small enough.

Let
\[
\mathcal T_{\hat G}:=\{x\in\partial^*\hat G:\ x\cdot\nu_{\hat G}=0\},
\qquad
\Sigma_{\hat G}:=\mathcal T_{\hat G}\cup\mathcal S_{\hat G}.
\]
If $\hat G^S$ denotes the star-shaped rearrangement of $\hat G$, then Lemma~\ref{lem:star trace outside bad} gives
\[
\partial^* \hat G \setminus(\mathcal C_{\mathcal B}\cup\mathcal T_{\hat G})
=\partial^* \hat G^S \setminus\mathcal C_{\mathcal B}
\]
up to negligible sets. Since the additional interface set
$\mathcal S_{\hat G}$ is $\HH^{n-1}$-negligible, the same identity holds with
$\Sigma_{\hat G}$ in place of $\mathcal T_{\hat G}$, and
$\nu_{\hat G}=\nu_{\hat G^S}$ on the left-hand
side. On the
boundary of the star-shaped set $\hat G^S$ one has
$\frac{x}{|x|}\cdot \nu_{\hat G^S}\ge 0$. On the tangency part of
$\Sigma_{\hat G}$ the two integrands below are both equal to $1$, while
$\mathcal S_{\hat G}$ is negligible for $\HH^{n-1}$. Hence
\begin{equation*}
\int_{\partial^* \hat G \setminus \mathcal C_{\mathcal B}}   \left(1-\frac{(x\cdot \nu_{\hat G})^2}{|x|^2}\right) \, d\mathscr H^{n-1}\ge \int_{\partial^* \hat G \setminus \mathcal C_{\mathcal B}}   \left(1-\frac{x}{|x|}\cdot \nu_{\hat G}\right) \, d\mathscr H^{n-1}.
\end{equation*}
Then Proposition~\ref{leveling oscillation}, applied to $\hat G$ with
$\lambda=r\mu$, together with  \eqref{RGhat singular bound} yields
\begin{equation}\label{to excess}
    \begin{split}
\int_{\partial^* {\hat G}} \left(1-\frac{(x\cdot \nu_{\hat G})^2}{|x|^2}\right) \, d\mathscr H^{n-1}
&\ge \ c(n)\mathscr H^{n-1}(\partial^* \hat G\cap \mathcal C_{\mathcal B})+ c(n)\int_{\partial^* \hat G \setminus \mathcal C_{\mathcal B}}   \left(1-\frac{(x\cdot \nu_{\hat G})^2}{|x|^2}\right) \, d\mathscr H^{n-1}\\
&\quad \quad - C(n) \|\hat R^a_G\|_{L^2(\partial^*\hat G)}^{2}-C(n)\| R\|_{L^2(\partial^*E)}^{\frac{2}{1-\eta}}\\
&\ge  \  c(n) {\rm Exc}(\hat G)- C(n) \|\hat R^a_G\|_{L^2(\partial^*\hat G)}^{2}-C(n)\|R\|_{L^2(\partial^*E)}^{\frac{2}{1-\eta}}. 
\end{split}
\end{equation}
Now, according to Proposition~\ref{stability sets},  by first choosing $\kappa=\kappa(n)>0$ small so that
$$C(n)\kappa\le \frac \zeta {32(n-1+\zeta)},$$ and then  $\sz=\sz(n)>0$ together with $\ez=\ez(n)>0$ small enough, we conclude  that
 \begin{align*}
&\ (1-\sz)   \int_{\partial^* \hat G }   \left(1-\frac{(x\cdot \nu_{\hat G})^2}{|x|^2}\right) \, d\mathscr H^{n-1}  -C(n)\sz{\rm Exc}(\hat G) \\
&\quad \quad -\left(n-1+C(n)\sz\right)(1+\ez)\int_{\partial^* \hat G}   \left(\frac{1}{|x|}+|x|-2\right)\, d\mathscr H^{n-1}\\
  &\ge \  \frac \zeta {8(n-1+\zeta)} \left[ \kappa \mathscr H^{n-1}(\partial^* \hat G\cap \mathcal{C}_{\mathcal B})+ \int_{\partial^* \hat G }   \left(1-\frac{(x\cdot \nu_{\hat G})^2}{|x|^2}\right) \, d\mathscr H^{n-1}  \right]\\
 &\quad  + \frac{\zeta}{16}  \int_{\partial^* \hat G}   \left( \frac{1}{|x|}+|x|-2\right)\, d\mathscr H^{n-1}   -C(n)\sz{\rm Exc}(\hat G)-(1-\sz)\kappa \mathscr H^{n-1}(\partial^* \hat G\cap \mathcal{C}_{\mathcal B})\\
 &\ge   \  c(n) \left[ {\rm Exc}(\hat G) + |\hat G\Delta B|^2  \right]-C(n) \|r\hat R\|_{L^2(\partial^* \hat  E)}^{2}-C(n)\|R\|_{L^2(\partial^*E)}^{\frac{2}{1-\eta}},
\end{align*}
where, in the last line, the negative
$\kappa\,\mathscr H^{n-1}(\partial^*\hat G\cap\mathcal C_{\mathcal B})$ term
is absorbed by \eqref{to excess} and the choice of $\kappa$, while the
$C(n)\sigma{\rm Exc}(\hat G)$ term is absorbed by choosing $\sigma$ small; we
also used \eqref{RGhat bound}. 
By the construction of $G$ and Lemma~\ref{small tentacle}, after scaling and
translating,
\[
|\hat E\Delta \hat G|+\mathscr H^{n-1}(\partial^*\hat G\Delta\partial^*\hat E)
\le C(n)\left(  \|r \hat  R\|_{L^2(\partial^* \hat E)}^{\frac{2}{1-\eta}} + \|R\|_{L^2(\partial^*E)}^{\frac{2}{1-\eta}} \right).
\]
Moreover, the construction of $G$ and  \eqref{r x0} imply
$$\left|1-\frac x {|x|}\cdot \nu_{\hat G}\right|\le C(n)\|R\|_{L^2(\partial^*E)}^{\frac{2}{1-\eta}}  \quad \text{ for }  \mathscr H^{n-1}\text{-a.e. }x\in \partial^*\hat G\setminus \partial^*\hat E,$$
and therefore
$$|{\rm Exc}(\hat E)-{\rm Exc}(\hat G)|\le C(n)\left(  \|r \hat  R\|_{L^2(\partial^* \hat E)}^{\frac{2}{1-\eta}} + \|R\|_{L^2(\partial^*E)}^{\frac{2}{1-\eta}} \right). $$
Consequently, we conclude from \eqref{final estimate} and the triangle inequality that, for $\sz=\sz(n)>0$ small and $\ez=\ez(n)>0$ even smaller, and then for the initial excess $\delta(n)$ small enough to absorb the product
$|\hat E\Delta B|{\rm Exc}(\hat E)$,
\begin{align}
  & C(n) \|r \hat  R\|^2_{L^2(\partial^* \hat  E)} +C(n)  \|r \hat  R\|_{L^2(\partial^* \hat E)}^{\frac{2}{1-\eta}} + C(n)  \| R\|^2_{L^2(\partial^*E)} \nonumber \\
& \ge  c(n) \left[ {\rm Exc}(\hat G) + |\hat G\Delta B|^2   \right]  - C(n )  |\hat E\Delta B| {\rm Exc}(\hat E) -C(n)\ez |\hat E\Delta B|^2\nonumber \\
& \ge  c(n) \left(  {\rm Exc}(\hat E) +|\hat E\Delta B|^2\right) -C(n)\left(  \|r \hat  R\|_{L^2(\partial^* \hat E)}^{\frac{2}{1-\eta}} + \|R\|_{L^2(\partial^*E)}^{\frac{2}{1-\eta}} \right). \label{almost result}
\end{align}
Now, recalling \eqref{difference 1}, \eqref{difference 2}, $\| R\|_{L^2(\partial^*E)}\le c(n)<1$, and that 
\[
\|r\hat R\|_{L^2(\partial^*\hat E)}^2
=r^{3-n}\|R\|_{L^2(\partial^*E)}^2
\]
with $r\sim1$, we conclude from \eqref{almost result} that
$$ {\rm Exc}( E) +| E\Delta B|^2 \le  C(n)     \| R\|^2_{L^2(\partial^*E)}. $$
This together with \eqref{general mu} proves the normalized estimate, and
therefore, by the initial translation reduction, proves \eqref{eq:main thm}.

\end{document}